\documentclass[12pt]{extarticle}
\usepackage[utf8]{inputenc}
\usepackage[T1]{fontenc}

\usepackage{amssymb,amsfonts}
\usepackage{amsthm}
\usepackage{graphicx}
\usepackage[all,arc]{xy}
\usepackage[colorlinks=true, allcolors=myblue]{hyperref}
\usepackage{enumerate}
\usepackage{bbm}
\usepackage{dsfont}
\usepackage{mathrsfs}
\usepackage{import}
\usepackage[utf8]{inputenc}
\usepackage{tikz}
\usepackage{xspace}
\usepackage{ytableau}
\usepackage{subcaption}
\usepackage{xcolor}
\definecolor{myblue}{rgb}{0, 0.4470, 0.7410}
\definecolor{myorange}{rgb}{0.8500, 0.3250, 0.0980}
\definecolor{mygreen}{rgb}{0.4660, 0.6740, 0.1880}
\definecolor{mypurple}{rgb}{0.4940, 0.1840, 0.5560}
\definecolor{myteal}{rgb}{0.0, 0.5, 0.5}
\definecolor{mycyan}{rgb}{0.0,0.45,0.65}
\definecolor{mybordeaux}{rgb}{0.70,0.25,0.32}

\usepackage[left=1in,top=1in,right=1in,bottom=1in]{geometry}
\usepackage{mathtools}
\usepackage{comment}
\usepackage{makecell} 

\newtheorem{thm}{Theorem}[section]
\newtheorem{cor}[thm]{Corollary}
\newtheorem{prop}[thm]{Proposition}
\newtheorem{lem}[thm]{Lemma}

\theoremstyle{definition}
\newtheorem{defn}[thm]{Definition}

\newtheorem{example}[thm]{Example}

\newtheorem{rmk}[thm]{Remark}

\newtheorem{wfl}[thm]{Workflow}

\def\N{\mathbb{N}}
\def\Z{\mathbb{Z}}
\def\Q{\mathbb{Q}}
\def\R{\mathbb{R}}
\def\C{\mathbb{C}}

\def\PP{\mathbb{P}}

\title{\bf Positive Geometries from Cubic Surfaces}
\date{}

\author{Bernd Sturmfels and Simon Telen}
\begin{document}
\maketitle
	
\begin{abstract} \noindent
We present a study of cubic surfaces from the novel perspective
of positive geometry. Our positive geometries have
dimension two (the surface minus its 27 lines),  dimension three
(its complement in 3-space), and dimension four (the moduli space).
In each case we explore the positive arrangement, its combinatorial rank, and the
canonical~forms.
\end{abstract}
	
\section{Introduction} \label{sec1}

Seeing the $27$ lines on a cubic surface is one of the most beautiful vistas in $19$th century mathematics. The subject has inspired many generations of algebraic geometers.
The surface with its $27$ lines is shown on the right in Figure \ref{fig:27circles}.
The left model represents lines as circles.

\begin{figure}[h]
\centering
 \includegraphics[width = 1.00\linewidth]{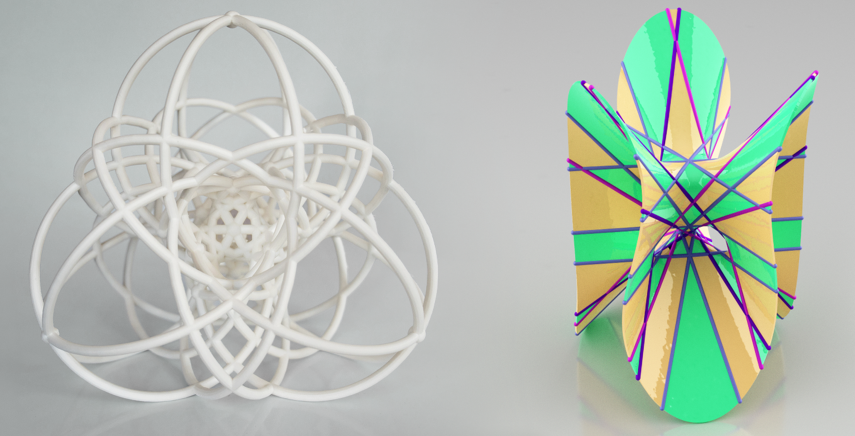} 
\caption{The $27$ lines on a cubic surface in real projective $3$-space.}
\label{fig:27circles}
\end{figure}

\smallskip

Our article is meant to be enjoyable for non-experts.
To this end, Section \ref{sec:glossary} offers a~guide to relevant
areas of mathematics, with pointers to textbooks.
The reader will find a glossary of technical terms that appear  
in the later sections, and even in the two theorems below.

In this paper, we show how the rich classical theory
 of cubic surfaces lends itself naturally to the current
perspective
of positive geometry \cite{Lam, Lam_moduli, RST, Harvard}, a subject at the interface of mathematics and theoretical physics \cite{ABL,ABHY}. The paper is partly expository, building on earlier work on del Pezzo surfaces \cite{EGPSY, HKT, RSS2, SY} while incorporating a range of recent developments in positive geometry \cite{BD, CT, HP, KPS}. Along the way, we establish several new definitions and new results.
It is our aim to be understandable for a wide range of readers.

Our exposition will be guided by the following theorem.
We shall derive a  detailed proof.

\begin{thm} \label{thm:130}
Let $X$ be a general real cubic surface in $\PP^3$
whose $27$ lines are real, and let $Y$ be the union of these lines.
The complement $X_\R \backslash Y_\R$ consists of $130$ curvy polygons,
namely $10$ triangles, $90$ quadrilaterals and $30$ pentagons.
Each of these is a positive geometry for $(X,Y)$. Their $130$ canonical
differential forms span the space $\Omega^2_{\rm log}(X\backslash Y)$
which has dimension $109$.
\end{thm}

Here $\PP^3$ denotes the projective space of dimension $3$, with
homogeneous coordinates $y = (y_0,y_1,y_2,y_3)$.
A cubic surface $X$ is the solution set in $\PP^3$ of a homogeneous polynomial equation
of degree three. The following cubic equation in $y$ will serve as our running example:
\begin{equation}
\label{eq:cubicsurface}
31 y_0 y_1 y_2-23 y_0 y_1 y_3+19 y_0 y_2 y_3-23 y_1 y_2 y_3
-32 y_0^2 y_1+45 y_0 y_1^2
-16 y_2^2 y_3+13 y_2 y_3^2
  \,\,= \,\, 0. \quad
\end{equation}
This surface contains $27$ straight lines $\PP^1$. 
The third column of Table \ref{tab:27} exhibits $27$
pairs of linear equations. These equations define
the arrangement $Y$ of all $27$ lines that lie on $X$.

The varieties in this paper have their points over the 
algebraically closed field of complex
numbers $\C$. Their defining polynomials
have coefficients in the real numbers $\R$.
Using the letter $\R$ as a subscript means that
we restrict a variety to the subset of
points with real coordinates.
For instance, the complex line $\PP^1$ is the Riemann sphere,
whereas $\PP^1_\R$ is a circle.

Every line in $\PP^3_\R$ is a circle.
Each of the $27$ circles on the cubic surface $X_\R$  intersects $10$ other circles on $X_\R$.
A 3D printed model  is shown on the left in
Figure \ref{fig:27circles}. It shows a {\em Clebsch cubic}, which has
$10$ \emph{Eckardt points}. These are points where
three circles intersect. 
The $27$ lines divide the cubic into $120$ curvy quadrilaterals,
seen on the right in Figure \ref{fig:27circles}.
The Clebsch cubic is not general because of its Eckardt points.
Theorem \ref{thm:130} does not apply.

The $10$ Eckardt points seen in Figure \ref{fig:27circles}
correspond to the question marks in Figure \ref{fig:endler2}.
On a general cubic surface $X$, like \eqref{eq:cubicsurface},
no three lines intersect. Here, each of the $10$ question marks
is replaced  by an up-triangle or by a down-triangle. This replaces 
three of the six quadrilaterals around each question mark
by a pentagon. This gives the number $130$ in  Theorem~\ref{thm:130}.
It explains the $10$ triangles, $30$ pentagons and $90$ quadrilaterals in
$X_\R \backslash Y_\R$.

We point out that the $130$ polygons in Theorem \ref{thm:130} were already known to 
Felix Klein in 1873, see \cite[page 570]{Klein}, and they also appear in Beniamino Segre's book from 1942 \cite[\S 65]{Segre}.  We provide a
present-day proof in Section \ref{sec9}, partly based on symbolic~computations.

The totality of all smooth cubic surfaces is parametrized by
the moduli space $\mathcal{Y}(3,6)$.
This is a non-compact complex
variety of dimension $4$. Its real locus $\mathcal{Y}(3,6)_\R$
is a disconnected real manifold of dimension $4$. We shall prove the following
theorem about the moduli space.

\begin{thm} \label{thm:432}
The moduli space $\mathcal{Y}(3,6)_\R$ has $432$ connected components,
each of which is a curvy version of the simple $4$-polytope in Figure \ref{fig:pezzotope}.
Each of these $432$ pezzotopes is a positive geometry for the  Yoshida pair $(X,Y)$, which
is a $W(E_6)$-invariant compactification of $\mathcal{Y}(3,6)$.
Their $432$ canonical
differential forms span the space $\Omega^4_{\rm log}(X\backslash Y)$
which has dimension $150$.
\end{thm}

The formulations in Theorems~\ref{thm:130} and \ref{thm:432} are meant to be similar.
This commonality will guide our discussion.
The most important keyword in both theorems is 
{\em positive geometry}. We explain what this means in
Definition \ref{def:positivegeometry}. That definition was inspired by 
 Brown and Dupont  \cite{BD}, and it
refines that in \cite[Section 1]{Harvard}. The original definition
due to Arkani-Hamed, Bai and Lam \cite{ABL} rests
on a recursion, which reflects the structure of scattering amplitudes in physics.
All other objects in Theorems~\ref{thm:130} and \ref{thm:432} 
will also be defined later.

\begin{figure}
\centering
\includegraphics[width = 0.7\linewidth]{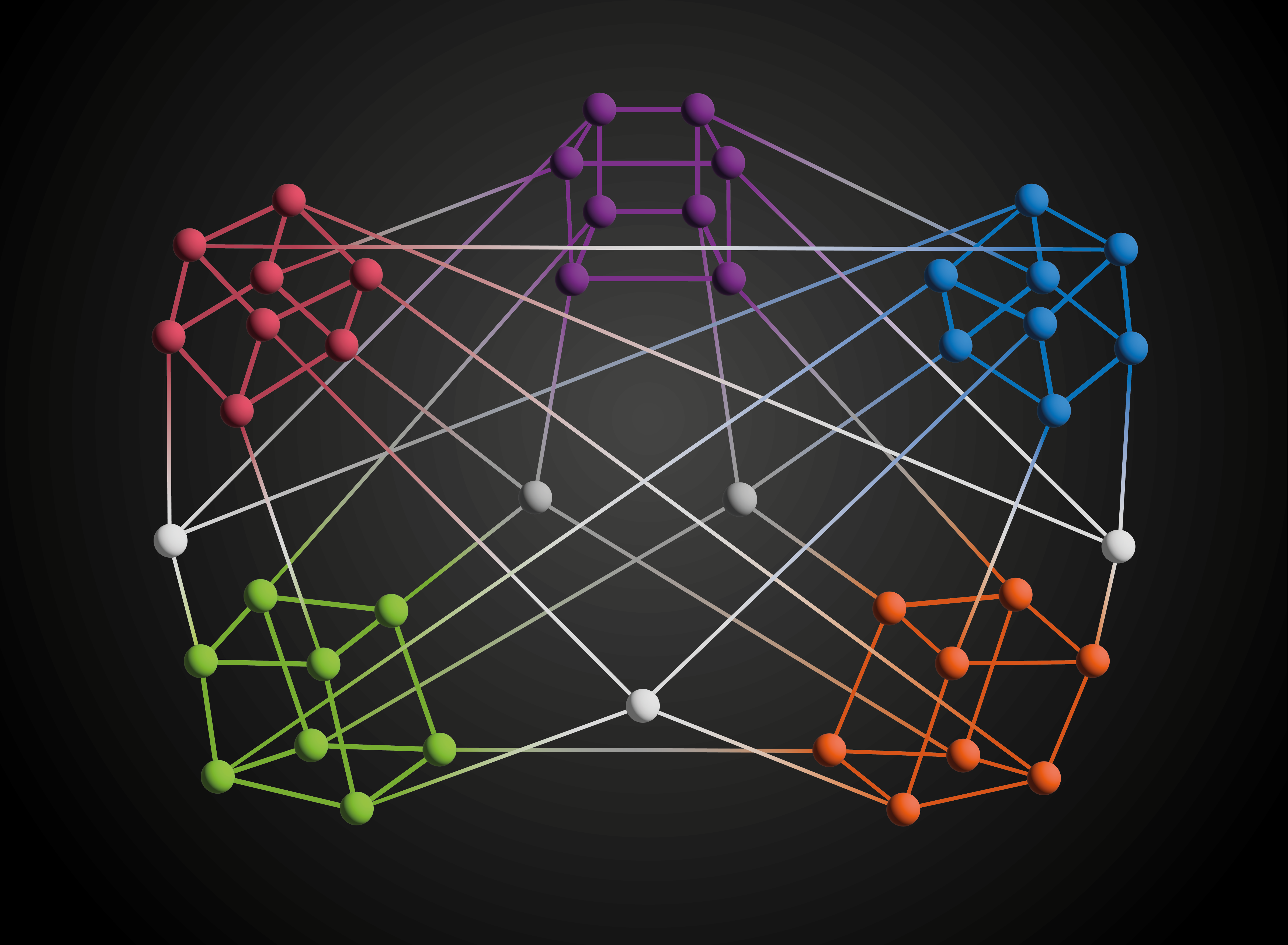}
\caption{Edge graph of the ${\rm E}_6$ pezzotope from \cite[Figure 1]{EGPSY}.
Each of its $15$ facets is a simple $3$-polytope.
Ten facets are associahedra. Five facets are cubes, here shown in color.}
\label{fig:pezzotope}
\end{figure}

The 432 pezzotopes and their symmetry were described by Sekiguchi and Yoshida~\cite{SY}. We study this structure from the perspective of positive geometry, in the sense of Section \ref{sec5}.

The present article is organized as follows.
Section \ref{sec:glossary} offers introductory material,
which should be helpful to  readers
who do not have a background in algebraic geometry.
Section~\ref{sec2} constructs
 the cubic surface by blowing up the plane at six points.
We explain this clearly,
with pictures in Figures  \ref{fig:firstpentagon}
and \ref{fig:adjointquadrilateral}.
The cubic in (\ref{eq:cubicsurface}) is derived from the six points
in (\ref{eq:sixpoints}). 
Its $130$ regions are displayed in
 (\ref{eq:10triangles}), (\ref{eq:pentagons}), (\ref{eq:quadrilaterals}).
Section \ref{sec3} offers a conceptual explanation of these data. This rests on
the Schl\"afli graph $\mathcal{S}^{10}_{27}$
and its symmetry group,  the Weyl group $W({\rm E}_6)$.

Understanding integrals is central to physics. Integrals are pairings between
integration cycles and differential forms.
Section \ref{sec4} introduces the canonical
differential forms that are essential for
positive geometries.  For the $130$ regions on a general cubic surface,
we derive them in
(\ref{eq:triangle1}), (\ref{eq:canformquadrilateral}) and (\ref{eq:canonicalpentagon}).
Lemma \ref{lem:109} says that $109$ of the forms are linearly independent.

Section \ref{sec5} starts from relative homology of a pair $(X,Y)$,
and it introduces the key definitions for our discussion:
positive arrangement, positive geometry, and combinatorial rank.
 Workflow \ref{workflow} offers a blueprint
for future studies of positive geometries arising in practice.

In Section \ref{sec6} we study some geometries in dimension three:
an arrangement derived from Cayley's cubic surface (Proposition \ref{prop:41}),  
$\mathcal{M}_{0,6}$ via the 
Segre cubic  (Proposition \ref{prop:segrecubic}), and a recent
construction due to Koefler, Pavlov and Sinn \cite{KPS}. This is our Proposition~\ref{prop:KPS}.

Section \ref{sec7} is devoted to stringy integrals and their field theory limits.
We explain these physics concepts  in an easy-to-grasp
mathematics language.
This section takes the first steps towards a generalization of the theory of stringy canonical forms developed in \cite{AHL}.

In Section \ref{sec8} we formally introduce the
${\rm E}_6$ pezzotope, which is the
$4$-dimensional polytope in
Figure \ref{fig:pezzotope}.
Theorem \ref{thm:clebsch} explains how a
subdivision of the pezzotope governs the family of
cubic surfaces depicted in Figure \ref{fig:endler2}.
We also introduce the theory of $u$-equations, as in (\ref{eq:perfectuE6}),
with emphasis on the toric geometry perspective 
developed by Calvo Cortes and Telen \cite{CT}.

Section \ref{sec9} covers the 
$4$-dimensional moduli space 
$\mathcal{Y}(3,6)$. This parametrizes configurations of
 six points in general position in the plane $\PP^2$,
 and hence smooth cubic surfaces in $\PP^3$. 
 Our understanding of this space and its symmetries yields the proof of Theorem \ref{thm:130}. 

 Section \ref{sec10} introduces a positive arrangement $(X,Y)$ with
$X \backslash Y = {\cal Y}(3,6)$. Here $X$ is the Yoshida variety in $\PP^{39}$.
The tropicalization of $X$ is the normal crossing boundary  of the
Naruki arrangement $(X',Y')$.
This allows us to study $\Omega^4_{\rm log}(X \backslash Y)$ and to show
 that~the~pezzotopes  in $ {\cal Y}(3,6)$
 are positive geometries.
Theorem \ref{thm:cr150} completes the proof of Theorem~\ref{thm:432}.
  We conclude in Section \ref{sec11} with a study of the canonical forms on
${\cal Y}(3,6)$. 
They span the $150$-dimensional space $\Omega^4_{\rm log}(X \backslash Y)$.
Explicit generators are given in Corollaries \ref{cor:elfzwei} and~\ref{cor:elfdrei}.

The Julia code accompanying this paper can be found online at the Zenodo page \cite{Zenodo}.

\section{Guide and Glossary} \label{sec:glossary}

This article is written for  a general audience of mathematicians.
Some physicists will find it useful as well.
Whether the material we present is understandable will
depend on the reader's background. 
In this section we go over  prerequisites,
give definitions of many technical~terms, and offer pointers
to textbooks. Readers who are familiar with the language of
algebraic geometry may skip this section. Everyone else
is encouraged to take a peek, even if briefly.

This section is organized into ten short subsections.
Each offers a synopsis of what would be a one-semester course, 
either at an advanced undergraduate level or at the graduate level.
Our guide is accompanied by a list of $20$ text books
that are suitable for in-depth study.
These books are listed as the first $20$ items in the bibliography
at the end of this paper.

\subsection*{Basic Algebra}
A {\em field} $\mathbb{K}$ is a set of numbers, such as rational numbers
$\Q$, the real numbers $\R$, or the complex numbers $\C$.
The field $\mathbb{K}$ is {\em algebraically closed} if every
nonconstant polynomial in $\mathbb{K}[x]$ has a zero in $\mathbb{K}$.
Hence $\C$ is algebraically closed, by the Fundamental Theorem of Algebra,
but $\Q$ and $\R$ are not. We perform computations in the
{\em ring} $R = \mathbb{K}[x_1,\ldots,x_n]$ of polynomials in $n$ variables.
Sometimes we replace $R$ with the ring 
$\mathbb{K}[x_1^{\pm 1}, \ldots, x_n^{\pm 1}]$ of {\em Laurent polynomials},
where the exponents can be negative integers.
The {\em (Laurent) monomials} $x^a = x_1^{a_1} \cdots x_n^{a_n}$
form a basis for $R$ as a vector space over $\mathbb{K}$.
If $a_i \in \{0,1\}$ then $x^a$ is a {\em squarefree monomial}.

An {\em ideal} $I$ in $R$ is a $\mathbb{K}$-subspace that is closed under multiplication
by elements from $R$. 
An~ideal  $I$ is {\em principal} if it is generated by one element.
It is {\em prime} if $fg \in I$ implies $f \in I$ or $g \in I $.
Operations on ideals are discussed in \cite{CLO, MS}.
For instance, the {\em saturation} appears in
\cite[Section 4.4]{CLO}. Subvarieties of {\em projective space}
$\PP^{m}$ are described by {\em homogeneous ideals} 
\cite[Section 8.3]{CLO}. For varieties in $\PP^{m} \times \PP^n$, one
uses {\em bihomogeneous ideals} in a polynomial ring with a
$\Z^2$-grading.
A polynomial with $2,3$, or $8$ terms is called
{\em binomial}, {\em trinomial} or {\em octanomial}.
Ratios of two polynomials are {\em rational functions}.
These form the field $\mathbb{K}(x_1,\ldots,x_n)$.

\subsection*{Algebraic Geometry I}
For a first introduction to algebraic geometry we recommend the textbooks \cite{CLO, CR, Reid}.
Every variety has a unique decomposition into irreducible
components, $V = V_1 \cup \cdots \cup V_k$, read off from a
 primary decomposition of its ideal \cite[Chapter 3]{INLA}.
If $k = 1$ then $V$ is  {\em irreducible}.
The {\em dimension} of $V$ can be defined via 
 Gr\"obner bases \cite[Chapter 9]{CLO}. This agrees with the usual
manifold definition when $\mathbb{K} = \R,\C$.
Varieties of dimension $1,2,3,4$ are 
{\em curves}, {\em surfaces}, {\em threefolds}, {\em fourfolds}.
A variety of dimension $n-1$ in $\mathbb{K}^n$ is a {\em hypersurface}.
Points of $V$ where
the Jacobian matrix of ideal generators drops rank are {\em singular points}.
All singular points form the {\em singular locus} of $V$.
We call $V$ {\em smooth} if it has no singular points.
The {\em Zariski closure} of a set $W$ in $\mathbb{K}^n$ is the
smallest variety containing $W$. One often replaces
$\mathbb{K}^n$ with {\em projective space} $\PP^n$,
for which {\em homogeneous coordinates} $(x_0:\cdots:x_n)$ are used
\cite[Chapter~8]{CLO}. 

The projective line $\PP^1$
over $\R$ is a circle and over $\C$ is the {\em Riemann sphere}.
The most basic invariant of a curve $C$ is its genus $g$.
We call $C$ {\em rational} if $g=0$ and {\em elliptic} if $g=1$.
Plane curves of degree $2$ are {\em conics}, and these are rational.
We write $\PP^m \dashrightarrow \PP^n$ for {\em rational maps}, given
by $n+1$ homogeneous polynomials of the same degree in $m+1$ variables.
These are undefined at points where all polynomials vanish.
The map is {\em birational} if it has a rational inverse.
An example is the {\em Cremona transformation}
$\PP^2 \dashrightarrow \PP^2, 
(x:y:z) \mapsto  (yz:xz:yz)$.
A surface $S$ in $\PP^3$ is the zero set of a homogeneous
polynomial in $4$ variables. A plane $P$ in $\PP^3$ is a {\em tritangent plane} of 
the surface $S$
if the curve $P \cap S$ has precisely $3$ singular points.

The {\em Grassmannian} parametrizes
$d$-dimensional subspaces of $\mathbb{K}^n$.
  This is a projective variety of dimension $d(n-d)$, embedded in
  $\PP^{\binom{n}{d}-1}$ by the $d \times d$ minors
  of a $d \times n$ matrix.
  These are known as {\em Pl\"ucker coordinates}. They satisfy quadratic relations
called  {\em Pl\"ucker relations}. An example
     is $p_{12} p_{34} - p_{13} p_{24} + p_{14} p_{23} = 0$
  for $d=2$. 
    We refer to \cite[Chapter 5]{INLA}.

\subsection*{Combinatorics}
Our graphs are undirected, simple and have no loops.
A cycle in a graph $G$ is {\em chordless} if it is an induced subgraph.
We write $K_n$ for the {\em complete graph} 
and $K_{m,n}$ for the {\em complete bipartite graph}.
The {\em clique complex} of a graph $G$ is a
{\em simplicial complex} on the same vertex set: its simplices
are the {\em cliques}, i.e.~collections of vertices that are pairwise connected in $G$.

We recommend Thomas' book \cite{Thomas} for a first introduction to
geometric combinatorics. A {\em polytope} $P$ is the {\em convex hull}
of a finite subset of $\R^n$. A {\em face} of $P$ is a subset on which a
linear function attains its minimum. Faces of dimension $1$ and $2$ are 
{\em vertices} and {\em edges}.
Assuming ${\rm dim}(P) = n$, 
the faces of dimension $n-1$ are called {\em facets} of $P$.
A polytope $P$ is {\em simple} if every vertex is incident to $n$ edges.
The $f$-vector $(f_0,f_1,\ldots,f_{n-1})$ records the number
$f_i$ of $i$-dimensional faces of $P$.
The $3$-cube is a simple polytope with f-vector $(8,12,6)$.

The {\em Newton polytope} of a polynomial
$p$ in $n$ variables is the convex hull of all points
$a \in \N^n$ such that $x^a$ has nonzero coefficient in $p$.
The {\em associahedron} is the Newton polytope~of
$$ p\,\, =\,\,  (1+x)(x+y)(y+z)(1+x+y)(x+y+z)(1+x+y+z). $$
This is a simple $3$-polytope with f-vector $(14,21,9)$. It generalizes to $n$ dimensions.
Both {\em matroids} \cite[Chapter 14]{INLA} and {\em oriented matroids} are useful
tools in the study of polytopes.

Symmetry groups in geometric combinatorics are often
generated by reflections at a collection of hyperplanes.
Humphrey's book \cite{Humphreys} gives an introduction to
finite reflection groups and their associated {\em reflection arrangements}.
These are classified by {\em root systems}, and the groups are known as
{\em Weyl groups}. 
We are especially interested in the root system $E_6$ which consists
of the normal vectors to  $36$ reflecting hyperplanes in $\R^6$.
The associated {\em  Weyl group $W(E_6)$} has order $51840$, and it
is the main hero in our story  of cubic surfaces.

\subsection*{Complex Analysis}
Every smooth variety over $\C$ is a complex manifold, and hence
an object of complex analysis.
The book by Griffiths and Harris \cite{GH} is an excellent resource
at the interface of algebraic geometry and complex analysis.
{\em Differential forms} are introduced in \cite[Chapter 0]{GH}, and
an in-depth coverage of {\em residues} can be found in \cite[Chapter 5]{GH}.
Mathematics students wishing to learn positive geometry should look at
Aizenberg and Yuzhakov's book  \cite{AY}. This
is a foundational text on multidimensional complex analysis, focusing on integral representations, multidimensional residues, and their applications.
Residues are covered in \cite[Chapter III]{AY}.

Every rational function $f$ in $n$ variables gives rise to a
meromorphic differential~form
$${\rm d} \,{\rm log}(f) \, = \, \frac{{\rm d} f}{f} \,= \,
\frac{1}{f} \sum_{i=1}^n \frac{\partial f}{\partial x_i} \, {\rm d} x_i. $$
This is a $1$-form. By taking the wedge product of $n$ of them, one
obtains a {\em top form} on $\C^n$:
$$ {\rm d} \, {\rm log}(f_1) \,\wedge \,
{\rm d} \,{\rm log}(f_2) \,\wedge \,\cdots \, \wedge \, {\rm d} \, {\rm log}(f_n) . $$
These are called {\em logarithmic forms}.
 They provide the mildest singularities (pole structure) along a boundary or divisor to study open manifolds. Two open manifolds seen in our paper are
 cubic surfaces with their $27$ lines removed, and the $4$-dimensional moduli space
 $\mathcal{Y}(3,6)$.
 
\subsection*{Real Algebraic Geometry}

The field $\R$ of real numbers is not algebraically closed.
But it is an ordered field, and it is {\em real closed}. A subset of $\R^n$ is
{\em semialgebraic} if it is a finite Boolean combination  of sets 
\begin{equation}
\label{eq:semialgebraic}
 \{ \,x \in \R^n \,: \, p_i(x) > 0 \,\,\,
{\rm for} \,\,\, i = 1,2,\ldots,s \bigr\} , 
\end{equation}
where $p_1,\ldots,p_s$ are polynomials. Our paper talks about {\em regions},
{\em chambers}  and {\em domains} as examples of such sets. 
Workflow \ref{workflow} is our most important agenda item for these concepts.

A {\em real variety} is a complex variety $V$ that is defined by polynomials in $\R[x_1,\ldots,x_n]$.
Its set of real points might be empty. The conic 
$\{x^2 +y^2+ 1 = 0\}$ has no real points. We consider irreducible $V$ that have
 at least one {\em smooth real point}. Under that hypothesis,
the set $V_\R$ of real points is {\em Zariski dense} in $V$.
This holds for every {\em real moduli space} in this paper.
The real points of
$\mathcal{M}_{0,n}$ and $\mathcal{Y}(3,6)$,
indicated by a subscript $\R$, are Zariski dense.

When compared to complex algebraic geometry, the number of
scholars in real algebraic geometry is small.
 But the field has strong connections to applications,
such as optimization, statistics, chemistry, biology, and -- of course -- physics.
After all, a {\em positive geometry} is first and foremost a semialgebraic set. Our readers are encouraged to
 study real algebraic  algebraic geometry from an algorithmic perspective.
For that we recommend the textbook by
 Basu, Pollack and Roy \cite{BPR}.
 Theobald's book \cite{Theobald} 
 emphasizes connections with optimization.
 His presentation focuses on  {\em spectrahedra} and
certificates of positivity, like {\em sum-of-squares} representations.
These topics are
of increasing interest also in the physics community.

\subsection*{Algebraic Topology}

Complex manifolds and semialgebraic regions are topological spaces.
Algebraic topology offers a linear algebra approach to such objects
by assigning finite-dimensional vector spaces. Known as {\em homology}
and {\em cohomology}, these vector spaces are featured in every graduate course
on algebraic topology. For a first study, we recommend Hatcher's widely used textbook
\cite{Hatcher}.

Every complex manifold of dimension $d$ is a real manifold of dimension $2d$.
This is~how one studies its topology. 
Every irreducible complex variety $V$ is {\em connected}.
This is proved in \cite[Section VII.2]{Shafarevich}. 
The same holds for the open varieties $\C^d \backslash V$. However,
if $V$ is a real variety then  $\,V_\R$, its complement $\R^d \backslash V_\R$,
and the semialgebraic set (\ref{eq:semialgebraic}) can be  disconnected.
Each of them is written uniquely as a disjoint union of {\em connected components}.
Real and complex projective spaces are {\em compact}.
Each smooth projective variety $X$ is a compact~manifold.

The $i$th {\em homology group} $H_i(X,\Q)$ is the space
of {\em $i$-cycles} in $X$ modulo boundaries of {\em $(i+1)$-chains} in $X$.
Its vector space dual is the {\em  cohomology group} $H^i(X,\Q)$. 
Elements of these spaces are {\em classes}.
In many cases of interest, the 
direct sum $H^*(X,\Q) = \oplus_{i} H^i(X,\Q)$ is a $\Q$-algebra,
called the {\em cohomology ring}. 
If ${\rm dim}(X) = d$ then $H_d(X,\Q)$ (or $H^d(X,\Q)$) is the {\em middle (co)homology}.
Every continuous map $f : X \rightarrow Y$ induces a
 linear map $f_* : H^i(X,\Q) \rightarrow H^i(Y,\Q)$
 which is known as the {\em pushforward}.
 The {\em Euler characteristic}~is
 $$ \chi(X) \,\,  = \,\, \sum_{i=0}^{2d} (-1)^i \cdot {\rm dim}_\Q \bigl( H_i(X,\Q) \bigr) \,\,
 = \,\, \sum_{i=0}^{2d} (-1)^i \cdot {\rm dim}_\Q \bigl( H^i(X,\Q) \bigr). $$
 
 If $Y$ is a subspace of $X$ then one considers the {\em relative
 homology} $H_i(X,Y; \Q)$. Its elements are {\em relative cycles},~i.e.~chains in $X$ 
 whose boundaries are chains in $Y$, modulo relative boundaries,
i.e.~chains that are homologous to chains in $Y$. This appears before Definition~\ref{def:positivearrangement}.

Now consider any compact $d$-dimensional connected real manifold ${\cal M}$ with boundary $\partial {\cal M}$. In Workflow \ref{workflow}, ${\cal M}$ could be the closure in $\mathbb{R}^d$ of a connected component of $\mathbb{R}^d \backslash V_\R$. 
We call ${\cal M}$ {\em orientable} if
$H_d({\cal M}, \partial {\cal M}; \Q) \simeq \Q$. The interior ${\rm int}({\cal M})$ is orientable if and only if ${\cal M}$ is orientable. 
The M\"obius strip and the Klein bottle are surfaces that fail to be orientable.

Each curve in complex projective space $\PP^n$ is an orientable
real surface. These are known as {\em Riemann surfaces}.
Relative homology is a foundational ingredient in 
{\em mixed Hodge theory}.  The approach to positive geometry proposed
by Brown and Dupont \cite{BD}  rests on~this.

\subsection*{Toric Geometry}
{\em The world is toric}. This slogan, put forth in \cite[Section 8.3]{INLA}, reflects
the fact that many algebraic varieties that arise in applications have a particularly
nice combinatorial structure. 
These are the {\em toric varieties}. For a systematic
textbook introduction to toric geometry we refer to
the book by Cox, Little and Schenck \cite{CLS}.
A quick guide is found in \cite[Chapter 8]{INLA}.

An irreducible variety is {\em toric} if it has a parametrization by monomials. 
This is equivalent to saying that its  prime ideal is generated by
binomials. For instance, the set of points $(s^2,st, t^3)$ is a 
surface in $\C^3$, namely the toric surface defined by the binomial $x^3 z^2 - y^6$.
The exponent vectors of the parametrizing monomials span a polyhedral cone.
The toric variety is divided into strata that are indexed by the
faces of the cone \cite[Theorem 8.14]{INLA}. For a projective toric variety $X$,
the role of the cone is played by a convex polytope \cite[Corollary 8.21]{INLA}.
The strata corresponding to its proper faces form the {\em toric boundary} of $X$.
The {\em moment map} identifies the polytope 
with positive points on $X$; see \cite[Theorem 8.24]{INLA}.

For example, the following {\em smooth toric~surface}
is a special case of \cite[Example 8.26]{INLA}:
  $$ X \,=\,\{ \,(s_0 t_0: s_0 t_1 : s_1 t_0: s_1 t_1) \,| \, 
s_0,s_1,t_0,t_1 \in \C\, \} \,\, = \,\, V(x_0x_3-x_1x_2)\,\, \subset \,\, \PP^3. $$
Here the polytope is a square. The four edges of the square
correspond to four curves $\PP^1$ on $X$.
Their union is the toric boundary of $X$,
where two of the  coordinates $x_0,x_1,x_2,x_3$ are zero.
As an abstract variety, $X$ is isomorphic to $\PP^1 \times \PP^1$,
with intrinsic  coordinates $\bigl((y_0:y_1),(z_0:z_1)\bigr)$.
Such  {\em Cox coordinates}
are available for every toric variety \cite[Chapter~5]{CLS}.

\subsection*{Tropical Geometry}
Tropical geometry is a piecewise-linear version of algebraic geometry.
Its objects, the {\em tropical varieties}, are defined by 
polynomials over the {\em min-plus algebra}. The textbook
by Maclagan and Sturmfels \cite{MS} gives an introduction.
For a first glimpse we recommend \cite[Chapter 7]{INLA}.
There is a strong connection between toric geometry and tropical geometry \cite[Chapter 6]{MS}.

Let $V$ be an algebraic variety in $n$-space, defined over a field $\mathbb{K}$
with a valuation, like the $p$-adic numbers or the Puiseux series.
The process of {\em tropicalization} replaces $V$ by the 
{\em tropical variety} ${\rm trop}(V)$. The latter is a piecewise-linear object in
$\R^n$, of the same dimension as $V$. This makes sense also if the 
polynomials defining $V$ have real or complex coefficients, so their
valuation is $0$. In that case, the tropicalization ${\rm trop}(V)$ is a polyhedral fan in
$\R^n$. See \cite[Figure 7.1]{INLA} for a drawing that shows the tropicalization
of the $3 \times 3$ determinant.

If $V$ is a $d$-dimensional linear subspace of $\C^n$ then
${\rm trop}(V)$ is the {\em Bergman fan}. This fan is determined
purely combinatorially from the rank $d$ matroid of $V$; see \cite[Section 4.2]{MS}.
The parameter space for linear subspaces of $\C^n$ is the Grassmannian.
Its tropicalization is the {\em tropical Grassmannian}. This is a fan
of dimension $d(n-d)$ in a tropical projective space of
dimension $\binom{n}{d}-1$, whose points correspond to
tropical linear spaces \cite[Theorem 4.3.17]{MS}.

The fans we described give rise to {\em tropical compactifications}
\cite[Section 6.4]{MS}. For instance, if $V$ is a linear space
then this represents an arrangement of $n$ hyperplanes.
The Bergman fan furnishes a natural compactification of the arrangement complement.
Also for the open varieties that appear in this paper, the tropical compactifications 
tend to be normal crossing. We can thus use them to identify a basis
of canonical forms for our positive geometries.

\subsection*{Algebraic Geometry II} 
Hartshorne's text \cite{Hartshorne} is the classic for a graduate
course in algebraic geometry, but there are many other excellent sources.
Vakil's {\em Rising Sea} \cite{Vakil} is very popular.
For structures appearing in positive geometry such as differential forms and (mixed) Hodge structures, we recommend
the books by
Griffiths-Harris \cite{GH}, Peters-Steenbrink \cite{PetersSteenbrink} and Voisin \cite{Voisin}.

A {\em Weyl divisor} on a variety $V$ is a formal $\Z$-linear combination 
$D = \sum_{i=1}^s n_i D_i$ of {\em prime divisors} $D_i$.
These are irreducible subvarieties of codimension one in $V$.
The {\em linear system} $|D|$ consists of all effective divisors that are
linearly equivalent to $D$. For instance, if the $D_i$ are points on a curve $V$
then $|D|$ is a space of rational functions on $V$ 
whose poles and zeros are prescribed by $D$.
The divisor class $[D]$ is the corresponding element in the {\em divisor class group}.
This is the group of   Weyl divisors in $V$ modulo linear equivalence
\cite[Section 2.5]{Hartshorne}.

The {\em blow-up} of a variety $X$ along a finite subset $Z$ is a birational map
$ \phi: X \dashrightarrow Y $ which is undefined on $Z$
and is an isomorphism on $X \backslash Z$.
 The birational inverse is given by
a morphism $\phi^{-1}: Y \rightarrow X$. See \cite[Chapter 22]{Vakil}.
For $z \in Z$, the fiber over $z$ of the {\em blow-down} $\phi^{-1}$
is a divisor $E_z$ in $Y$, called the {\em exceptional divisor} of $z$.
In our paper, we have $X = \PP^2$, $Z $ consists of $6$ points, 
$Y$ is the cubic surface, and $E_z$ is one of the $27$ lines~\cite[Chapter 27]{Vakil}.
If $X$ is a singular variety then one seeks a smooth variety $Y$
with maps $\psi$ and $\psi^{-1}$ as above. Such a {\em desingularization}
is known to exist over $\C$, by a celebrated theorem of Hironaka.

If ${\cal U}$ is a variety which is open and dense in a projective variety $X$
then we call $X$  a {\em compactification} of ${\cal U}$. One usually assumes
that  $Y = X \backslash {\cal U}$ is a finite union
of prime divisors $D_1,\ldots,D_s$, called {\em boundary components}.
Hironaka's Theorem implies that we can replace any such {\em pair} $(X,Y)$
by a new pair $(X',Y')$ where $X\backslash Y = X' \backslash Y'$
and $Y'$ is smooth with boundary components
$D'_1,\ldots,D'_t$ that satisfy {\em normal crossing}.
Computing this in practice is difficult.
See \cite[Remark 6.4.9]{MS} for an informative discussion in the toric-tropical context.

A {\em moduli space} is a variety that parametrizes 
classes of other varieties 
modulo a natural notion of equivalence. For example,
the moduli space $\mathcal{M}_{0,n}$ parametrizes $n$-tuples
of distinct points in $\PP^1$ modulo projective transformations; see \cite[Theorem 6.4.12]{MS}.
Our moduli space $\mathcal{Y}(3,6)$ extends this to points in $\PP^2$.
One seeks compactifications of these moduli spaces.
Sometimes these arise naturally from a 
{\em log canonical embedding}, where a basis
of canonical forms serves as coordinates on a big projective space.
This is similar to the canonical embedding of a curve by
its $g$ holomorphic $1$-forms, which is explained in \cite[Section IV.5]{Hartshorne}.
These form the linear system of the {\em canonical divisor}.
This enters into the {\em adjunction formula} \cite[Proposition V.1.5]{Hartshorne} 
which can be used to compute the genus of a curve on a surface.

A {\em period} is an integral of a closed algebraic differential form along a homological cycle on an algebraic variety.
These yield a countable subset of the complex numbers 
that can be represented on a computer. 
Periods encode special functions of great interest in~physics. They are central objects in (mixed) Hodge theory, for which we refer to \cite{PetersSteenbrink,Voisin}.

\subsection*{Theoretical Physics}

Many  fields of research connect mathematics and physics.
Positive geometry is a very tiny dot in this vast landscape.
It arose from particle physics, namely from the study of
{\em scattering amplitudes}.  This led to the {\em amplituhedron}
which is the image of the positive points in a real Grassmannian
under a linear projection from its Pl\"ucker embedding.
A simpler object, the {\em associahedron} is a positive
geometry that represents the
{\em biadjoint scalar amplitude} in 
 {\em $\phi^3$ theory}. This arises from the
  {\em Parke-Taylor formula} which is the reciprocal
of products of adjacent Pl\"ucker coordinates.
An introduction to these concepts can be found in
the book by Elvang and Huang \cite{EH}.
However, this is not easy to read for
mathematicians with no background in
quantum field theory.  Another source of inspiration
is {\em string theory}. This leads to concepts like
{\em stringy integral}, {\em regulated integral} and
{\em field theory limit}, for which our paper refers
to recent research articles. A  textbook
on Positive Geometry is yet to be written.

\section{Six Points, 27 Lines, and 130 Regions}
\label{sec2}

We have now reached the point where 
our story of the cubic surface begins.
This section starts gently, by
 illustrating Theorem \ref{thm:130} with an explicit example. We build a cubic surface from a
configuration of six points in the projective plane $\mathbb{P}^2_\mathbb{R}$. Concretely, we fix the points
\begin{equation}
\label{eq:sixpoints}
\begin{matrix}
 E_1 &=& (1:0:0), & \qquad
 E_2 &=& (0:1:0), & \qquad
 E_3 &=& (0:0:1), \\
 E_4 &=& (1:-1:1),  & \qquad
 E_5 &=& (6:-4:3),  & \qquad
 E_6 &=& (24:-11:8).
 \end{matrix}
 \end{equation}
 These points are the columns of the $3 \times 6$ matrix 
 in (\ref{eq:sixpoints7})
 with $\, a\,=\,b\,=\,c\,=\,1 \,$ and $\,d=2$.

 Let $F_{ij}$ be the line spanned by $E_i$ and $E_j$,
 and let $G_k$ be the conic passing through the five points
$\{E_1,E_2,E_3,E_4,E_5,E_6\} \backslash \{E_k\}$.
In total, we have introduced $27$ geometric objects,
namely $6$ points, $15$ lines, and $6$ conics. These divide the plane $\PP^2_\R$ into $130$ regions.
Some regions have linear boundaries, so they are polygons. Figure \ref{fig:firstpentagon}
depicts such a pentagon. Many regions  have quadratic pieces
$G_i$ and special points $E_i$ in their boundaries.
Figure \ref{fig:adjointquadrilateral} shows the region 
$E_5F_{35}G_5F_{45}$. This is a curvy triangle in $\PP^2_\R$.
On the cubic surface to be created in $\PP^3_\R$, the point $E_5$ is replaced by a line segment,
so we get a curvy quadrilateral.

\begin{figure}[h]
\centering
\includegraphics[width = 0.72\linewidth]{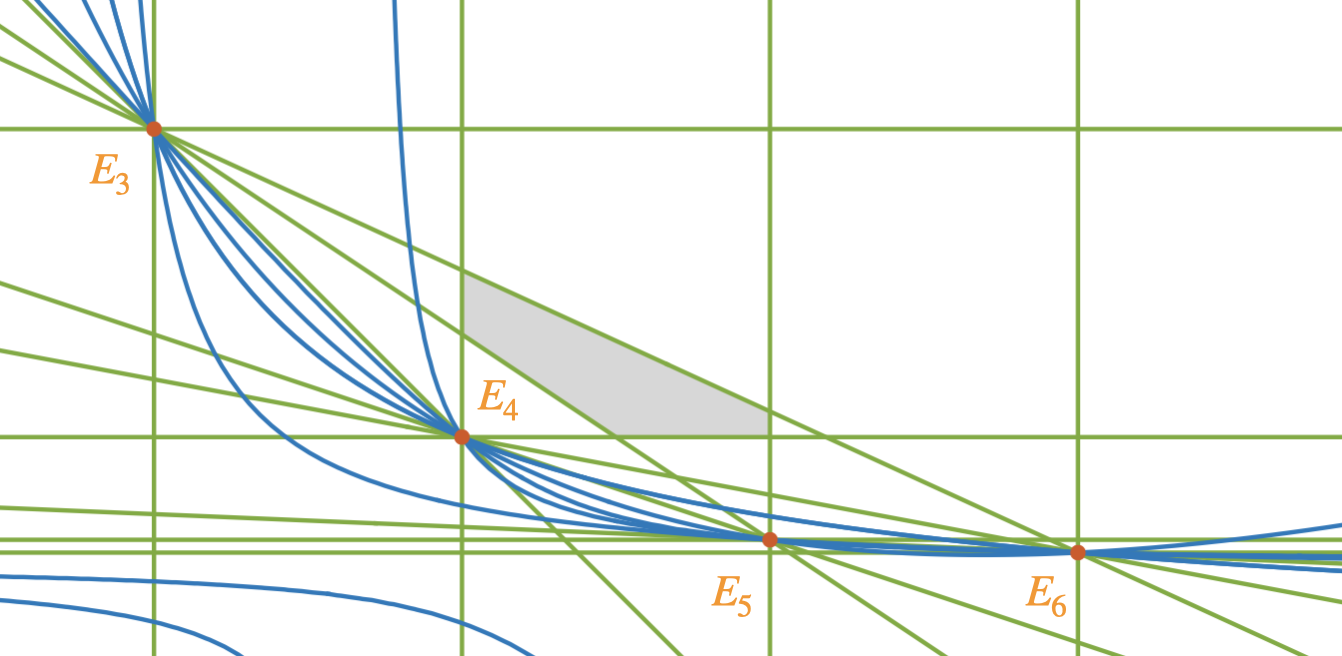}
\caption{The pentagon $F_{14}F_{35}F_{24}F_{36}F_{25}$ in the
real projective plane $\PP_\R^2$.}
\label{fig:firstpentagon}
\end{figure}

\begin{figure}
\centering
\includegraphics[width =0.55\linewidth]{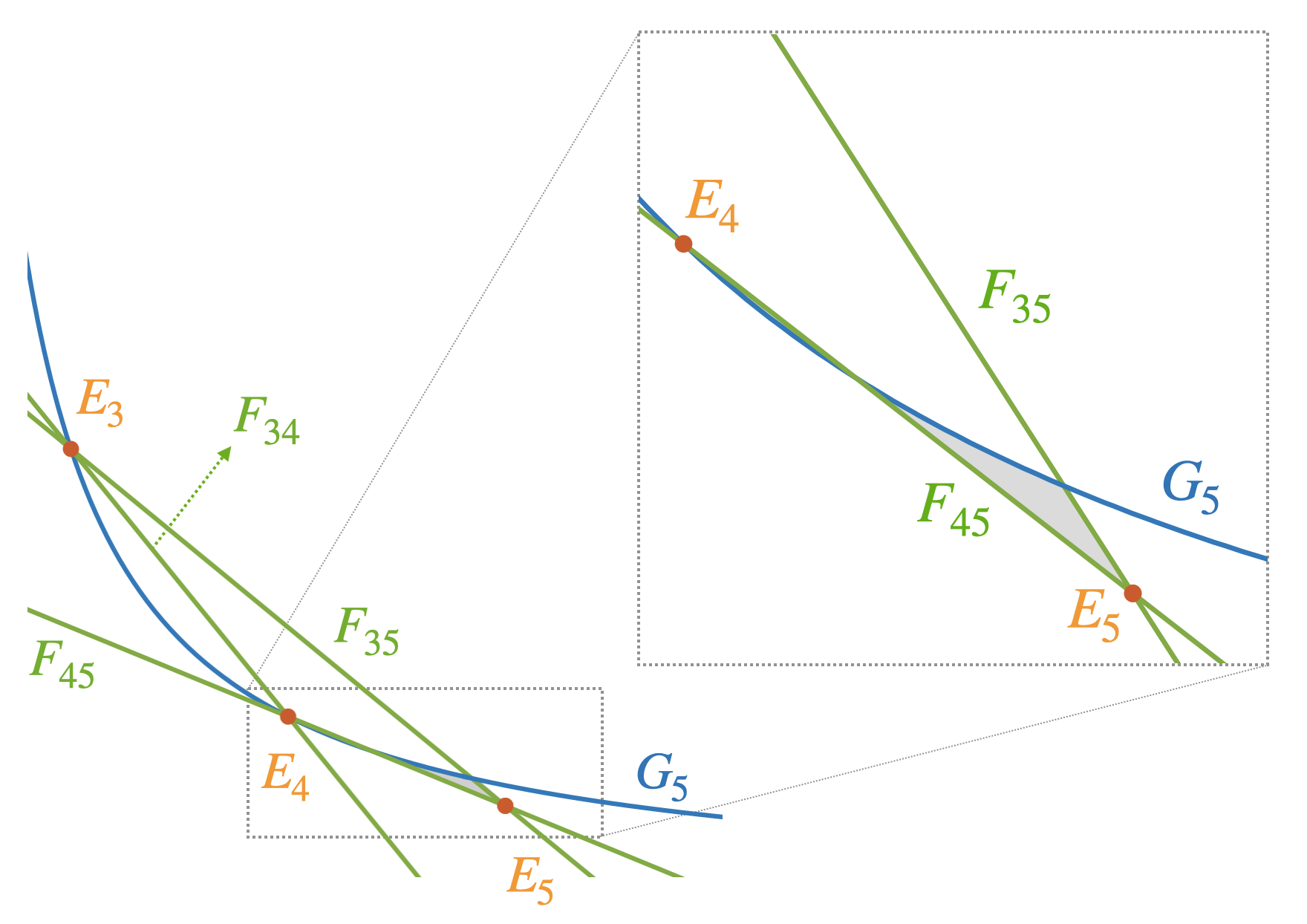} 
\includegraphics[width = 0.44\linewidth]{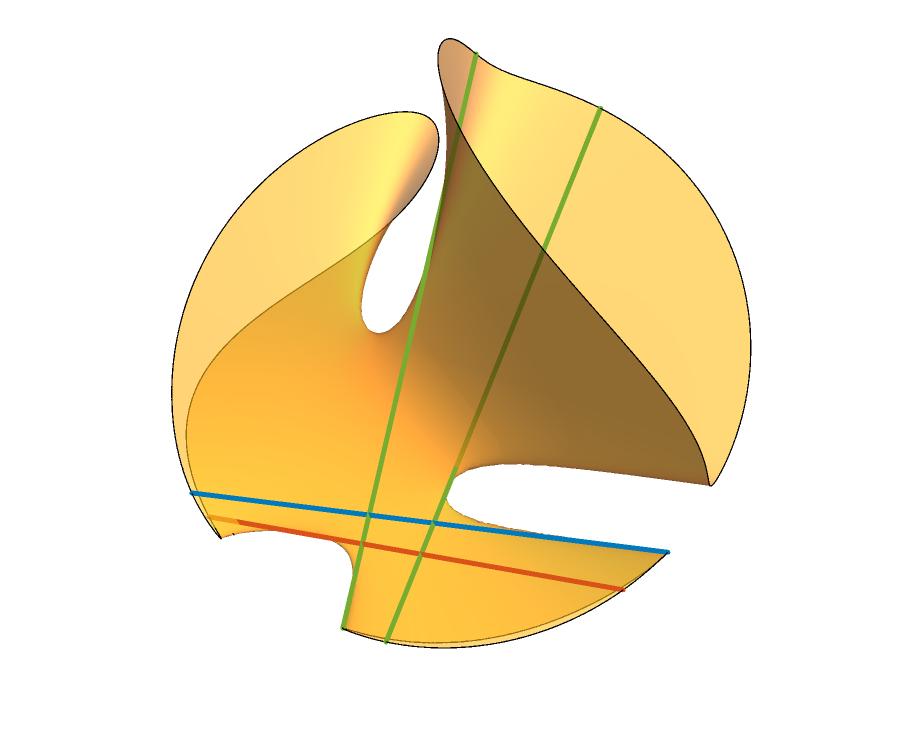}
\caption{The quadrilateral $E_5F_{35}G_5F_{45}$ on the cubic surface
comes from a triangle in $\PP^2_\R$.
 }
\label{fig:adjointquadrilateral}
\end{figure}

\begin{table}
$$ \begin{footnotesize}
\begin{matrix}
\hbox{Label} & & \hbox{Point, line or conic in $\PP^2$} & \hbox{Line in $\PP^3$} \smallskip \\
E_1 &  & \langle x_1,x_2 \rangle &  \langle y_1 - 3 y_3,6 y_0 - y_2 \rangle \\
E_2 &  & \langle x_0,x_2 \rangle &  \langle 3 y_1 + y_2 - 2 y_3,2 y_0 - y_2 \rangle \\
E_3 &  & \langle x_0,x_1 \rangle &  \langle 15 y_1 - 13 y_3,26 y_0 - 30 y_2 - 13 y_3 \rangle \\
E_4  &  & \langle x_0 + x_1,x_1 + x_2 \rangle &  \langle y_1 + y_2,16 y_0 + 7 y_2 + 13 y_3 \rangle \\
E_5 &  & \langle 2 x_0 + 3 x_1,3 x_1 + 4 x_2 \rangle &  \langle 45 y_1 + 32 y_2 - 7 y_3,2 y_0 + y_3 \rangle \\
E_6 &  & \langle 11 x_0 + 24 x_1,8 x_1 + 11 x_2 \rangle &  \langle 15 y_1 + 11 y_2,22 y_0 + 13 y_3 \rangle \smallskip \\
F_{12}  & &  \langle x_2 \rangle &  \langle y_2,y_0 \rangle \\
F_{13} &  & \langle x_1 \rangle &  \langle y_3,y_1 \rangle \\
F_{14}  &  & \langle x_1 + x_2 \rangle &  \langle 15 y_1 + 8 y_2 + 3 y_3,y_0 + y_3 \rangle \\
F_{15}  &  & \langle 3 x_1 + 4 x_2 \rangle &  \langle 3 y_1 + 2 y_2,64 y_0 - 2 y_2 + 39 y_3 \rangle \\
F_{16}  &  & \langle 8 x_1 + 11 x_2 \rangle &  \langle 17 y_1 + 11 y_2 - 3 y_3,17 y_0 - y_2 + 8 y_3 \rangle \\
F_{23}  &  & \langle x_0 \rangle & \langle 3 y_1 + 2 y_2,y_0 + y_3 \rangle \\
F_{24}  &  & \langle x_0 - x_2 \rangle &  \langle y_3,32 y_0 - 45 y_1 - 31 y_2 \rangle \\
F_{25}  &  & \langle x_0 - 2 x_2 \rangle &  \langle y_1,19 y_0 - 16 y_2 + 13 y_3 \rangle \\
F_{26}  &  & \langle x_0 - 3 x_2 \rangle &  \langle 3 y_1 - y_3,4 y_0 - 3 y_2 + y_3 \rangle \\
F_{34}  &  & \langle x_0 + x_1 \rangle &  \langle 23 y_1 + 16 y_2 - 13 y_3,y_0 \rangle \\
F_{35}  &  & \langle 2 x_0 + 3 x_1 \rangle &  \langle y_2,32 y_0 - 45 y_1 + 23 y_3 \rangle \\
F_{36}  &  & \langle 11 x_0 + 24 x_1 \rangle &  \langle 30 y_1 - 22 y_2 - 13 y_3,y_0 - 2 y_2 \rangle \\
F_{45}  &  & \langle x_0 + 3 x_1 + 2 x_2 \rangle &  \langle 3 y_1 - y_3,y_0 - 2 y_2 \rangle \\
F_{46}  &  & \langle 3 x_0 + 16 x_1 + 13 x_2 \rangle &  \langle y_2,y_1 \rangle \\
F_{56}  &  & \langle x_0 + 24 x_1 + 30 x_2 \rangle &  \langle y_3,y_0 \rangle \smallskip \\
G_1 &  & \langle 4 x_0^2 - 9 x_0 x_1 - 34 x_0 x_2 - 21 x_1 x_2 \rangle &  \langle 15 y_1 - 13 y_3,2 y_0 - y_2 \rangle \\
G_2 &  & \langle 109 x_0 x_1 + 96 x_1^2 + 154 x_0 x_2 + 141 x_1 x_2 \rangle &  \langle 135 y_1 + 77 y_2 - 78 y_3,6 y_0 - y_2 \rangle \\
G_3 &  & \langle 7 x_0 x_1 + 10 x_0 x_2 - 5 x_1 x_2 - 8 x_2^2 \rangle &  \langle y_1 - 3 y_3,2 y_0 - 2 y_2 - 7 y_3 \rangle \\
G_4 & &  \langle 15 x_0 x_1 + 22 x_0 x_2 + 3 x_1 x_2 \rangle &  \langle 15 y_1 + 11 y_2,16 y_0 + y_2 + 15 y_3 \rangle \\
G_5 &  & \langle 13 x_0 x_1 + 22 x_0 x_2 + 9 x_1 x_2 \rangle &  \langle 429 y_1 + 352 y_2 - 39 y_3,22 y_0 + 13 y_3 \rangle \\
G_6 &  & \langle x_0 x_1 + 2 x_0 x_2 + x_1 x_2 \rangle &  \langle y_1 + y_2,2 y_0 + y_3 \rangle  \\
\end{matrix} \end{footnotesize}
$$
\caption{Computation of the $27$ lines on the cubic surface (\ref{eq:cubicsurface})
from the $6$ points in (\ref{eq:sixpoints}).}
\label{tab:27}
\end{table}

In computer algebra, we encode
each of our $27$ objects $E_i,F_{ij},G_j$  algebraically, namely~by its homogeneous ideal in the polynomial ring
$\Q[x_0,x_1,x_2 ]$. In that representation, $E_i$ is an ideal generated by
two linear forms, $F_{ij}$ is a principal ideal generated
by one linear form, and $G_k$ is a principal ideal generated by one quadratic form.
The $27$ ideals are listed in 
Table \ref{tab:27}. 

We next  compute the associated cubic surface in $\PP^3$.
To do this, we choose four linearly independent cubics that vanish at our six points. 
As suggested in \cite{PSS}, a convenient choice~is
\begin{equation}
\label{eq:octamodel}
     y_0 \,=\, F_{12} F_{34} F_{56} \,,\quad 
  y_1 \,=\, F_{13} F_{25} F_{46} \,,\quad y_2 \,=\, F_{12} F_{35} F_{46} \,, \quad
  y_3 \,=\, F_{13} F_{24} F_{56} . 
  \end{equation}
  Using elimination, we find that
the four ternary forms in (\ref{eq:octamodel}) satisfy the cubic equation (\ref{eq:cubicsurface}).
This octanomial \cite{PSS} defines a cubic surface in $\PP^3$. Geometrically, it is the image of the~map
\begin{equation}
\label{eq:blowup}
 \PP^2 \dashrightarrow \PP^3, \,(x_0:x_1:x_2) \mapsto (y_0:y_1:y_2:y_3).
 \end{equation}
This map is well-defined and one-to-one on
$\PP^2 \backslash \{E_1,\ldots,E_6\}$, but it blows up the six points.

The cubic surface contains $27$ lines.
Each of them is represented by its homogeneous prime
ideal, which is generated by two linear forms in the homogeneous
coordinates $(y_0:y_1:y_2:y_3)$ on $\PP^3$.
These $27$ ideals in $\Q[y_0,y_1,y_2,y_3]$ are listed in the third column of Table \ref{tab:27}.
They arise from the $27$ ideals in the second column.
Let $I$ be the bihomogeneous prime ideal in $\Q[x_0,x_1,x_2;y_0,y_1,y_2,y_3]$
whose variety in $\PP^2 \times \PP^3$ is the graph of the map~(\ref{eq:blowup}).
For each object in $\PP^2$, we add its ideal to $I$, we saturate, and then
we eliminate $x_0,x_1,x_2$ to get the ideal in the third column.
This process is reversible. For each line in $\PP^3$, we add its ideal to $I$,
we saturate, and then we eliminate
 $y_0,y_1,y_2,y_3$ to get the ideal in the second column.

The $27$ lines define an arrangement $Y$ on the cubic surface $X$.
The complement $X_\R \backslash Y_\R$ has $130$ connected 
components. These are open regions in $X_\R$, and they arise from the regions in $\PP^2_\R$.
We find $10$ triangles,
$90$ quadrilaterals and $30$ pentagons, as promised in~Theorem~\ref{thm:130}.

\begin{table}[h]
$$ \begin{footnotesize}
\!\!\!\!\! \begin{matrix}
\hbox{Line}\, \,L & & \hbox{Circular ordering of the $10$ points on $L$} 
&\phantom{d}& \hbox{Line}\, \,L & & \hbox{Circular ordering of the $10$ points on $L$} 
\smallskip 
\\

{\color{darkgray} E_2} & & {\color{darkgray} G_1}, {\color{darkgray} F_{12}}, {\color{darkgray} F_{26}}, {\color{darkgray} F_{25}}, {\color{darkgray} F_{24}}, {\color{darkgray} G_3}, {\color{darkgray} F_{23}}, {\color{mybordeaux} G_4},{\color{darkgray} G_5}, {\color{mybordeaux} G_6}  & & 
{\color{mybordeaux} E_1} & & {\color{mybordeaux} G_6}, {\color{darkgray} G_5}, {\color{mybordeaux} G_4}, {\color{darkgray} G_3}, {\color{mybordeaux} G_2}, {\color{darkgray} F_{16}}, {\color{mybordeaux} F_{15}}, {\color{darkgray} F_{14}}, {\color{mybordeaux} F_{13}}, {\color{darkgray} F_{12}} \\
{\color{darkgray} E_4} & & {\color{darkgray} F_{34}}, {\color{darkgray} G_3}, {\color{mybordeaux} F_{24}}, {\color{darkgray} F_{14}}, {\color{mybordeaux} F_{46}}, {\color{darkgray} F_{45}}, {\color{darkgray} G_5}, {\color{mybordeaux} G_6}, {\color{darkgray} G_1}, {\color{mybordeaux} G_2}  & & 
 {\color{mybordeaux} E_3} & & {\color{mybordeaux} G_4}, {\color{darkgray} G_5}, {\color{mybordeaux} G_6}, {\color{darkgray} G_1}, {\color{mybordeaux} G_2}, {\color{darkgray} F_{34}}, {\color{mybordeaux} F_{35}}, {\color{darkgray} F_{36}}, {\color{mybordeaux} F_{13}}, {\color{darkgray} F_{23}} \\
 {\color{darkgray} E_6} & & {\color{darkgray} G_5}, {\color{darkgray} F_{56}}, {\color{mybordeaux} G_4}, {\color{darkgray} G_3}, {\color{mybordeaux} G_2}, {\color{darkgray} F_{16}}, {\color{darkgray} G_1}, {\color{mybordeaux} F_{26}}, {\color{darkgray} F_{36}}, {\color{mybordeaux} F_{46}} & & 
  {\color{mybordeaux} E_5} & &  {\color{mybordeaux} G_6}, {\color{darkgray} F_{45}}, {\color{mybordeaux} F_{35}}, {\color{darkgray} F_{25}}, {\color{mybordeaux} F_{15}}, {\color{darkgray} F_{56}}, {\color{mybordeaux} G_4}, {\color{darkgray} G_3}, {\color{mybordeaux} G_2} ,{\color{darkgray} G_1}
  \vspace{0.15cm} \\ 
 {\color{darkgray} F_{12}} & & {\color{darkgray} F_{45}}, {\color{darkgray} F_{36}}, {\color{mybordeaux} F_{35}}, {\color{darkgray} F_{34}}, {\color{mybordeaux} G_2}, {\color{darkgray} E_2}, {\color{darkgray} G_1}, {\color{mybordeaux} E_1}, {\color{darkgray} F_{56}}, {\color{mybordeaux} F_{46}} & &
 {\color{mybordeaux} F_{13}} & &  {\color{mybordeaux} E_1}, {\color{darkgray} G_1}, {\color{mybordeaux} F_{26}}, {\color{darkgray} F_{25}}, {\color{mybordeaux} F_{24}}, {\color{darkgray} G_3}, {\color{mybordeaux} E_3}, {\color{darkgray} F_{45}}, {\color{mybordeaux} F_{46}} ,{\color{darkgray} F_{56}}\\
 {\color{darkgray} F_{14}} & & {\color{darkgray} F_{36}}, {\color{darkgray} F_{25}}, {\color{mybordeaux} F_{35}}, {\color{darkgray} E_4}, {\color{mybordeaux} G_4}, {\color{darkgray} F_{23}}, {\color{darkgray} F_{56}}, {\color{mybordeaux} E_1}, {\color{darkgray} G_1}, {\color{mybordeaux} F_{26}} & & 
 {\color{mybordeaux} F_{15}} & & {\color{mybordeaux} F_{26}}, {\color{darkgray} F_{36}}, {\color{mybordeaux} F_{46}}, {\color{darkgray} G_5}, {\color{mybordeaux} E_5}, {\color{darkgray} F_{34}}, {\color{mybordeaux} F_{24}}, {\color{darkgray} F_{23}}, {\color{mybordeaux} E_1}, {\color{darkgray} G_1} \\
 {\color{darkgray} F_{16}} & & {\color{darkgray} F_{25}}, {\color{darkgray} F_{34}}, {\color{mybordeaux} F_{24}}, {\color{darkgray} F_{23}}, {\color{mybordeaux} E_1}, {\color{darkgray} E_6}, {\color{darkgray} G_1}, {\color{mybordeaux} G_6}, {\color{darkgray} F_{45}}, {\color{mybordeaux} F_{35}} & &  \\
 {\color{darkgray} F_{23}} & & {\color{darkgray} F_{14}}, {\color{darkgray} F_{56}}, {\color{mybordeaux} F_{15}}, {\color{darkgray} F_{16}}, {\color{mybordeaux} G_2}, {\color{darkgray} G_3}, {\color{darkgray} E_2}, {\color{mybordeaux} E_3}, {\color{darkgray} F_{45}}, {\color{mybordeaux} F_{46}} & & 
 {\color{mybordeaux} F_{24}} & &  {\color{mybordeaux} G_4}, {\color{darkgray} F_{56}}, {\color{mybordeaux} F_{15}}, {\color{darkgray} F_{16}}, {\color{mybordeaux} G_2}, {\color{darkgray} E_2}, {\color{mybordeaux} F_{13}}, {\color{darkgray} F_{36}}, {\color{mybordeaux} F_{35}},{\color{darkgray} E_4} \\
 {\color{darkgray} F_{25}} & & {\color{darkgray} F_{36}}, {\color{darkgray} F_{14}}, {\color{mybordeaux} F_{46}}, {\color{darkgray} G_5}, {\color{mybordeaux} E_5}, {\color{darkgray} F_{16}}, {\color{darkgray} F_{34}}, {\color{mybordeaux} G_2}, {\color{darkgray} E_2}, {\color{mybordeaux} F_{13}} & & 
 {\color{mybordeaux} F_{26}} & &  {\color{mybordeaux} F_{15}}, {\color{darkgray} E_6}, {\color{mybordeaux} G_6}, {\color{darkgray} F_{45}}, {\color{mybordeaux} F_{35}}, {\color{darkgray} F_{34}}, {\color{mybordeaux} G_2}, {\color{darkgray} E_2}, {\color{mybordeaux} F_{13}},{\color{darkgray} F_{14}} \\
   {\color{darkgray} F_{34}} & & {\color{darkgray} E_4}, {\color{darkgray} G_3}, {\color{mybordeaux} G_4}, {\color{darkgray} F_{56}}, {\color{mybordeaux} F_{15}}, {\color{darkgray} F_{16}}, {\color{darkgray} F_{25}}, {\color{mybordeaux} F_{26}}, {\color{darkgray} F_{12}}, {\color{mybordeaux} E_3} & & 
 {\color{mybordeaux} F_{35}} & &  {\color{mybordeaux} F_{46}}, {\color{darkgray} G_5}, {\color{mybordeaux} E_5}, {\color{darkgray} F_{16}}, {\color{mybordeaux} F_{26}}, {\color{darkgray} F_{12}}, {\color{mybordeaux} E_3}, {\color{darkgray} G_3}, {\color{mybordeaux} F_{24}},{\color{darkgray} F_{14}} \\
  {\color{darkgray} F_{36}} & & {\color{darkgray} F_{25}}, {\color{darkgray} F_{14}}, {\color{mybordeaux} F_{15}}, {\color{darkgray} E_6}, {\color{mybordeaux} G_6}, {\color{darkgray} F_{45}}, {\color{darkgray} F_{12}}, {\color{mybordeaux} E_3}, {\color{darkgray} G_3}, {\color{mybordeaux} F_{24}} & &  \\
 {\color{darkgray} F_{45}} & & {\color{darkgray} E_4}, {\color{darkgray} G_5}, {\color{mybordeaux} E_5}, {\color{darkgray} F_{16}}, {\color{mybordeaux} F_{26}}, {\color{darkgray} F_{36}}, {\color{darkgray} F_{12}}, {\color{mybordeaux} F_{13}}, {\color{darkgray} F_{23}}, {\color{mybordeaux} G_4} & & 
 {\color{mybordeaux} F_{46}} & & {\color{mybordeaux} G_4}, {\color{darkgray} E_4}, {\color{mybordeaux} F_{35}}, {\color{darkgray} F_{25}}, {\color{mybordeaux} F_{15}}, {\color{darkgray} E_6}, {\color{mybordeaux} G_6}, {\color{darkgray} F_{12}}, {\color{mybordeaux} F_{13}},{\color{darkgray} F_{23}} \\
 {\color{darkgray} F_{56}} & & {\color{darkgray} G_5}, {\color{darkgray} E_6}, {\color{mybordeaux} E_5}, {\color{darkgray} F_{34}}, {\color{mybordeaux} F_{24}}, {\color{darkgray} F_{23}}, {\color{darkgray} F_{14}}, {\color{mybordeaux} F_{13}}, {\color{darkgray} F_{12}}, {\color{mybordeaux} G_6} \vspace{0.15cm} \\
   {\color{darkgray} G_1} & & {\color{darkgray} E_6}, {\color{darkgray} F_{16}}, {\color{mybordeaux} E_5}, {\color{darkgray} E_4}, {\color{mybordeaux} E_3}, {\color{darkgray} E_2}, {\color{darkgray} F_{12}}, {\color{mybordeaux} F_{13}}, {\color{darkgray} F_{14}}, {\color{mybordeaux} F_{15}} & & 
 {\color{mybordeaux} G_2} & & {\color{mybordeaux} E_5}, {\color{darkgray} E_6}, {\color{mybordeaux} E_1}, {\color{darkgray} F_{23}}, {\color{mybordeaux} F_{24}}, {\color{darkgray} F_{25}}, {\color{mybordeaux} F_{26}}, {\color{darkgray} F_{12}}, {\color{mybordeaux} E_3}, {\color{darkgray} E_4} \\
   {\color{darkgray} G_3} & & {\color{darkgray} E_4}, {\color{darkgray} F_{34}}, {\color{mybordeaux} E_5}, {\color{darkgray} E_6}, {\color{mybordeaux} E_1}, {\color{darkgray} F_{23}}, {\color{darkgray} E_2}, {\color{mybordeaux} F_{13}}, {\color{darkgray} F_{36}}, {\color{mybordeaux} F_{35}} & &  
 {\color{mybordeaux} G_4} & & {\color{mybordeaux} E_3}, {\color{darkgray} F_{45}}, {\color{mybordeaux} F_{46}}, {\color{darkgray} F_{14}}, {\color{mybordeaux} F_{24}}, {\color{darkgray} F_{34}}, {\color{mybordeaux} E_5}, {\color{darkgray} E_6}, {\color{mybordeaux} E_1}, {\color{darkgray} E_2} \\
{\color{darkgray} G_5} & & {\color{darkgray} E_4}, {\color{darkgray} F_{45}}, {\color{mybordeaux} F_{35}}, {\color{darkgray} F_{25}}, {\color{mybordeaux} F_{15}}, {\color{darkgray} E_6}, {\color{darkgray} F_{56}}, {\color{mybordeaux} E_1}, {\color{darkgray} E_2}, {\color{mybordeaux} E_3} & &
 {\color{mybordeaux} G_6} & & {\color{mybordeaux} E_3}, {\color{darkgray} E_4}, {\color{mybordeaux} E_5}, {\color{darkgray} F_{16}}, {\color{mybordeaux} F_{26}}, {\color{darkgray} F_{36}}, {\color{mybordeaux} F_{46}}, {\color{darkgray} F_{56}}, {\color{mybordeaux} E_1},{\color{darkgray} E_2} \\
 \end{matrix} \end{footnotesize}
\vspace{-0.4cm}
$$
\caption{The $10$ lines that meet any given line, listed in the circular order of the intersections.}
\label{tab:circular}
\end{table}

We now determine all $130$ curvy polygons explicitly.
Our point of departure is the fact that each of the $27$
lines $L$ on $X$  intersects precisely $10$ of the other $26$ lines.
We mark these $10$ intersection points on $L$,
and we label each point with the name of the other line.  Now,
topologically, the real line $L$ is actually a circle in $\PP_\R^3$,
so we obtain a unique circular ordering of the other $10$ lines that meet $L$.
We list the resulting $27$ circular orderings in Table \ref{tab:circular}.

Table \ref{tab:circular} describes a graph ${\cal G}(X)$ with $135$ vertices and $270$ edges.
Thus the subdivision of the real surface has f-vector $(135,270,130)$.
The vertices of ${\cal G}(X)$ are the $135$ unordered pairs of lines that meet, like
$\{{\color{mybordeaux} E_1},{\color{darkgray} G_6}\}$, $\{{\color{mybordeaux} E_1},{\color{darkgray} G_5}\}$, $\ldots\,\,$, $\{{\color{mybordeaux} G_6},{\color{mybordeaux} F_{46}}\}$. The edges
are the $270$ pairs of pairs that have an element in common \emph{and} are 
adjacent in the circular order from Table~\ref{tab:circular}. 

For instance,
$\{{\color{mybordeaux} G_6},{\color{darkgray} F_{36}}\}$
and $\{{\color{darkgray} F_{45}},{\color{darkgray} F_{36}}\}$ 
form an edge because they 
appear next to each other in row ${\color{darkgray} F_{36}}$. The graph ${\cal G}(X)$ has a natural embedding in $X_\R$, by sending its vertices to the intersection points of the lines. The chordless $k$-cycles of ${\cal G}(X)$ are referred to as the $k$-gons of ${\cal G}(X)$, or of $X$. After filling in the question marks, ${\cal G}(X)$ is the graph in Figure \ref{fig:endler2}. 
Theorem \ref{thm:130} claims that ${\cal G}(X)$ contains $10$ triangles, $90$ quadrilaterals and $30$ pentagons.

\smallskip

In our code \cite{Zenodo}, we construct our graph in Julia and compute its chordless cycles. We find that it indeed contains $10$ triangles.
They involve the $15$ lines we colored in {\color{darkgray} gray}:
\begin{equation}
\label{eq:10triangles}
\begin{small}
 \begin{matrix}
 & {\color{darkgray} F_{12}} {\color{darkgray} F_{36}} {\color{darkgray} F_{45}} & {\color{darkgray} F_{14}} {\color{darkgray} F_{23}} {\color{darkgray} F_{56}} & {\color{darkgray} F_{14}} {\color{darkgray} F_{25}} {\color{darkgray} F_{36}} &  {\color{darkgray} F_{16}} {\color{darkgray} F_{25}} {\color{darkgray} F_{34}} & \\
 {\color{darkgray} E_2} {\color{darkgray} F_{12}} {\color{darkgray} G_1} & {\color{darkgray} E_2} {\color{darkgray} F_{23}} {\color{darkgray} G_3} &
  {\color{darkgray} E_4} {\color{darkgray} F_{34}} {\color{darkgray} G_3} & {\color{darkgray} E_4} {\color{darkgray} F_{45}} {\color{darkgray} G_5}  & {\color{darkgray} E_6} {\color{darkgray} F_{56}} {\color{darkgray} G_5}& {\color{darkgray} E_6} {\color{darkgray} F_{16}} {\color{darkgray} G_1}.
\end{matrix}  
\end{small}
\end{equation}
 These $10$ triples of lines are among the $45$ tritangent planes of the cubic surface.
The $12$ lines that do not appear in the triangles  are shown in {\color{mybordeaux} red}.
They form a {\em Schl\"afli double-six}:
 \begin{equation}
 \label{eq:doublesix}
  \begin{matrix}
 {\color{mybordeaux} E_1} & {\color{mybordeaux} E_3} & {\color{mybordeaux} E_5}	& {\color{mybordeaux} F_{24}} & {\color{mybordeaux} F_{26}} & {\color{mybordeaux} F_{46}} \\
 {\color{mybordeaux} F_{35}} & {\color{mybordeaux} F_{15}} & {\color{mybordeaux} F_{13}} &  {\color{mybordeaux} G_6} &  {\color{mybordeaux} G_4} & {\color{mybordeaux} G_2} \\
 \end{matrix}
 \end{equation}

We next list the curvy quadrilaterals and curvy pentagons in our arrangements. They are bounded
by the $30$ five-cycles and the $90$ four-cycles in our graph. The $30$ pentagons are 

\begin{footnotesize}
\begin{equation}
\label{eq:pentagons}
\begin{array}{ccccc}
{\color{mybordeaux} F_{26}}{\color{darkgray} F_{45}}{\color{darkgray} F_{36}}{\color{darkgray} F_{12}}{\color{mybordeaux} F_{35}} & {\color{darkgray} F_{12}}{\color{darkgray} F_{45}}{\color{mybordeaux} F_{13}}{\color{mybordeaux} E_3}{\color{darkgray} F_{36}} & {\color{darkgray} F_{25}}{\color{darkgray} F_{34}}{\color{mybordeaux} F_{26}}{\color{mybordeaux} F_{35}}{\color{darkgray} F_{16}} & {\color{mybordeaux} F_{46}}{\color{mybordeaux} G_6}{\color{darkgray} F_{56}}{\color{darkgray} G_5}{\color{darkgray} E_6} & {\color{darkgray} E_6}{\color{darkgray} F_{16}}{\color{darkgray} G_1}{\color{mybordeaux} E_5}{\color{mybordeaux} G_2} \\
{\color{darkgray} E_2}{\color{mybordeaux} F_{26}}{\color{mybordeaux} F_{13}}{\color{darkgray} G_1}{\color{darkgray} F_{12}} & {\color{mybordeaux} E_3}{\color{mybordeaux} G_4}{\color{darkgray} F_{45}}{\color{darkgray} E_4}{\color{darkgray} G_5} & {\color{mybordeaux} E_3}{\color{darkgray} G_1}{\color{darkgray} E_2}{\color{darkgray} F_{12}}{\color{mybordeaux} G_2} & {\color{mybordeaux} F_{35}}{\color{darkgray} G_3}{\color{darkgray} E_4}{\color{darkgray} F_{34}}{\color{mybordeaux} E_3} & {\color{darkgray} F_{23}}{\color{darkgray} F_{56}}{\color{mybordeaux} F_{24}}{\color{mybordeaux} G_4}{\color{darkgray} F_{14}} \\
{\color{darkgray} F_{16}}{\color{darkgray} F_{34}}{\color{mybordeaux} F_{15}}{\color{mybordeaux} E_5}{\color{darkgray} F_{25}} & {\color{darkgray} F_{14}}{\color{mybordeaux} F_{35}}{\color{mybordeaux} F_{24}} {\color{darkgray} F_{36}}{\color{darkgray} F_{25}} & {\color{mybordeaux} E_5}{\color{darkgray} F_{56}}{\color{darkgray} E_6}{\color{darkgray} G_5}{\color{mybordeaux} F_{15}} & {\color{mybordeaux} F_{13}}{\color{darkgray} G_3}{\color{darkgray} E_2}{\color{darkgray} F_{23}}{\color{mybordeaux} E_3} & {\color{darkgray} E_2}{\color{mybordeaux} F_{24}}{\color{mybordeaux} G_2}{\color{darkgray} F_{23}}{\color{darkgray} G_3} \\
{\color{darkgray} F_{23}}{\color{darkgray} F_{56}}{\color{darkgray} F_{14}}{\color{mybordeaux} E_1}{\color{mybordeaux} F_{15}} & {\color{darkgray} F_{14}}{\color{darkgray} F_{25}}{\color{mybordeaux} F_{46}}{\color{mybordeaux} F_{15}}{\color{darkgray} F_{36}} & {\color{darkgray} F_{34}}{\color{mybordeaux} G_4}{\color{mybordeaux} F_{24}}{\color{darkgray} E_4}{\color{darkgray} G_3} & {\color{darkgray} E_4}{\color{mybordeaux} G_2}{\color{mybordeaux} E_5}{\color{darkgray} G_3}{\color{darkgray} F_{34}} & {\color{mybordeaux} E_1}{\color{darkgray} F_{12}}{\color{darkgray} G_1}{\color{darkgray} E_2}{\color{mybordeaux} G_6} \\
{\color{darkgray} E_4}{\color{darkgray} G_5}{\color{darkgray} F_{45}}{\color{mybordeaux} E_5}{\color{mybordeaux} G_6} & {\color{mybordeaux} F_{13}}{\color{darkgray} F_{25}}{\color{darkgray} F_{36}}{\color{darkgray} F_{14}}{\color{mybordeaux} F_{26}} & {\color{darkgray} E_6}{\color{darkgray} F_{16}}{\color{mybordeaux} E_1}{\color{mybordeaux} F_{15}}{\color{darkgray} G_1} & {\color{mybordeaux} F_{13}}{\color{darkgray} F_{56}}{\color{darkgray} F_{14}}{\color{darkgray} F_{23}}{\color{mybordeaux} F_{46}} & {\color{darkgray} E_4}{\color{mybordeaux} F_{46}}{\color{mybordeaux} F_{35}}{\color{darkgray} G_5}{\color{darkgray} F_{45}} \\
{\color{darkgray} E_6}{\color{mybordeaux} G_4}{\color{mybordeaux} E_1}{\color{darkgray} G_5}{\color{darkgray} F_{56}} & {\color{mybordeaux} F_{46}}{\color{mybordeaux} G_6}{\color{darkgray} F_{36}}{\color{darkgray} F_{45}}{\color{darkgray} F_{12}} & {\color{darkgray} F_{25}}{\color{darkgray} F_{34}}{\color{darkgray} F_{16}}{\color{mybordeaux} F_{24}}{\color{mybordeaux} G_2} & {\color{darkgray} F_{16}}{\color{darkgray} G_1}{\color{darkgray} E_6}{\color{mybordeaux} F_{26}}{\color{mybordeaux} G_6} & {\color{mybordeaux} E_1}{\color{mybordeaux} G_4}{\color{darkgray} E_2}{\color{darkgray} F_{23}}{\color{darkgray} G_3}.
\end{array}
\end{equation}
\end{footnotesize}
The $90$ quadrilaterals~are
\begin{footnotesize}
\begin{equation}
\label{eq:quadrilaterals}
\begin{array}{cccccc}
{\color{darkgray} E_2}{\color{mybordeaux} F_{26}}{\color{mybordeaux} G_2}{\color{darkgray} F_{12}} & {\color{darkgray} E_6}{\color{darkgray} F_{36}}{\color{mybordeaux} F_{15}}{\color{mybordeaux} F_{46}} & {\color{mybordeaux} E_5}{\color{mybordeaux} F_{35}}{\color{darkgray} G_5}{\color{darkgray} F_{25}} & {\color{darkgray} F_{34}}{\color{darkgray} F_{56}}{\color{mybordeaux} F_{24}}{\color{mybordeaux} F_{15}} & {\color{darkgray} F_{25}}{\color{mybordeaux} G_2}{\color{mybordeaux} F_{24}}{\color{darkgray} E_2} & {\color{darkgray} E_2}{\color{mybordeaux} G_6}{\color{mybordeaux} E_1}{\color{darkgray} G_5} \\
{\color{mybordeaux} E_5}{\color{mybordeaux} F_{15}}{\color{darkgray} G_5}{\color{darkgray} F_{25}} & {\color{mybordeaux} E_3}{\color{mybordeaux} G_4}{\color{darkgray} F_{45}}{\color{darkgray} F_{23}} & {\color{darkgray} E_4}{\color{mybordeaux} F_{24}}{\color{mybordeaux} G_4}{\color{darkgray} F_{14}} & {\color{darkgray} F_{12}}{\color{darkgray} F_{45}}{\color{mybordeaux} F_{13}}{\color{mybordeaux} F_{46}} & {\color{darkgray} F_{34}}{\color{mybordeaux} G_4}{\color{mybordeaux} E_5}{\color{darkgray} F_{56}} & {\color{darkgray} F_{34}}{\color{darkgray} F_{56}}{\color{mybordeaux} F_{24}}{\color{mybordeaux} G_4} \\
{\color{mybordeaux} E_3}{\color{mybordeaux} G_4}{\color{darkgray} E_2}{\color{darkgray} G_5} & {\color{darkgray} F_{12}}{\color{mybordeaux} F_{35}}{\color{mybordeaux} E_3}{\color{darkgray} F_{34}} & {\color{mybordeaux} E_5}{\color{mybordeaux} F_{35}}{\color{darkgray} G_5}{\color{darkgray} F_{45}} & {\color{darkgray} F_{23}}{\color{darkgray} F_{45}}{\color{mybordeaux} F_{13}}{\color{mybordeaux} E_3} & {\color{mybordeaux} E_1}{\color{darkgray} F_{12}}{\color{darkgray} G_1}{\color{mybordeaux} F_{13}} & {\color{darkgray} E_6}{\color{darkgray} G_3}{\color{mybordeaux} E_5}{\color{mybordeaux} G_2} \\
{\color{darkgray} E_4}{\color{darkgray} G_1}{\color{mybordeaux} E_3}{\color{mybordeaux} G_6} & {\color{mybordeaux} F_{26}}{\color{mybordeaux} F_{35}}{\color{darkgray} F_{12}}{\color{darkgray} F_{34}} & {\color{darkgray} F_{12}}{\color{mybordeaux} F_{35}}{\color{mybordeaux} E_3}{\color{darkgray} F_{36}} & {\color{darkgray} E_6}{\color{darkgray} F_{16}}{\color{mybordeaux} E_1}{\color{mybordeaux} G_2} & {\color{darkgray} F_{16}}{\color{mybordeaux} G_6}{\color{mybordeaux} E_5}{\color{darkgray} F_{45}} & {\color{mybordeaux} F_{46}}{\color{mybordeaux} G_6}{\color{darkgray} F_{56}}{\color{darkgray} F_{12}} \\
{\color{darkgray} F_{14}}{\color{darkgray} G_1}{\color{mybordeaux} E_1}{\color{mybordeaux} F_{13}} & {\color{mybordeaux} F_{15}}{\color{mybordeaux} F_{26}}{\color{darkgray} F_{14}}{\color{darkgray} F_{36}} & {\color{darkgray} E_2}{\color{mybordeaux} F_{24}}{\color{mybordeaux} F_{13}}{\color{darkgray} F_{25}} & {\color{mybordeaux} F_{13}}{\color{darkgray} F_{56}}{\color{darkgray} F_{12}}{\color{mybordeaux} E_1} & {\color{mybordeaux} E_3}{\color{mybordeaux} G_4}{\color{darkgray} E_2}{\color{darkgray} F_{23}} & {\color{mybordeaux} F_{15}}{\color{mybordeaux} F_{26}}{\color{darkgray} E_6}{\color{darkgray} G_1} \\
{\color{darkgray} E_4}{\color{darkgray} G_1}{\color{mybordeaux} E_5}{\color{mybordeaux} G_6} & {\color{mybordeaux} F_{35}}{\color{mybordeaux} F_{46}}{\color{darkgray} F_{25}}{\color{darkgray} G_5} & {\color{darkgray} E_2}{\color{mybordeaux} G_6}{\color{mybordeaux} E_3}{\color{darkgray} G_1} & {\color{darkgray} F_{14}}{\color{darkgray} G_1}{\color{mybordeaux} E_1}{\color{mybordeaux} F_{15}} & {\color{mybordeaux} F_{13}}{\color{darkgray} F_{56}}{\color{darkgray} F_{12}}{\color{mybordeaux} F_{46}} & {\color{darkgray} E_6}{\color{mybordeaux} G_4}{\color{mybordeaux} E_1}{\color{darkgray} G_3} \\
{\color{darkgray} E_6}{\color{mybordeaux} G_4}{\color{mybordeaux} E_5}{\color{darkgray} F_{56}} & {\color{mybordeaux} E_1}{\color{darkgray} F_{16}}{\color{darkgray} F_{23}}{\color{mybordeaux} F_{15}} & {\color{darkgray} F_{23}}{\color{darkgray} F_{45}}{\color{mybordeaux} G_4}{\color{mybordeaux} F_{46}} & {\color{darkgray} E_4}{\color{mybordeaux} F_{46}}{\color{mybordeaux} G_4}{\color{darkgray} F_{14}} & {\color{mybordeaux} F_{13}}{\color{mybordeaux} F_{46}}{\color{darkgray} F_{23}}{\color{darkgray} F_{45}} & {\color{mybordeaux} E_5}{\color{mybordeaux} F_{35}}{\color{darkgray} F_{16}}{\color{darkgray} F_{25}} \\
{\color{mybordeaux} E_1}{\color{darkgray} F_{12}}{\color{darkgray} F_{56}}{\color{mybordeaux} G_6} & {\color{darkgray} E_4}{\color{darkgray} G_5}{\color{mybordeaux} E_3}{\color{mybordeaux} G_6} & {\color{mybordeaux} F_{35}}{\color{darkgray} G_3}{\color{darkgray} E_4}{\color{mybordeaux} F_{24}} & {\color{darkgray} E_6}{\color{darkgray} F_{36}}{\color{mybordeaux} F_{15}}{\color{mybordeaux} F_{26}} & {\color{darkgray} E_4}{\color{mybordeaux} F_{46}}{\color{mybordeaux} F_{35}}{\color{darkgray} F_{14}} & {\color{mybordeaux} F_{26}}{\color{darkgray} F_{45}}{\color{darkgray} F_{36}}{\color{mybordeaux} G_6} \\
{\color{darkgray} E_2}{\color{mybordeaux} F_{26}}{\color{mybordeaux} F_{13}}{\color{darkgray} F_{25}} & {\color{darkgray} E_2}{\color{mybordeaux} G_6}{\color{mybordeaux} E_3}{\color{darkgray} G_5} & {\color{mybordeaux} E_1}{\color{darkgray} G_5}{\color{darkgray} E_2}{\color{mybordeaux} G_4} & {\color{darkgray} E_4}{\color{mybordeaux} F_{24}}{\color{darkgray} F_{14}}{\color{mybordeaux} F_{35}} & {\color{mybordeaux} F_{26}}{\color{darkgray} F_{34}}{\color{darkgray} F_{12}}{\color{mybordeaux} G_2} & {\color{darkgray} F_{25}}{\color{mybordeaux} G_2}{\color{darkgray} E_2}{\color{mybordeaux} F_{26}} \\
{\color{darkgray} F_{14}}{\color{darkgray} F_{23}}{\color{mybordeaux} G_4}{\color{mybordeaux} F_{46}} & {\color{mybordeaux} E_5}{\color{mybordeaux} F_{35}}{\color{darkgray} F_{16}}{\color{darkgray} F_{45}} & {\color{mybordeaux} F_{13}}{\color{darkgray} G_3}{\color{darkgray} F_{36}}{\color{mybordeaux} F_{24}} & {\color{mybordeaux} F_{35}}{\color{darkgray} G_3}{\color{darkgray} F_{36}}{\color{mybordeaux} E_3} & {\color{darkgray} E_6}{\color{mybordeaux} G_4}{\color{mybordeaux} E_5}{\color{darkgray} G_3} & {\color{darkgray} E_6}{\color{darkgray} G_3}{\color{mybordeaux} E_1}{\color{mybordeaux} G_2} \\
{\color{mybordeaux} F_{15}}{\color{darkgray} G_5}{\color{darkgray} E_6}{\color{mybordeaux} F_{46}} & {\color{darkgray} F_{16}}{\color{darkgray} F_{34}}{\color{mybordeaux} F_{15}}{\color{mybordeaux} F_{24}} & {\color{darkgray} E_4}{\color{mybordeaux} G_2}{\color{mybordeaux} E_3}{\color{darkgray} F_{34}} & {\color{darkgray} F_{23}}{\color{darkgray} F_{56}}{\color{mybordeaux} F_{24}}{\color{mybordeaux} F_{15}} & {\color{mybordeaux} F_{35}}{\color{darkgray} F_{14}}{\color{darkgray} F_{25}}{\color{mybordeaux} F_{46}} & {\color{mybordeaux} F_{13}}{\color{darkgray} F_{25}}{\color{mybordeaux} F_{24}}{\color{darkgray} F_{36}} \\
{\color{mybordeaux} F_{26}}{\color{darkgray} F_{45}}{\color{darkgray} F_{16}}{\color{mybordeaux} F_{35}} & {\color{mybordeaux} F_{24}}{\color{mybordeaux} G_2}{\color{darkgray} F_{16}}{\color{darkgray} F_{23}} & {\color{darkgray} F_{23}}{\color{mybordeaux} G_2}{\color{mybordeaux} E_1}{\color{darkgray} G_3} & {\color{mybordeaux} E_1}{\color{darkgray} G_5}{\color{darkgray} F_{56}}{\color{mybordeaux} G_6} & {\color{mybordeaux} F_{46}}{\color{mybordeaux} G_6}{\color{darkgray} F_{36}}{\color{darkgray} E_6} & {\color{darkgray} E_2}{\color{mybordeaux} F_{24}}{\color{mybordeaux} F_{13}}{\color{darkgray} G_3} \\
{\color{mybordeaux} F_{15}}{\color{darkgray} G_5}{\color{darkgray} F_{25}}{\color{mybordeaux} F_{46}} & {\color{mybordeaux} F_{15}}{\color{darkgray} F_{23}}{\color{darkgray} F_{16}}{\color{mybordeaux} F_{24}} & {\color{mybordeaux} F_{13}}{\color{mybordeaux} E_1}{\color{darkgray} F_{56}}{\color{darkgray} F_{14}} & {\color{mybordeaux} E_3}{\color{darkgray} F_{34}}{\color{darkgray} F_{12}}{\color{mybordeaux} G_2} & {\color{mybordeaux} E_5}{\color{mybordeaux} F_{15}}
{\color{darkgray} F_{56}}{\color{darkgray} F_{34}} & {\color{darkgray} E_6}{\color{darkgray} F_{36}}{\color{mybordeaux} F_{26}}{\color{mybordeaux} G_6} \\
{\color{darkgray} E_4}{\color{mybordeaux} G_2}{\color{mybordeaux} E_5}{\color{darkgray} G_1} & {\color{mybordeaux} F_{13}}{\color{mybordeaux} E_3}{\color{darkgray} G_3}{\color{darkgray} F_{36}} & {\color{mybordeaux} F_{26}}{\color{darkgray} F_{45}}{\color{darkgray} F_{16}}{\color{mybordeaux} G_6} & {\color{darkgray} F_{34}}{\color{mybordeaux} G_4}{\color{mybordeaux} E_5}{\color{darkgray} G_3} & {\color{mybordeaux} E_1}{\color{darkgray} F_{16}}{\color{darkgray} F_{23}}{\color{mybordeaux} G_2} & {\color{darkgray} F_{14}}{\color{darkgray} G_1}{\color{mybordeaux} F_{15}}{\color{mybordeaux} F_{26}} \\
{\color{darkgray} F_{16}}{\color{darkgray} G_1}{\color{mybordeaux} E_5}{\color{mybordeaux} G_6} & {\color{darkgray} F_{14}}{\color{darkgray} G_1}{\color{mybordeaux} F_{13}}{\color{mybordeaux} F_{26}} & {\color{darkgray} G_3}{\color{mybordeaux} F_{35}}{\color{mybordeaux} F_{24}}{\color{darkgray} F_{36}} & {\color{darkgray} E_4}{\color{mybordeaux} G_2}{\color{mybordeaux} E_3}{\color{darkgray} G_1} & {\color{darkgray} E_4}{\color{mybordeaux} F_{46}}{\color{mybordeaux} G_4}{\color{darkgray} F_{45}} & {\color{darkgray} F_{25}}{\color{darkgray} F_{34}}{\color{mybordeaux} F_{26}}{\color{mybordeaux} G_2}.
\end{array}
\end{equation}
\end{footnotesize}
At first glance, these lists look complicated. However, the situation is more transparent 
and more beautiful than it may seem.  To reveal this beauty is the objective of the next section.

\section{Schl\"afli Graph, Weyl Group, and Pictures}
\label{sec3}

We begin by introducing the
{\em Schl\"afli graph}~$\mathcal{S}^{10}_{27}$. This graph has
$27$ vertices, one for each of the $27$ lines $E_i,F_{ij},G_j$ on the cubic surface $X$.
Two vertices of $\mathcal{S}^{10}_{27}$ are connected by an edge
if and only if the two corresponding lines intersect.  Hence, the Schl\"afli graph 
$\mathcal{S}^{10}_{27}$ is $10$-regular, and it has $135$ edges.
These edges are listed in Table \ref{tab:circular}. Here is a key result.

\begin{thm} The symmetry group of the
Schl\"afli graph $\mathcal{S}^{10}_{27}$ is generated by the symmetric group $S_6$, 
which acts by permuting the labels $\{1,2,3,4,5,6\}$, together with one involution
\begin{equation}
\label{eq:cremona}
( {\color{mybordeaux}  F_{13}} \, {\color{darkgray} G_5}) \,
( {\color{mybordeaux}  F_{15}} \, {\color{darkgray} G_3}) \,
( {\color{mybordeaux}  F_{35}} \, {\color{darkgray} G_1}) \,
( {\color{mybordeaux}  F_{24}} \, {\color{darkgray} E_6 }) \,
( {\color{mybordeaux}  F_{26}} \, {\color{darkgray} E_4 }) \,
( {\color{mybordeaux}  F_{46}} \, {\color{darkgray} E_2 }) .
\end{equation}
This symmetry group has order $51840$. It is the 
{\em Weyl group} $W({\rm E}_6)$ of the {\em root system} ${\rm E}_6$. 
\end{thm}

We conclude that the symmetry group
$W({\rm E}_6)$ acts transitively on the labels $E_i,F_{ij},G_j$ of the $27$ lines.
Each transposition in $S_6$ acts as a  product of six $2$-cycles, just like (\ref{eq:cremona}). 
 For instance, the transposition $1 \leftrightarrow 2$ acts as
 $\,({\color{mybordeaux} E_1}\,{\color{darkgray} E_2})({\color{mybordeaux} G_2}\,{\color{darkgray} G_1})({\color{mybordeaux} F_{13}} \,{\color{darkgray} F_{23}})( {\color{mybordeaux} F_{24}}\,{\color{darkgray} F_{14}}) ({\color{mybordeaux} F_{15}}\,{\color{darkgray} F_{25}})({\color{mybordeaux} F_{26}}\,{\color{darkgray} F_{16}})$.
 We obtain the following result for the 
$W({\rm E}_6)$-action on the labels of the $130$ regions of~$X \backslash Y$.

\begin{prop} \label{prop:Schlafli}
The Schl\"afli graph $\mathcal{S}^{10}_{27}$ has $45$ triangles,
namely the $W({\rm E}_6)$-orbit of (\ref{eq:10triangles}).
It has $1080$ four-cycles, namely the
$W({\rm E}_6)$-orbit of (\ref{eq:quadrilaterals}).
There are two orbits of five-cycles.
The orbit of (\ref{eq:pentagons}) has size $4320$. The orbit
of the pentagon $\,{\color{darkgray} F_{14}} {\color{mybordeaux} F_{15}} {\color{darkgray} F_{23}} {\color{mybordeaux} F_{24}} {\color{darkgray} E_4}\,$ has size $2592$.
\end{prop}
\begin{proof}
These orbits are found via a direct computation \cite{Zenodo}. 
\end{proof}

A cycle of ${\cal S}^{10}_{27}$ may or may not appear as a curvy polygon in any particular cubic surface $X$.
If it does then we call it {\em exposed}. The exposed cycles are the chordless cycles of the graph ${\cal G}(X)$ constructed in Section \ref{sec2}. For instance,  among the $4320$ pentagons in the first orbit,
precisely $30$ are exposed in $X$. The pentagons in the second orbit are never exposed.
The stabilizer of the exposed triangles (\ref{eq:10triangles}) is a subgroup
of order $120$ in $W({\rm E}_6)$. Hence $\mathcal{S}^{10}_{27}$ has $432 $ distinct configurations
of $10$ exposed triangles for the various real cubic surfaces $X$.
We shall return to this in our discussion of the moduli space
and its positive geometry structure.

\begin{figure}[h]
$$ \!\!\!\!\!\!\!
\includegraphics[height = 6.8cm]{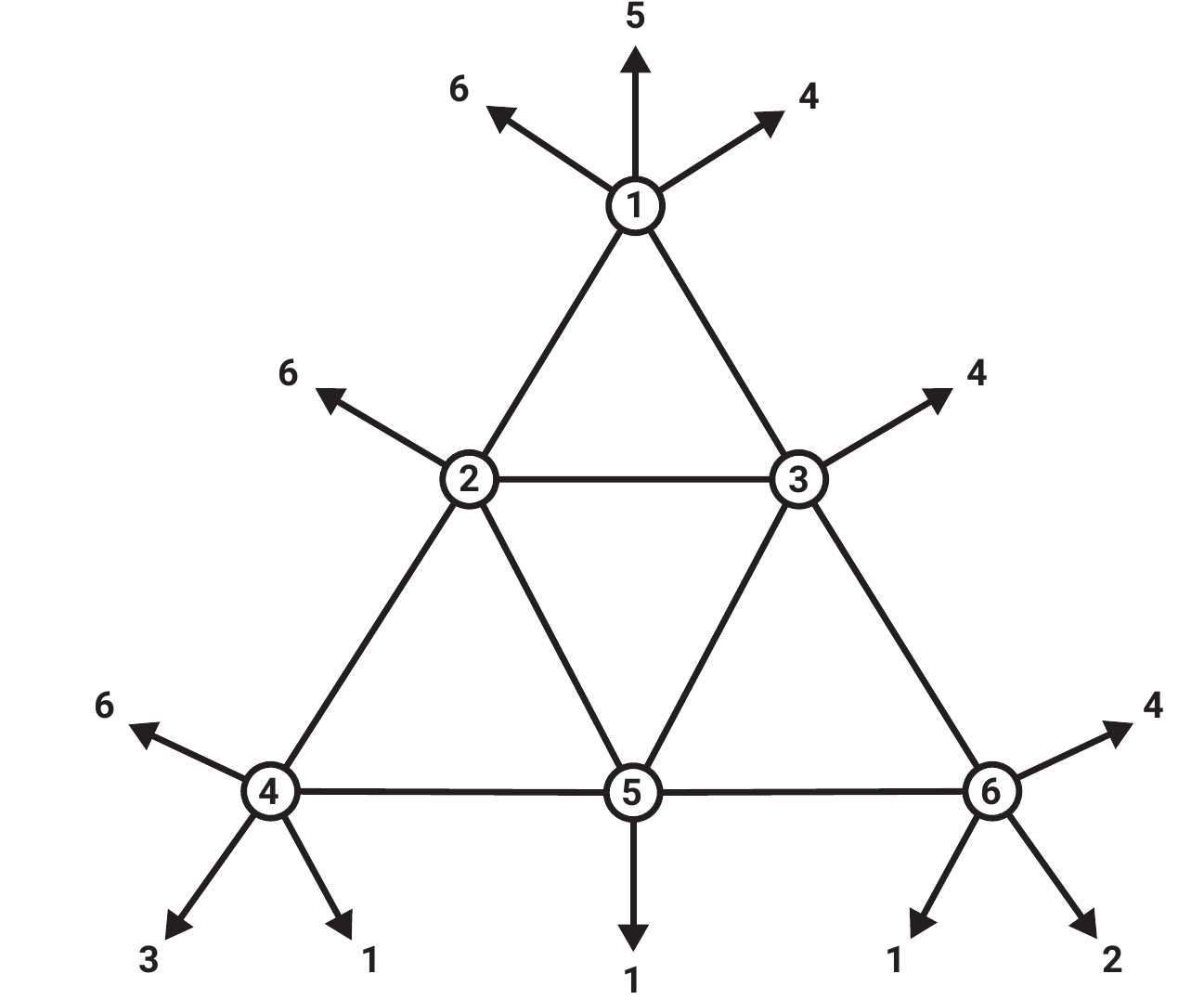} \, \, \, 
\includegraphics[height = 6.8cm]{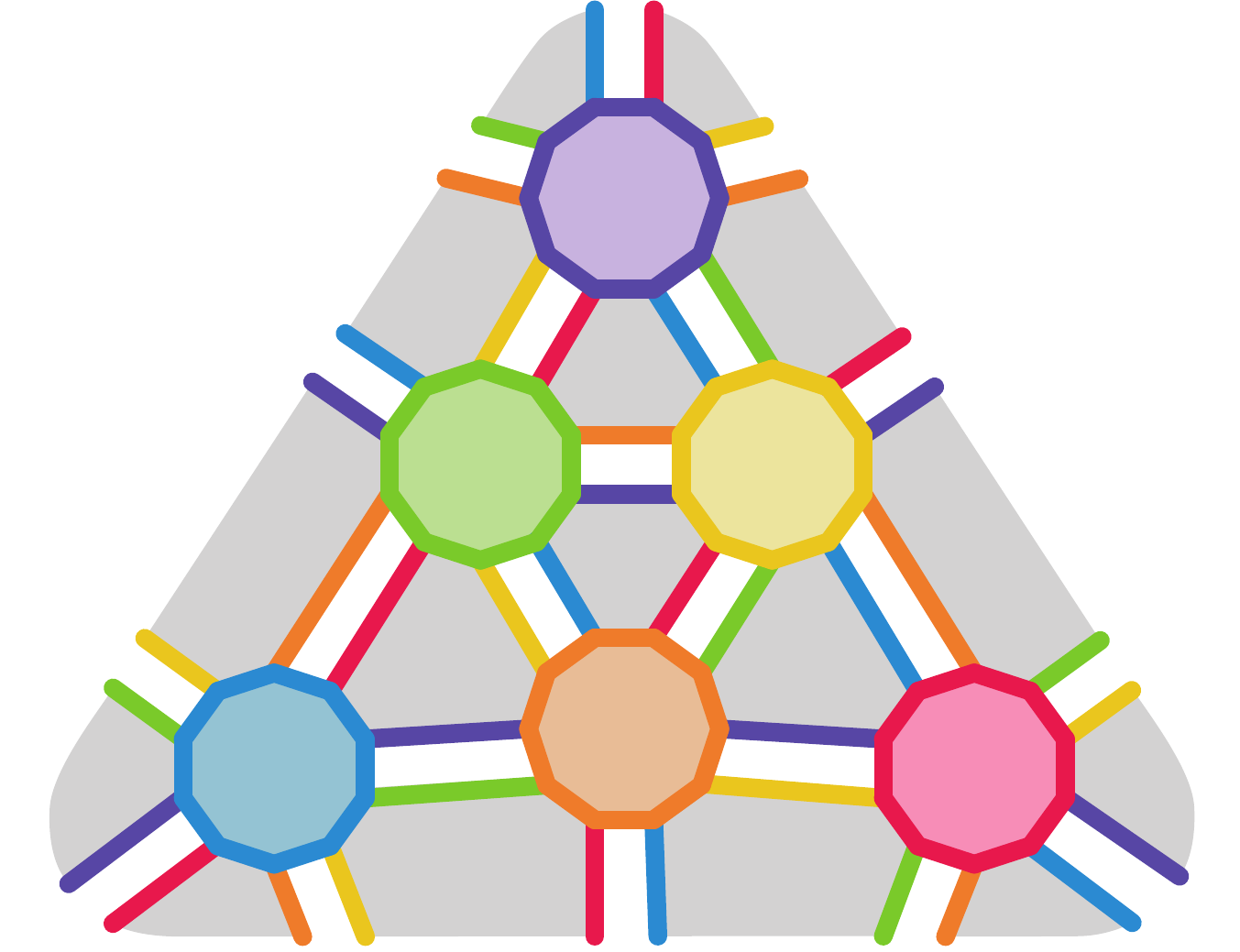}
$$
\caption{Minimal triangulation of the real projective plane (left).
Blowing up the six points  $E_1,E_2,\ldots,E_6$ and
inserting  the six conics $G_1,G_2,\ldots,G_6$ yields a Schl\"afli double-six (right).}
\label{fig:endler1}
\end{figure}

We next discuss a pretty picture that shows the $130$  regions
on the cubic surface. Using the involution (\ref{eq:cremona}), we first replace the
double-six in (\ref{eq:doublesix}) with the standard double-six:
 \begin{equation}
 \label{eq:doublesix2}
  \begin{matrix}
  E_1 & E_2 & E_3 & E_4 & E_5 & E_6 \\
  G_1 & G_2 & G_3 & G_4 & G_5 & G_6 
  \end{matrix}
 \end{equation}
This corresponds to six points
$E_1, \ldots, E_6$ in the projective plane $\PP_\R^2$ together with the
six conics $G_1,\ldots,G_6$, where $G_i$ passes though the points other than $E_i$.
Figure \ref{fig:endler1} shows
this arrangement. On the left we see the minimal triangulation
of $\PP_\R^2$. This has six vertices, $15$ edges and $10$ triangles.
The six vertices are the points $E_1,\ldots, E_6$ that are to be blown up. The outward pointing arrows are to be identified as indicated by their label. For instance, the upward pointing arrow leaving $E_1$ is an edge which connects to $E_5$ from below. 

The picture on the right in Figure \ref{fig:endler1} is a cartoon of $\PP^2_\R$ after the blow-up.
Each point $E_i$ has been replaced by a circle. The circle is a $10$-gon with opposite edges
identified, so it is a $5$-cycle. The conic $G_i$ is also $5$-cycle. Here, $G_i$ is 
 shown in the same color as~$E_i$.
We obtain a subdivision of our cubic surface $X$ into $25$ polygons,
namely $10$ gray hexagons and $15$ white quadrilaterals. We will see in Example \ref{ex:cr17} that these are positive geometries, just like the $130$ polygons in Theorem \ref{thm:130}.
We count $30$ vertices and $60$ edges, namely ten of each color.

We now add the $15$ lines $F_{ij}$ that are not in  (\ref{eq:doublesix2})
to the colorful picture in Figure \ref{fig:endler1}.
They are indicated by the thin line segments in Figure \ref{fig:endler2}.
Each of the $15$ quadrilaterals in Figure~\ref{fig:endler1} is divided into
four small quadrilaterals.
Each of the $10$ shaded hexagons in Figure~\ref{fig:endler1} is divided
into either six or seven polygons. 
The $10$ question marks indicate that indecision.

\begin{figure}[h]
$$
\includegraphics[width = 0.77\linewidth]{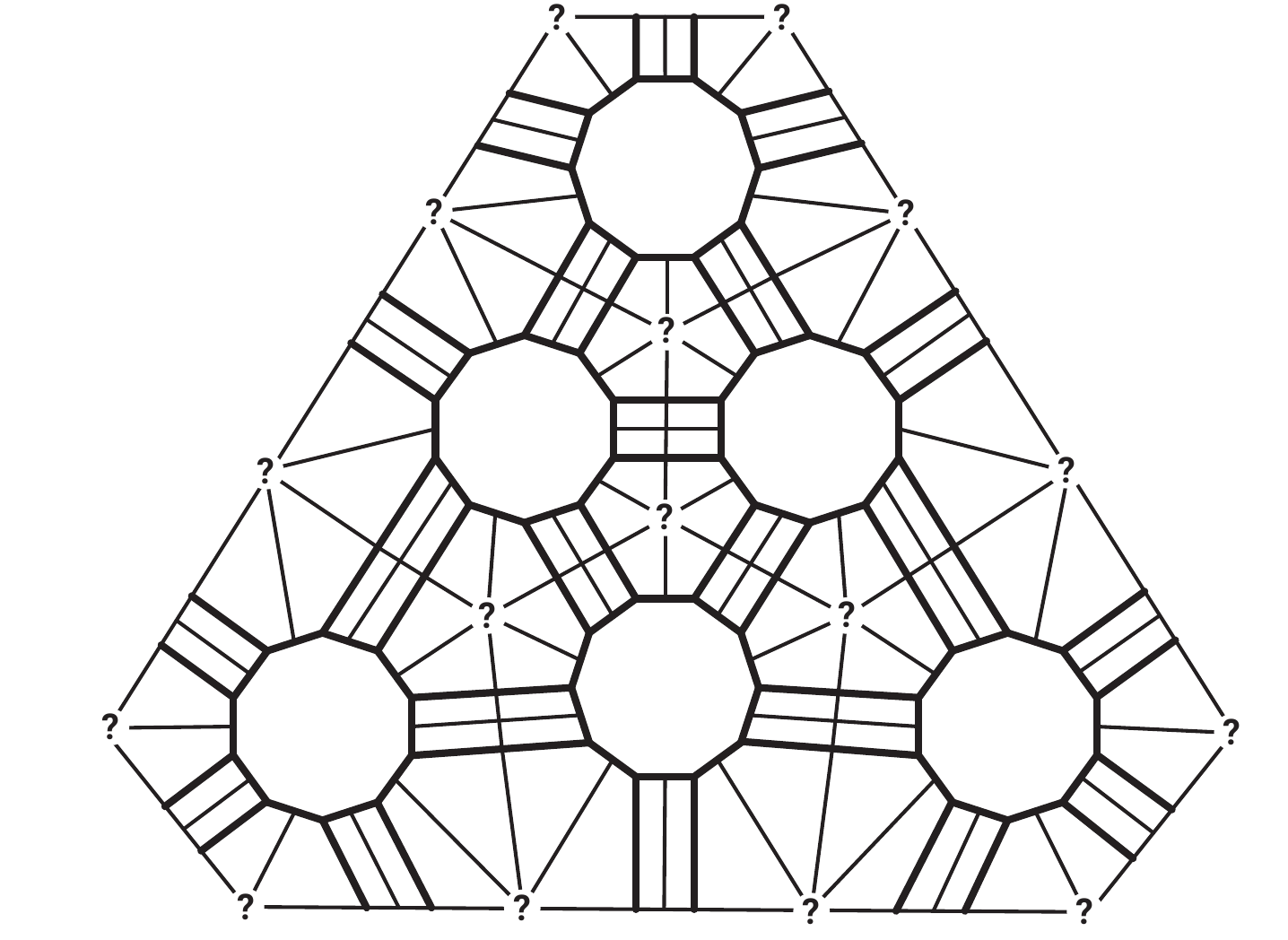}
$$
\caption{The $130$ regions on a real cubic surface.
We insert $15$ thin lines into $12$ colorful lines from the double-six in Figure \ref{fig:endler1}.
Question marks indicate triangles or Eckardt points.}
\label{fig:endler2}
\end{figure}

If the three lines in a hexagon meet, then the
question mark becomes an  {\em Eckardt point}, and the
hexagon splits into six quadrilaterals. Theorem \ref{thm:130} refers
to the generic case when each
question mark is a triangle. Here the  hexagon splits
into $3$ pentagons, $3$ quadrilaterals, and the triangle.
Figure \ref{fig:endler2} is a picture of the graph ${\cal G}(X)$ from Section \ref{sec2}. It shows all $130$ polygons,
namely $10$ triangles, $90$ quadrilaterals and $30$ pentagons.
Each of the $27$ lines is represented by a cycle of length $10$. 
 The intersection graph of these cycles is the Schl\"afli graph $\mathcal{S}^{10}_{27}$.
 Our readers are encouraged to spot all
  $270$ edges and all $135$ vertices in Figure \ref{fig:endler2}.

\section{Differential Forms}
\label{sec4}
A distinguished feature of a positive geometry is the fact that it comes with a canonical differential form. Since our geometries
live in $\PP^2$ and in $\PP^3$, we start by explaining
top-dimensional differential forms on these projective spaces. Such a form on $\PP^3$ looks like
\begin{equation}
\label{eq:formomega3}
\begin{matrix}
 \omega \, = \, r(y_0,y_1,y_2, y_3) \, \Omega_{\mathbb{P}^3}, \,\,\,
 {\rm where}
 \qquad \qquad\qquad \qquad\qquad \qquad \qquad \qquad
  \smallskip \\ \qquad
 \Omega_{\mathbb{P}^3} \,=\, 
 {\rm d}y_0 \!\wedge\! {\rm d} y_1 \! \wedge\! {\rm d} y_2 
 -  {\rm d}y_0 \!\wedge\! {\rm d} y_1 \! \wedge\! {\rm d} y_3 
 +  {\rm d}y_0 \!\wedge\! {\rm d} y_2 \! \wedge\! {\rm d} y_3 
-  {\rm d}y_1 \!\wedge\! {\rm d} y_2 \! \wedge\! {\rm d} y_3 .
\end{matrix}
\end{equation}
The coefficient $r(y)$ is a homogeneous rational function of degree $-4$.
Hence $r(y) = g(y)/h(y)$ where $g(y)$ and $h(y)$ are homogeneous polynomials
satisfying ${\rm degree}(h) = {\rm degree}(g) + 4$.

Similarly, a differential $2$-form on the projective plane $\PP^2$ is an expression
\begin{equation}
\label{eq:formomega2}
\begin{matrix}
 \omega \, = \, r(x_0,x_1,x_2) \, \Omega_{\mathbb{P}^2}, \,\,\,
 {\rm where} \quad
  \Omega_{\mathbb{P}^2} \,=\, 
 {\rm d}x_0 \!\wedge\! {\rm d} x_1 -
 {\rm d}x_0 \!\wedge\! {\rm d} x_2 +
{\rm d}x_1 \!\wedge\! {\rm d} x_2.
\end{matrix}
\end{equation}
The coefficient $r(x)$ is a homogeneous rational function of degree $-3$ in $x = (x_0,x_1,x_2)$.

The purpose of a $2$-form $\omega$ is to be integrated on a $2$-dimensional integration region~$\Gamma \subset \mathbb{P}^2$. In the affine chart $\{x_0 \neq 0\}$, the integral is the familiar one from multivariable calculus:
\begin{equation}
\label{eq:calculus2}
\int_{\Gamma} \omega \,\,=\,
\int_{\Gamma} r(x_0,x_1,x_2) \, \Omega_{\mathbb{P}^2} \, = \, \int_{\Gamma} r(1,x_1,x_2) \, {\rm d} x_1 \wedge {\rm d} x_2 .
\end{equation}
Conversely, every rational form $\omega$ on $\mathbb{P}^2$ is uniquely determined by an affine
expression of the form $\tilde{r}(x_1,x_2) \,  {\rm d} x_1 \!\wedge \! {\rm d} x_2$ for $\tilde{r} \in \mathbb{C}(x_1,x_2)$.
To pass from $\C^2 $ to $\PP^2$, one homogenizes by setting
 $x_i \mapsto x_i/x_0$ for $i=1,2$.
The analogous statement holds for $3$-forms on $\C  ^3$ and  $\PP^3$.

\begin{example}[Pentagon] \label{ex:pentagon}
We consider the following rational $2$-form on the plane $\PP^2$:
\begin{equation}
\label{eq:pentagonform}
 \omega \,\, = \,\,\frac{ 65 \, F_{23} F_{45} - 104 \, F_{24}F_{35}}{F_{14}F_{35}F_{24}F_{36}F_{25}} \, \Omega_{\mathbb{P}_2} .
 \end{equation}
The numerator  is a conic through the
four points $E_2, E_3,E_4, E_5 $. Indeed, the two summands generate
the ideal of these points. The coefficients are chosen so that the
conic also passes through the intersection  point $F_{14} \cap F_{36}$. 
The form $\omega$ is the canonical form of the pentagon
formed by the lines in the denominator. That pentagon is shown in 
Figure~\ref{fig:firstpentagon}. The conic in the numerator is the \emph{adjoint curve}. The pentagon, together with its canonical form $\omega$, is our first
example of a {\em positive geometry}. This example illustrates the key property of the canonical form of a
polygon $S$: it has poles of order one along the five boundary lines of $S$. 

The conics $G_1, \ldots, G_6$ in Figure \ref{fig:firstpentagon} can also be obtained as the adjoint curves of curvy pentagons on $X$. The cartoon in Figure \ref{fig:pentagon} shows the conic $G_6$ passing through the points $E_1, E_2, E_3,E_4, E_5$. It appears in the numerator of the canonical form of the gray pentagon.
\end{example}

\begin{figure}[h]
\centering
\includegraphics[width = 0.45\linewidth]{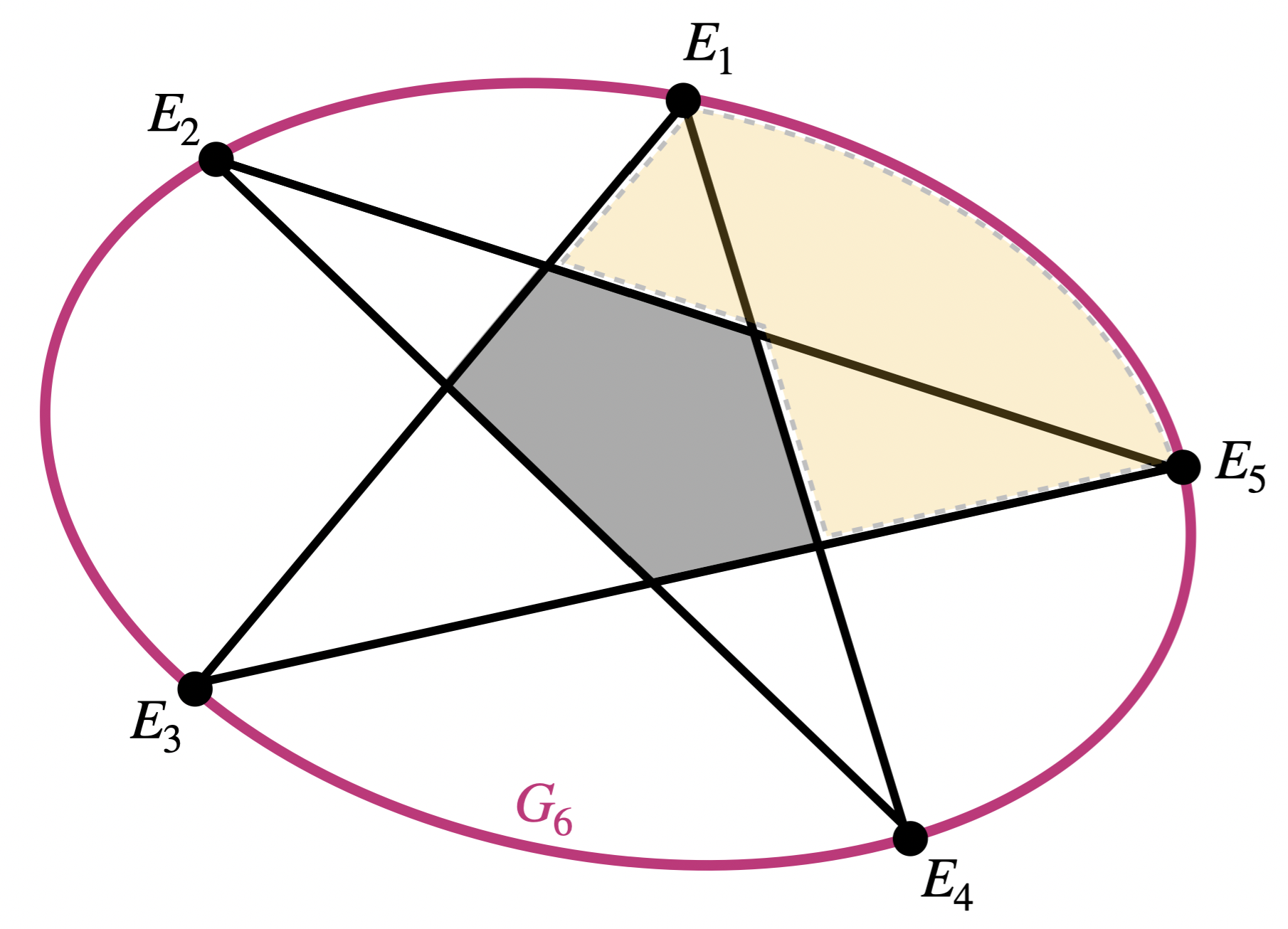}
\caption{The gray shape and the yellow shape are positive geometries on the cubic~surface.}
\label{fig:pentagon}
\end{figure}
We now return to (\ref{eq:formomega3}) and consider a
homogeneous rational function $r(y)$ of degree $-4$.
Let $\Gamma$ be a semialgebraic region in $\PP^3_\R$ which does not intersect the poles of $r(y)$.
We view $\Gamma$ as a homology class, and we view
$\omega$ as a cohomology class.
Their pairing is the integral
\begin{equation}
\label{eq:calculus3}
\int_{\Gamma} \omega \,\,=\,
\int_{\Gamma} r(y_0,y_1,y_2,y_3) \, \Omega_{\mathbb{P}^3} \, = \, \int_{\Gamma} r(1,y_1,y_2,y_3) 
\, {\rm d} y_1 \wedge {\rm d} y_2  \wedge {\rm d} y_3.
\end{equation}

There are two natural constructions for turning the $3$-dimensional objects in (\ref{eq:calculus3})
into $2$-dimensional objects. For the region $\Gamma$ this is achieved
by passing to its {\em boundary} $\partial \Gamma$. For the form $\omega$ this is
achieved by passing to its {\em residue} along an irreducible component of $\partial \Gamma$.

We fix a homogeneous cubic $f(y)$, such as (\ref{eq:cubicsurface}),
that defines a smooth cubic surface $ X = V(f)$ in $\PP^3$.
Let $\omega$ be a rational $3$-form on $\PP^3$ which has
a pole of order one along $X$. For the definition of order see
\cite[Chapter II, \S 6]{Hartshorne}. Under these assumptions, we can write
\begin{equation} \label{eq:decompomega} \omega \,\, =\, \,  \eta \wedge \frac{{\rm d} f}{f} \,+\, \eta',\end{equation}
where $\eta$ is a $2$-form and $\eta'$ is a $3$-form on $\PP^3$, both having no poles along $X$. 
This decomposition gives a rational $2$-form on the cubic surface $X$. By restriction from $\PP^3$, we obtain
\begin{equation}
\label{eq:residue}
 {\rm Res}_{X} \, \omega \, = \, \eta_{|X}.
\end{equation}
This $2$-form is called the  \emph{Poincar\'e residue} of $\omega$ along $X$.
We shall use this residue operator to represent the $2$-forms on $X$
that are needed for the space $\Omega^2_{\rm log}(X \backslash Y)$ in Theorem \ref{thm:130}.
    
Explicitly, we can represent the $3$-form in (\ref{eq:decompomega}) by two homogeneous polynomials
$g(y)$ and $h(y)$ in $y= (y_0,y_1,y_2,y_3)$ whose degrees satisfy ${\rm deg}(h) = {\rm deg}(g) + 1$. Namely, we write
\begin{equation} \label{eq:omegacubicsurface} \omega \, = \, \frac{g(y)}{f(y) h(y)} \, \Omega_{\mathbb{P}^3}. \end{equation}
In the affine coordinates, this equals   (\ref{eq:calculus3}), where
$ r(1,y_1,y_2,y_3) \, = \, \frac{g(1,y_1,y_2,y_3)}{f(1,y_1,y_2,y_3)h(1,y_1,y_2,y_3)} $.

\smallskip

The residue of $\omega$ along $X$ is the restriction to $X$ of the $2$-form on $\mathbb{P}^3$ given locally by
\begin{equation} \label{eq:localexpression}  \frac{g(1,y_1,y_2,y_3)}{  \frac{\partial f(1,y_1,y_2,y_3)}{\partial y_3} \, h(1,y_1,y_2,y_3)}  \, {\rm d}y_1 \wedge {\rm d} y_2 . \end{equation}
The restriction can be written conveniently  as a rational form on $\mathbb{P}^2$ via the parametrization \eqref{eq:blowup}. To compute this rational form, we replace $y_i $ by $ y_i/y_0$ in \eqref{eq:localexpression} and 
substitute \eqref{eq:octamodel}.

Computing the canonical form of a positive geometry may in general be a hard task. For the triangles, quadrilaterals and pentagons on a cubic surface, we now provide a combinatorial recipe in terms of cycles in the Schl\"afli graph. We take advantage of the group action on polygons
described in Proposition \ref{prop:Schlafli}. It suffices to describe only one canonical form per orbit.
All others are obtained by relabeling. We limit ourselves to determining the canonical forms \emph{up to scaling}. In particular, for now, we ignore the orientation of the polygons. 

Each of the $45$ triangles
corresponds to a tritangent plane of $X$ in $\PP^3$.
We use the name of the triangle to label
the corresponding linear form $h$.
For instance, the triangle $S = F_{ij}F_{kl}F_{mn}$ gives the linear form
$h_{(ij)(kl)(mn)}$ in $y_0,y_1,y_2,y_3$. Its canonical form equals
\begin{equation}
\label{eq:triangle1}
 \omega(S) \,\, = \,\, {\rm Res}_X \frac{1}{f \, h_{(ij)(kl)(mn)}} \, \Omega_{\mathbb{P}^3} .
 \end{equation}
We next write $h_{ij} $ for the linear form given by 
 $S = E_iF_{ij}G_j$. Then the canonical form equals
 \begin{equation}
\label{eq:triangle2}
 \omega(S) \,\, = \,\, {\rm Res}_X \frac{1}{f \, h_{ij}} \, \Omega_{\mathbb{P}^3} .
 \end{equation}
One sees from this formula that $\omega(S)$ has simple poles along each boundary line of $S$.

The canonical form of a $4$-cycle is more interesting. We compute the following in \cite{Zenodo}.

\begin{lem} \label{lem:dist2vertex}
In the Schl\"afli graph $\mathcal{S}^{10}_{27}$,
every $4$-cycle $L_1 L_2 L_3 L_4$  has a unique vertex $L$ that has distance two from each $L_i$. 
We call this vertex $L$ the \emph{adjoint vertex} of $L_1L_2L_3L_4$. 
\end{lem}

The name \emph{adjoint vertex} hints at the fact that the distance-two line consists of the zeros of the canonical from. For instance, for the $4$-cycle $S = E_5F_{35}G_5F_{45}$ in Figure \ref{fig:adjointquadrilateral}, 
the adjoint vertex is $L = F_{34}$. The canonical form has simple poles along the cycle, and zeros along $L$: 
\begin{equation} \label{eq:canformquadrilateral} 
\omega(S) \,\,=\,\,  {\rm Res}_X \frac{h_{43}}{f\,h_{53}h_{45}}\,\Omega_{\mathbb{P}^3} \, \,= \,\, {\rm Res}_X \frac{h_{34}}{f\,h_{35}h_{54}}\,\Omega_{\mathbb{P}^3}.
\end{equation}
These equalities are up to scaling. The second equality is seen from the fact that the two forms have the same zeros and poles on $X$. For instance, for the form in the middle, the six candidate poles of $\omega(S)$ are $(E_5 F_{35} G_3) (E_4 F_{45} G_5)$. Two of those are cancelled by the numerator which vanishes on $E_4 F_{34} G_3$.
Here are two more examples of quadrilaterals:
$$ \begin{matrix}
\! F_{36}F_{15}F_{46}E_{6} \,\,{\rm has} \,\,
{\rm Res}_X \frac{h_{24}}{f\,h_{(15)(24)(36)}h_{64}}\,\Omega_{\mathbb{P}^3}
\,\,\,{\rm and}\,\,\, F_{56}F_{24}F_{15}F_{34} \,\,{\rm has}\,\,
{\rm Res}_X \frac{h_{(13)(26)(45)}}{f\,h_{(13)(24)(56)}h_{(15)(26)(34)}}\,\Omega_{\mathbb{P}^3}. 
\end{matrix}
$$

We now turn to pentagons. The Schl\"afli graph $\mathcal{S}^{10}_{27}$
has $4320$ five-cycles $S$ that are exposed in groups of $30$ on the various cubic surfaces $X$.
The pentagon in Figure \ref{fig:pentagon} is $S = F_{14} F_{35} F_{24} F_{36} F_{25}$.
In  (\ref{eq:pentagonform}) we wrote its canonical form $\omega(S)$ as a $2$-form on $\PP^2$.
We next describe the form $\omega(S)$ on the cubic surface $X \subset \PP^3$,
in a $W({\rm E}_6)$-invariant formulation.

We start from the following combinatorial fact about the Schl\"afli graph $\mathcal{S}^{10}_{27}$:
{\em each of the $4320$ potentially exposed $5$-cycles $S$ contains a unique $3$-cycle $H_1$.}
The $3$-cycle $H_1$ has a unique edge $E_1$ that is not an edge of $S$.
For instance, if $S = F_{14} F_{35} F_{24} F_{36} F_{25} $ then
$H_1 = F_{14} F_{36} F_{25}$ and $E_1 = F_{14} F_{36}$. Note that
$E_2 = S \backslash H_1$ is an edge of our $5$-cycle $S$.
That edge is contained in a unique $3$-cycle $H_2$
of the Schl\"afli graph. Let $L$ denote the added vertex.

 In our example, $E_2 =  F_{35} F_{24}$,
$L = F_{16}$ and $H_2 = F_{35} F_{24} F_{16}$.
Let $g = 0$ be the plane in $\PP^3$ that is spanned by the line
$L$ and the point corresponding to $E_1$.
The canonical form~is
\begin{equation}
\label{eq:canonicalpentagon}
 \omega(S) \,\, = \,\, {\rm Res}_X  \,\frac{g}{f \, h_{H_1} \, h_{H_2}}\,  \Omega_{\mathbb{P}^3}. 
 \end{equation}
In our example, $h_{H_1} = h_{(14)(36)(25)}$ and $h_{H_2} = h_{(35)(24)(16)}$. The linear form $g$ vanishes on the strict transform of $F_{16}$ and the image of $F_{14} \cap F_{36}$ under the blow-up map \eqref{eq:blowup}. 
By $W({\rm E}_6)$- symmetry, it suffices to check this formula for 
the special pentagon in Example \ref{ex:pentagon}.

We have constructed $45+1080+4320$ rational forms
on the cubic $X$. These forms have simple poles along the divisor $Y$ given by the $27$ lines.
They span the vector space $\Omega_{\rm log}^2(X \backslash Y)$ of all such forms.
In fact, a much smaller spanning set is given by the
$130 = 10+90+30$ forms that are exposed on our specific surface $X$.
The following lemma proves the number $109$ in Theorem~\ref{thm:130}.
It is gratifying to see how two different approaches 
 yield the same~result.

\begin{lem} \label{lem:109}
The dimension of $\,\Omega_{\rm log}^2(X \backslash Y)$ equals $109$.
\end{lem}

\begin{proof}[Two proofs]
Our first proof follows from \cite[Proposition 4.4]{BD}. 
The divisor $Y$ is normal crossing in $X$, and its intersection complex
is the Schl\"afli graph $\mathcal{S}^{10}_{27}$. This graph is connected. It has $27$ vertices
and $135$ edges.
Hence the genus of $\mathcal{S}^{10}_{27}$  is $  135-27+1 = 109$, as desired.

For the second proof, we use the following isomorphism from  \cite[Proposition 1.13]{BD}:
\begin{equation}
\label{eq:thirdproof}
 \Omega^2_{\log}(X \backslash Y) \,\simeq \,H^0(X, {\cal O}_X(K_X + Y)).  \end{equation}
     We express the divisor $Y$ in terms of the generators $H, E_1, \ldots, E_6$ of ${\rm Pic}(X)$. The divisor classes of
     the $15$ lines $F_{ij}$
      are given by $H-E_i-E_j$. The sum of these  classes equals $15H-5(\sum_{i =1}^6E_i)$. 
      The six lines $G_j$ have divisor class $2H-\sum_{i \neq j} E_i$, which sums to $12H - 5 \sum_{i=1}^6 E_i$. The remaining six lines are $E_1, \ldots, E_6$. The divisor class of $Y$ is therefore given by 
   \begin{equation}
   \label{eq:degreenine} Y \,\, = \,\, 27 \, H \,-\, 9 \,\bigl(\,\sum_{i = 1}^6 E_i\,\bigr) \, = \, -9 \, K_X. 
   \end{equation}
   This equation reflects the easy-to-verify fact that
the radical ideal of  the divisor $Y$ is a complete intersection. Namely,
this ideal in $\C[y_0,y_1,y_2,y_3]$ is generated by the cubic $f$ and a product of $9$ linear forms,
corresponding to pairwise disjoint triangles in $\mathcal{S}^{10}_{27}$.

        The adjunction formula gives $K_X = -h_{|X}$, where $h$ is the hyperplane class in $\mathbb{P}^3$. Hence, ${\cal O}_X(K_X + Y) =  {\cal O}_X(-8 \, K_X) =  {\cal O}_X(8 \, h_{|X})$.
        Its sections are the elements of degree $8$ in
the coordinate ring $\C[X] = \C[y_0,y_1,y_2,y_3]/\langle f \rangle$.
        Their dimension is $\binom{8+3}{3}-\binom{5+3}{3} = 109$.
   \end{proof}

\begin{rmk} \label{rmk:realrank} 
We identify linear relations among the $130$ canonical forms
of polygons on $X$. Each line $L$ gives one linear relation.
To see this, we identify $L$ with a conic $G_i$ in the plane model.
The canonical form of the interior of $G_i$ is zero.  This specifies
a relation among the forms $\omega(S)$, where $S$ ranges
over all regions in the interior of $G_i$. These $27$ linear relations span all
linear relations among the $130$ forms $\omega(S)$.
They are not independent, as they obey six linear relations;
one for each exceptional divisor $E_i$.
This computation implies that our $130$ forms \eqref{eq:triangle1}-\eqref{eq:canonicalpentagon} span the space $\Omega_{\rm log}^2(X \backslash Y)$, which has dimension 
 $130-27+ 6 = 109$. We confirm in \cite{Zenodo} that the 130 forms from our surface \eqref{eq:cubicsurface} span a $109$-dimensional space.
 \end{rmk}

\begin{rmk}
The canonical forms $\omega(S)$ define the
{\em log canonical embedding} of the pair $(X,Y)$. 
Our third proof says that this equals
the degree $8$ Veronese embedding of $X \subset \PP^3$. Its image is a surface in
$\PP^{108}$. That log canonical surface has
degree $192$, and its homogeneous prime ideal is generated by $5586$ quadrics
in $109$ variables. We refer to \cite[Section 5.1]{HPP}.
\end{rmk}

\section{What is a Positive Geometry?}
\label{sec5}

Positive Geometry is an emerging field of mathematics with origins in theoretical physics \cite{ABL,Lam,RST}. It studies the geometry and combinatorics underlying physical observables such as scattering amplitudes and cosmological correlators. Here ``geometry'' is meant in a broad sense, including complex, real and tropical algebraic geometry. The discovery of the amplituhedron by Arkani-Hamed and Trnka was an early milestone \cite{AT}. This was followed by numerous other 
geometries.
The physical principles underlying these objects prescribe several common features. This motivates the development of a unifying mathematical framework. The first step was taken in \cite{ABL}, which gives an axiomatic definition of a \emph{positive geometry}.

The Hodge-theoretic approach in \cite{BD} suggests an alternative definition, first formulated in \cite[Definition 1.13]{Harvard}. 
The two definitions are not quite equivalent, but many relevant examples align with both, as shown in the
lecture notes \cite{Harvard}. Our
 Definition \ref{def:positivearrangement} imposes
 reality constraints on the objects of \cite[Definition~1.13]{Harvard}.  
While the debate about the question in the section title is still ongoing, we
here propose the following workflow for practitioners.

We use the term {\em domain}
for a $d$-dimensional orientable semialgebraic subset of some $\R^n$.

\begin{wfl} \label{workflow}
When a domain $S$ is suspected to be a positive geometry, proceed as follows: 

\vspace{-0.13in}

\begin{enumerate}
\item Find an irreducible $d$-dimensional complex projective variety $X$ and a divisor $Y \subset X$ such that $S$ is embedded in $X_\R$, and its closure in $X_\R$ is a union of \emph{regions} in $X_\R \backslash Y_\R$. 
\vspace{-0.1in}
\item Show that $(X,Y)$ is a \emph{positive arrangement}. \vspace{-0.1in} 
\item Count the regions of $(X,Y)$ and determine the \emph{combinatorial rank}. \vspace{-0.1in}
\item Examine whether the combinatorial rank equals the \emph{real combinatorial rank}. \vspace{-0.1in}
\item Identify the \emph{canonical form} $\omega(S)$ by computing residues and proceeding recursively. \vspace{-0.1in}
\item  Your \emph{positive geometry} is the pair $(S,\omega(S))$. Repeat step 5 for other regions of interest. \vspace{-0.1in}
\end{enumerate}
\end{wfl}
We now turn to the formal set-up and define the italicized objects in Workflow \ref{workflow}.
Consider a $d$-dimensional irreducible complex projective variety $X$ and a divisor $Y \subsetneq X$ such that $X \backslash Y$ is smooth. A \emph{relative $d$-cycle} of the pair $(X,Y)$ is a $d$-dimensional chain in $X$ whose boundary is contained in $Y$. The formal $\mathbb{Q}$-linear combinations of the relative $d$-cycles constitute the $\mathbb{Q}$-vector space $Z_d(X,Y)$. Taking the quotient of $Z_d(X,Y)$
 by the group of boundaries gives the relative homology group $H_d(X,Y)$. The pair $(X,Y)$ is said to have \emph{genus zero} if the mixed Hodge numbers $h^{-p,0}$ of $H_d(X,Y)$ for $p >0$ are all equal to zero \cite[Section 3]{BD}. This technical condition can be checked using the tools  in \cite[Section~3.2]{BD}. 
We reiterate that Voisin's book \cite{Voisin} is an excellent source
for studying mixed Hodge theory.

\begin{defn} \label{def:positivearrangement}
A \emph{positive arrangement} is a pair $(X,Y)$ of an irreducible complex projective variety $X$ and a divisor $Y$ on $X$, which satisfies the following four conditions:
\begin{enumerate}
\item The variety $X$ is defined over $\mathbb{R}$, and $X$ has a smooth real point. \vspace{-0.1in}
\item Each irreducible component of $Y$ is defined over $\mathbb{R}$ and has a smooth real point. \vspace{-0.1in}
\item The complement $X \backslash Y$ of $Y$ in $X$ is contained in $X_{\rm reg}$. Equivalently, ${\rm Sing}(X) \subseteq Y$. \vspace{-0.1in}
\item The pair $(X,Y)$ has genus zero. 
\end{enumerate}
\end{defn}

Condition 1 ensures that the real points of $X$ are Zariski dense in $X$. Condition 2 does the same for each irreducible component of $Y$. Our running example is a positive~arrangement. 

\begin{example} \label{ex:posarr_cubicsurface}
The cubic surface $X$ in Section~\ref{sec2}, together with the divisor $Y$
given by its $27$ lines, is a positive arrangement.
 The pair $(X,Y)$ has genus zero by \cite[Corollary 3.14]{BD}.
 \end{example}

\emph{Positive geometries} are obtained from (oriented) regions in a positive arrangement. 

\begin{defn} \label{def:positivegeometry}
Let $(X,Y)$ be a positive arrangement.
A \emph{region} of $(X,Y)$ is the closure in $X_\R$ of an orientable connected component of the complement $X_\mathbb{R} \backslash Y_\mathbb{R}$. A \emph{positive geometry} of $(X,Y)$ is any relative cycle of dimension $d = \dim X$ that 
can be written as a formal $\mathbb{Q}$-linear combination of regions.  In particular, every region
is a positive geometry of $(X,Y)$.
\end{defn}

The positive geometries of $(X,Y)$ form a $\mathbb{Q}$-vector subspace of $Z_d(X,Y)$, where $d = \dim X$. 
That subspace is spanned by all  regions.
The regions are positive geometries.
  So are their unions; the union of two regions corresponds to the sum of their cycles in $Z_d(X,Y)$.

\begin{example} \label{ex:regionscubicsurface}
Each of the $130$ polygons $S$ from Theorem \ref{thm:130} is a region of the positive arrangement $(X,Y)$ from Example \ref{ex:posarr_cubicsurface}. Unions of these polygons are also positive geometries, see for instance the yellow shape in Figure \ref{fig:pentagon}. We now equip these with a differential~form. 
\end{example}

For each genus zero pair $(X,Y)$, Brown and Dupont  \cite[Definition 2.6]{BD} introduce a map
\begin{equation} \label{eq:browndupont}
    \omega: H_d(X,Y) \, \longrightarrow \, \Omega_{\rm log}^d(X\backslash Y), \quad \text{where} \quad d  = \dim X.
\end{equation}
This is a linear map of $\Q$-vector spaces. The map $\omega$
 takes a relative homology class $[\sigma]$ to  a logarithmic differential form $\omega([\sigma])$.
 The $\Q$-vector space $\Omega_{\rm log}^d(X\backslash Y)$ is contained
  in the space of rational forms on $X$ with at most simple poles along the irreducible components of $Y$.

\begin{defn} \label{def:canonicalform}
Let $(X,Y)$ be a positive arrangement and let $\sigma \in Z_d(X,Y)$ be a positive geometry of $(X,Y)$. The \emph{canonical form} $\omega(\sigma)$ is its image under the Brown--Dupont map \eqref{eq:browndupont}. That is, $\omega(\sigma) = \omega([\sigma])$, where $[\sigma]$ is the class of $\sigma$ in
the relative homology $H_d(X,Y)$.  
\end{defn}

\begin{example} \label{ex:canformlinesegment}
Consider the positive arrangement $(X,Y)$ where $X = \mathbb{P}^1$ and the divisor is $Y = \{a_1,a_2, \ldots, a_r\} \subset \PP_\R^1$. We assume 
$a_1 < a_2 < \cdots < a_r$ in $\R \subset \PP_\R^1$,
 and the line segment $ \sigma_{i,j} = [a_i,a_{j}]$ with $i < j$ is oriented from right to left.
 The canonical form of this region is
 \[ \omega(\sigma_{i,j}) \, = \, \frac{a_{j}-a_i}{(x-a_i)(a_{j}-x)} \, {\rm d}x. \]
Among these $\binom{r}{2}$ forms, only $r-1$ are linearly independent.
To get a basis, set $j= i+1$.
 This example is essential for the recursive construction of
     canonical forms in higher dimensions.
\end{example}

Our definition of positive geometry is easy to check in
many cases. 
The existence and uniqueness of the canonical form $\omega(\sigma)$ for
every region $\sigma$ is ensured by Brown and Dupont~\cite{BD}.
However, there is no free lunch. We are left with finding
the form, and with proving
 that it actually is the image of $[\sigma]$ under the map \eqref{eq:browndupont}. 
To do this, we can use \cite[Proposition~2.15]{BD}.
The recursive approach from \cite{ABL} is somewhat hidden but still present.

\begin{example} \label{ex:canformpentagon}
Let $X = \mathbb{P}^2$ and let $Y = Y_1 \cup \cdots \cup Y_5$ be the union of the five edge lines of the pentagon $\sigma$ in Figure \ref{fig:pentagon}. We claim that its canonical form $\omega(\sigma)$
is given up to scaling by
\begin{equation} \label{eq:canformpentagon}
\tilde\omega \, := \, \frac{C}{Y_1 Y_2 Y_3 Y_4 Y_5} \, \, \Omega_{\mathbb{P}^2}, \end{equation}
where $C$ is the pink conic in Figure \ref{fig:pentagon}. We will now prove this. 
For each $i = 1, \ldots, 5$, let $Z_i = Y_i \cap (\cup_{j \neq i} Y_j)$ be the set of four intersection points on $Y_i$ obtained from the other lines. 
The residue along $Y_i$ of a logarithmic $2$-form on $X \backslash Y$ gives a logarithmic $1$-form on $Y_i \backslash Z_i$:
\[ {\rm Res}_{Y_i} \, : \, \Omega^2_{\rm log}(X\backslash Y) \, \longrightarrow \, \Omega^1_{\rm log}(Y_i \backslash Z_i). \]
The edge $e_i = \sigma \cap Y_i$ is a positive geometry for
the positive arrangement  $(Y_i,Z_i)$. By Example \ref{ex:canformlinesegment},
 its canonical form is the unique form on $Y_i \simeq \mathbb{P}^1$ with only simple poles at the vertices of $e_i$.  
We claim that the image of $\tilde \omega$ under the residue map ${\rm Res}_{Y_i}$ has this property. 
The only possible poles of ${\rm Res}_{Y_i} \tilde \omega$ are the points $Z_i$. Since the conic $C$ passes through the points in $Z_i$ which are not vertices of $e_i$, $\tilde \omega$ indeed only has poles at these vertices. Hence, there exists a nonzero scalar $\gamma$ such that ${\rm Res}_{Y_i} \gamma \, \tilde \omega = \omega(e_i)$. Moreover,  up to scalar multiple, $\tilde \omega$ is the only form in $\Omega^2_{\rm log}(X\backslash Y)$ with this property. 
By \cite[Proposition 2.15]{BD}, this implies $\omega(\sigma) = \gamma \cdot \tilde \omega$. 
\end{example}

Readers who are familiar with the original definition of positive geometries in \cite{ABL} recognize the recursive nature of that definition in Example \ref{ex:canformpentagon}. In our experience, the recursive properties of canonical forms, originally postulated in \cite{ABL}, still yield the most effective tools for evaluating the Brown--Dupont map \eqref{eq:browndupont}, and for testing or proving the correctness of a candidate canonical form. The mixed Hodge theory framework from \cite{BD} 
led us to a more concise definition, and to an elegant way 
of showing existence and uniqueness of such a~form. 

\begin{prop} \label{prop:130canforms}
The $10$ triangles, $90$ quadrilaterals and $30$ pentagons on the cubic surface $X$ from Theorem \ref{thm:130} are positive geometries of $(X,Y)$. Up to scaling, their canonical forms are given by the expressions in \eqref{eq:triangle1}, \eqref{eq:canformquadrilateral} and \eqref{eq:canformpentagon}, by acting with the Weyl group $W({\rm E}_6)$.
\end{prop}

\begin{proof}
Let $S \subset X$ be a curvy $k$-gon where $k = 3, 4, 5$, and let $\omega$ be the corresponding form postulated in \eqref{eq:triangle1}, \eqref{eq:canformquadrilateral} and \eqref{eq:canformpentagon}. Let $L_i$ be any of the lines constituting the boundary $L_1\cup \cdots \cup L_k$ of $S$. By construction, the residue ${\rm Res}_{L_i} \, \omega$ has only simple poles at the two vertices of $S$ lying in $L_i$. Therefore, by Example \ref{ex:canformlinesegment}, the residue is the canonical form of the line segment $\sigma = S \cap L_i$ as a positive geometry of the positive arrangement $(L_i, L_i \cap (\cup_{j \neq i} L_j))$. We now apply \cite[Proposition 2.15]{BD} as in Example \ref{ex:canformpentagon} to conclude the proof. 
\end{proof}

Brown and Dupont \cite[Section 4]{BD} define the \emph{combinatorial rank} ${\rm cr}(X,Y)$ 
of a genus zero pair as the dimension of  the vector space $\,\Omega_{\rm log}^d(X\backslash Y)$.
We propose the following real analog.

\begin{defn} \label{def:realcombinatorialrank}
Fix a positive arrangement $(X,Y)$.
Its \emph{real combinatorial rank} ${\rm cr}_{\mathbb{R}}(X,Y)$ is the dimension of the $\mathbb{Q}$-vector space spanned by the canonical forms $\omega(\sigma)$, where $\sigma$ ranges over all regions $\sigma$. 
We have ${\rm cr}_{\mathbb{R}}(X,Y) \leq {\rm cr}(X,Y)$. Equality holds
whenever the regions generate the relative homology, i.e.~the natural map $H_d(X_{\mathbb{R}}, Y_{\mathbb{R}}) \rightarrow H_d(X,Y)$ is surjective. 
\end{defn}

Definition \ref{def:realcombinatorialrank} was first suggested to us by Mario Kummer, as was the next example.

\begin{example}
Let $X = \mathbb{P}^2$ and $Y = C_1 \cup C_2$ the union of two general conics.
 The combinatorial rank is ${\rm cr}(X,Y) = 3$,
  by \cite[Proposition 4.4]{BD}. However, the real combinatorial rank depends on
  the reality of the four points in $C_1 \cap C_2$. 
  If none of them is real then   ${\rm cr}_{\mathbb{R}}(X,Y) = 0$,
  if two are real then  ${\rm cr}_{\mathbb{R}}(X,Y) = 1$,
  and if all  four are real then  ${\rm cr}_{\mathbb{R}}(X,Y) = 3$.
  \end{example}

This paper revolves around geometries whose
combinatorial rank equals the real rank. Notably, if $(X,Y)$ is a cubic surface with
 $27$ real lines then we learn from Remark \ref{rmk:realrank} that
 \begin{equation}
 \label{eq:cr109}
 {\rm cr}_{\mathbb{R}}(X,Y) \,=\, {\rm cr}(X,Y) \,=\, 109. 
 \end{equation}
 Equality also holds for the pairs in the next section,
 and for the moduli space  in Theorem~\ref{thm:432}.
 
 Before ending with two more examples, we note that we have now completed all steps in Workflow \ref{workflow} for any of the curvy polygons contained in the cubic surface $X$ from~Section \ref{sec2}.
 
\begin{example} \label{ex:cr17}
The right image of Figure \ref{fig:endler1} represents the pair $(X,Y')$, where $X$ is the cubic surface and $Y'$ is the union
of $12$ lines on~$X$, namely the $12$ lines in
 the double-six  (\ref{eq:doublesix2}).
 The pair $(X,Y')$ is a positive arrangement.
Each of its $25$ regions is a positive geometry.
The tropicalization of $X \backslash Y'$ is the complete bipartite graph $K_{6,6}$ with 
six edges removed. Hence, by \cite[Proposition 4.4]{BD},
the (real) combinatorial rank is ${\rm cr}_\R(X,Y') = {\rm cr}(X,Y') =  17$.
\end{example}

\begin{example}
Let $X = \mathbb{P}^2$ and let $Y = C_1 \cup \cdots \cup C_r$ be a union of real rational curves $C_i \subset \mathbb{P}^2$. Then $(X,Y)$ is a positive arrangement. It has genus zero by \cite[Proposition 3.22]{BD}. An example with 
$r = 21 = 15+6$
is shown in Figure \ref{fig:firstpentagon}.  
Positive geometries in $(\PP^2,Y)$ have been studied in the literature under the names
{\em polypols} and {\em polycons} \cite{bruser}.
A remarkable positive geometry with $r=3$ conics was recently constructed
by Br\"user; see \cite[Figure~3]{bruser}.
\end{example}

\section{Three-Dimensional Geometries }
\label{sec6}

In this section we study positive arrangements of dimension three. We start from $(\PP^3,Y)$ where $Y$ is a smooth or singular cubic surface $V(f)$ together 
with one or more planes in $\PP^3$. 
{\em Cayley's cubic surface} $V(f)$ in Figure \ref{fig:cayley} has four isolated singular points. It is given by
\begin{equation}
\label{eq:cayley}
 f (y) \,\, = \,\,\, y_0 y_1 y_2 + y_0 y_1 y_3 + y_0 y_2 y_3 + y_1 y_2 y_3 . 
\end{equation} 
This example was discussed in \cite[Section 5.4]{ABL}, where
$Y = V(y_0 f)$. So, the divisor consists of the cubic surface
together with the plane spanned by three of its four singular points.
The pair $(\PP^3,Y)$ is a positive arrangement. A canonical form was given in 
\cite[equation (5.34)]{ABL}.

Thereafter, we present an example recently discovered by
Koefler, Pavlov and Sinn \cite{KPS}.  The context of their
article is the question whether the amplituhedron is a positive geometry. This is of considerable interest to physicists.
For the pair $(\PP^3,Y)$ studied in \cite[Section 6]{KPS}, the divisor $Y$ is a smooth cubic surface
together with five distinguished planes (Figure \ref{fig:KPS}).

\begin{figure}[h]
$$ 
\includegraphics[width =0.45\linewidth]{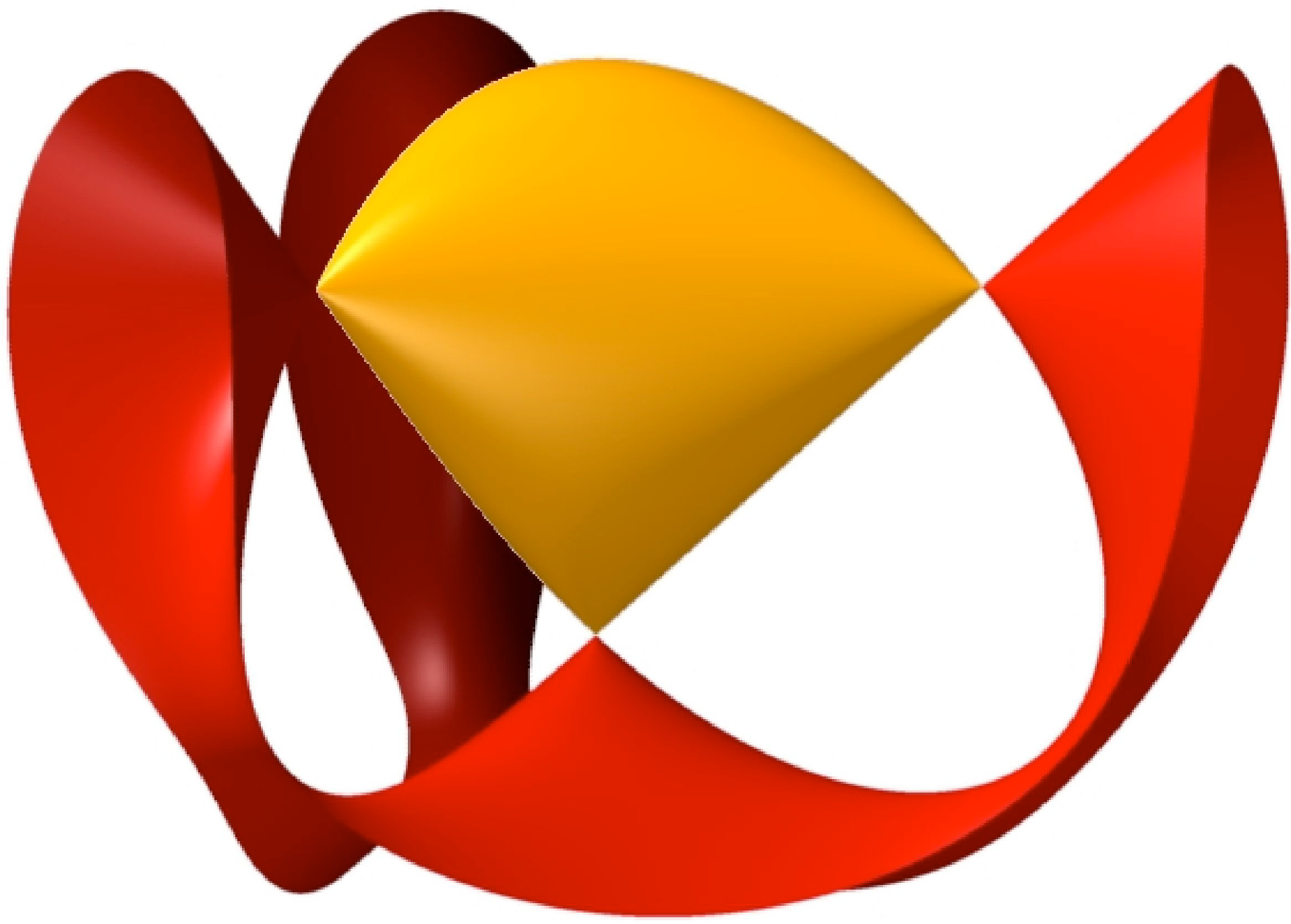} 
\vspace{-0.07in}
$$
\caption{Cayley's cubic surface is the boundary of a positive geometry in $\PP^3$.}
\label{fig:cayley}
\end{figure}

We begin with  the Cayley cubic (\ref{eq:cayley}), taken with
the boundary of its inscribed tetrahedron.
The picture in Figure \ref{fig:cayley} is
the logo of the Nonlinear Algebra group at MPI-MiS~Leipzig.

\begin{prop} \label{prop:41}
Let $X = \PP^3$ and $\,Y = V(y_0y_1y_2y_3 f)$. Then
$(X,Y)$ is a positive arrangement with $15$ regions.
The canonical forms of these regions  span the $4$-dimensional space
\begin{equation}
\label{eq:4dOmega}
 \Omega_{\rm log}^3(X\backslash Y) \,\,=\,\,
\biggl\{ \frac{\alpha y_0 y_1 y_2 + \beta y_0 y_1 y_3 + \gamma y_0 y_2 y_3 + \delta y_1 y_2 y_3}{y_0y_1y_2y_3\, f(y)}
 \, \Omega_{\mathbb{P}^3} \,:\,\alpha,\beta,\gamma,\delta \in \Q \,\biggr\}.
\end{equation}
Moreover, the real combinatorial rank of $(X,Y)$ equals the combinatorial rank, which is four. 
\end{prop}

\begin{proof}
We first show that $(X,Y)$ has genus zero. Projective space itself has $g(X) = g(\mathbb{P}^3) = 0$.
Setting $Y_i = V(y_i)$ and $Y_f = V(f)$,
the inequality from \cite[Corollary 3.13]{BD} reads 
\[ g(X,Y) \,\, \leq \, \, \sum_{i = 0}^3 g  (Y_i, Y_i \cap \overline{Y \backslash Y_i} ) \,  +  \, g(Y_f, Y_f \cap \overline{Y \backslash Y_f} ).  \]
The first sum is zero by \cite[Proposition 3.26(1)]{BD}, since $Y_i \cap \overline{Y \backslash Y_i}$ is a union of three lines in $Y_i \simeq \mathbb{P}^2$. But, we also have
 $g(Y_f, Y_f \cap \overline{Y \backslash Y_f} ) = 0$, by \cite[Corollary 3.13 and Theorem 3.23]{BD}. 

The four coordinate planes $y_i$ divide $\PP^3_\R$ into eight orthants.
The cubic surface $V(f)$ intersects seven of the eight orthants, and it divides them
into two regions each. The positive orthant $\PP^3_{>0}$ is disjoint from the
surface $V(f)$. Hence it is also a region of $(X,Y)$. Namely, it is
 the interior of the inscribed tetrahedron, which is the convex 
hull of the four singular points, seen in Figure \ref{fig:cayley}.
In total, we conclude that $(X,Y)$ has $15$ regions $S$ in $\PP^3_\R$.

Each canonical form $\omega(S)$ is an expression
$\frac{c(y)}{y_0y_1y_2y_3\, f(y)} \Omega_{\mathbb{P}^3}$,
where $c(y)$ is a cubic polynomial. The logarithmic forms generate a subspace of this $20$-dimensional vector space. We claim that ${\rm cr}(\mathbb{P}^3,Y) = \dim \Omega^3_{\log}(\mathbb{P}^3 \backslash Y) = 4$. 
To show this, consider the Cremona transformation $(y_0:y_1:y_2:y_3) \mapsto (\frac{1}{y_0}: \frac{1}{y_1}:\frac{1}{y_2}:\frac{1}{y_3})$.
This identifies $\mathbb{P}^3 \backslash Y$ with the complement of five general planes in $\PP^3$. Such an arrangement has combinatorial rank $4$, see \cite[Section~6.1]{BD}. 
 
We now deduce the explicit description of $\Omega^3_{\log}(\mathbb{P}^3 \backslash Y)$ in the proposition. The residue of a logarithmic form on $X \backslash Y$ along $y_0 = 0$ is a logarithmic form on $V(y_0) \simeq \mathbb{P}^2$. In particular, it has simple poles along $V(y_0) \cap V(y_1y_2y_3 \, f)$. One checks that, for this to hold, the numerator $c(y)$ must vanish along the lines $y_0 = y_i = 0$ for $ i = 1, 2, 3$. Taking residues along the other coordinate planes, we find that $c(y)$ vanishes on the six coordinate lines. This implies that
$c(y)$ is a linear combination of the four squarefree monomials $y_i y_j y_k$. Since ${\rm cr}(X,Y) = 4$, each choice of $c(y) = \alpha y_0 y_1 y_2 + \beta y_0 y_1 y_3 + \gamma y_0 y_2 y_3 + \delta y_1 y_2 y_3$ gives a logarithmic $3$-form. 

Five of the $15$ regions form the interior of the yellow {\em elliptope} in Figure~\ref{fig:cayley}.
Their canonical forms span a $4$-dimensional space. Indeed, they are given by the following coefficients in (\ref{eq:4dOmega}): 
$$ (\alpha,\beta,\gamma,\delta)\, \in \, \bigl\{
 (1,0,0,0), (0,1,0,0), (0,0,1,0),  (0,0,0,1), (-1,-1,-1,-1) \bigr\}. $$
This can be verified using \cite[Proposition 2.15]{BD}, like we did in Example \ref{ex:canformlinesegment}. 
We have now shown that ${\rm cr}_{\mathbb{R}}(X,Y) \,=\, {\rm cr}(X,Y) \,=\, 4$, which
 completes the proof of Proposition \ref{prop:41}.
 The statement in the caption of Figure \ref{fig:cayley} is true because we allow
   $ (\alpha,\beta,\gamma,\delta) = (0,0,0,0)$.
\end{proof}

We now turn to the positive geometry defined by Koefler, Pavlov and Sinn in \cite[Section~6]{KPS}.
Here $f(y)$ is replaced by a general smooth cubic 
with $27$ real lines. For instance, take $f(y)$ to be the polynomial in (\ref{eq:cubicsurface}).
Let $Q  = L_1 L_2 L_3 L_4$ be any of the $90$ curvy quadrilaterals (\ref{eq:quadrilaterals}) on the
surface $V(f)$. Let $q$ be a point in $\PP^3_\R \backslash V(f)$ that is near $Q$,
 so that the each triangle $q L_i$ intersects $V(f)$ only in $L_i$.
Let $P_i$ denote the plane spanned by that triangle for $i=1,2,3,4$.
Finally, let $P_5$ be a general plane that separates $q$ from $Q$.
In particular, $P_5 \cap V(f)$ is a smooth plane cubic curve that is
disjoint from the curvy quadrilateral $Q$.

We consider the pair $(X,Y)$ where $X = \PP^3$ and $Y = V(f) \cup P_1 \cup P_2 \cup P_3 \cup P_4 \cup P_5$.
The six components of the divisor $Y$ bound a region in $X$ that is combinatorially a cube.
We call this a {\em KPS-cube}, after \cite{KPS}, if it satisfies the geometric properties above. 
The KPS-cube has six facets: one is the curvy square $Q$ on the cubic surface,
and the other five are linear. Figure~\ref{fig:KPS} shows two views of a KPS-cube whose nonlinear facet is the quadrilateral in Figure~\ref{fig:adjointquadrilateral}.
\begin{figure}
\centering
\includegraphics[height = 7.1cm]{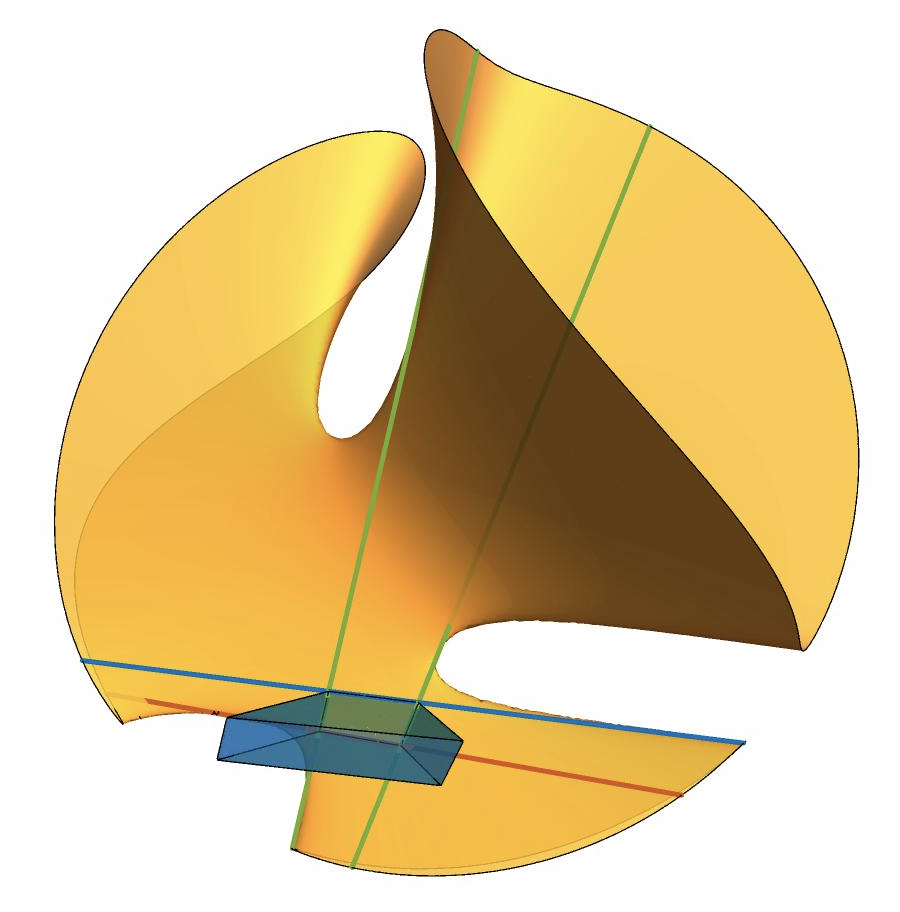}\, \, \, 
\includegraphics[height = 6.9cm]{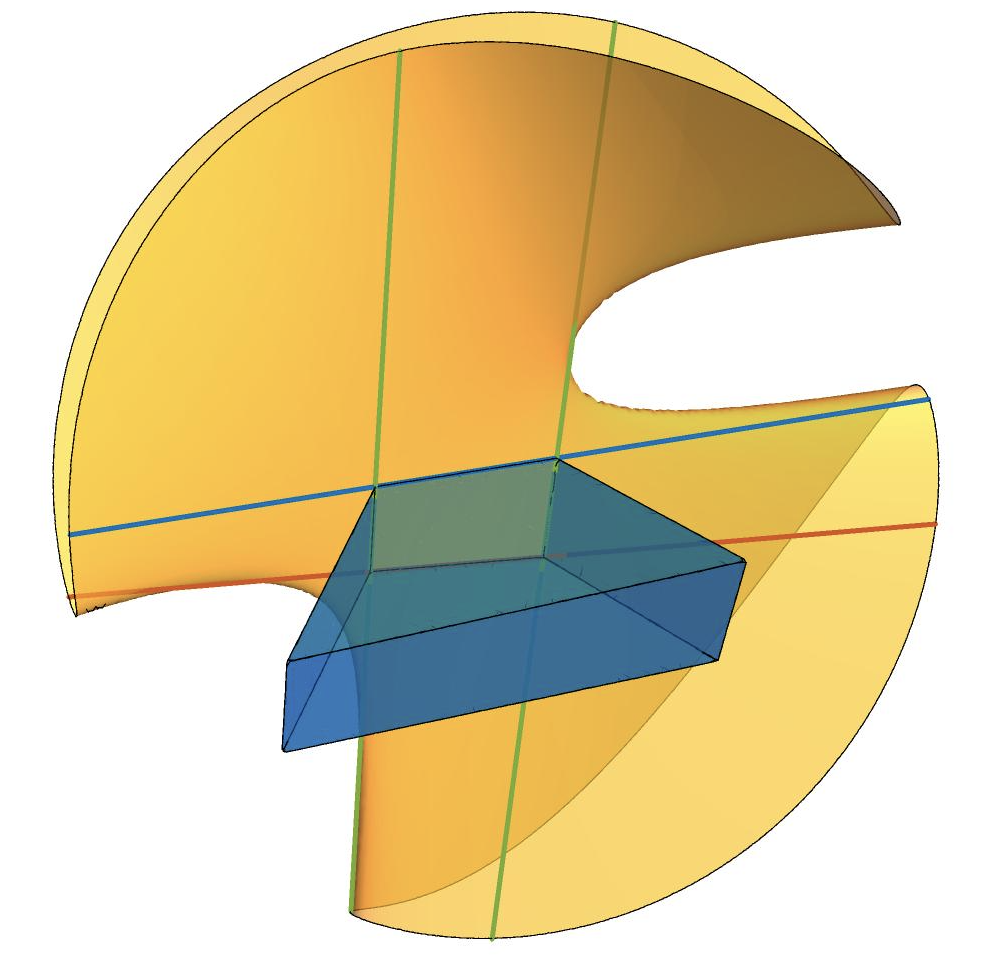}
\caption{A KPS-cube uses a smooth cubic surface for one of its facets.}
\label{fig:KPS}
\end{figure}

\begin{prop} \label{prop:KPS}
Given any quadrilateral $Q$ on a real cubic surface with $27$ real lines,
there exists a KPS-cube with facet $Q$. The corresponding pair
$(X = \mathbb{P}^3,Y)$ has genus one, so it is not a positive arrangement. 
In particular, the KPS-cube is not a positive geometry of~$(X,Y)$.
\end{prop}

\begin{proof}[Sketch of proof]
The first sentence says that there is nothing special about the numerical choices
in \cite[Section 6]{KPS}. We can always find planes $P_1,\ldots,P_5$ that satisfy
the requirements. The key point is that $P_5 \cap V(f)$ is a smooth curve
of genus one, which is disjoint from the real part of the cubic surface $V(f)$.
That genus one contributes $1$ to the mixed Hodge numbers $h^{-p,0}$
of $H_3(X,Y)$. In fact, it is the only non-zero contribution, and we conclude that
$(X,Y)$ has genus one. 
Hence the KPS-cube does not satisfy
Definition \ref{def:positivegeometry} for the given $(X,Y)$.
\end{proof}

The KPS-cube matters for physics because it is a toy model for the amplituhedron \cite{AT}. It was shown in
\cite[Sections 4-5]{KPS} that amplituhedra generally fail to be positive
geometries for their natural boundary arrangements.
   However, by
 \cite[Remark 6.1]{KPS}, the situation can be remedied
by blowing up  $X = \PP^3_\R$ along the 
elliptic curve $E = P_5 \cap V(f)$. Let $X'$ be this blow-up
and let $Y'$ be the modified boundary, which includes the
exceptional divisor $E'$ over $E$. Then $(X',Y')$ still
has genus one, because $X \backslash Y = X' \backslash Y'$.
We now define $Y''$ to be $Y'$ with $E'$ removed.
This changes the geometry considerably.
 The new pair
$(X',Y'')$ has genus zero. It is a positive arrangement.
Moreover, the KPS-cube is entirely unaffected by
the blow-up because it is disjoint from $E$. Therefore, 
it is still a region of $(X',Y'')$. And, in this embedding,
the KPS-cube is now a positive geometry, in the
sense of Definition \ref{def:positivegeometry}. 

It is currently unknown whether a  similar birational transformation 
exists for the $m=2$ amplituhedra in \cite[Section 4]{KPS}. This would turn them
into positive geometries, and hence provide an affirmative answer
to a longstanding open question.
We wish to emphasize that blowing up and removing loci from $Y'$
significantly alters the ambient space. 
These modifications change the combinatorial rank, which we find undesirable.
This underscores the importance of 
choosing a fixed positive arrangement $(X,Y)$ in point 1 of Workflow \ref{workflow}.

\smallskip

We close this section with
another $3$-dimensional positive arrangement
which is ubiquitous in geometry and physics,
namely the moduli space $\mathcal{M}_{0,6}$
with its Deligne-Mumford boundary.
To construct this geometry, we start from the cubic threefold $X = V(g)$ in $\PP^4$ defined by 
$$ g(y) \,=\, y_0 y_1 y_2 + y_1 y_2 y_3 + y_2 y_3 y_4 + y_3 y_4 y_0 + y_4 y_0 y_1. $$
Note the similarity with (\ref{eq:cayley}).
This variety is known as the {\em Segre~cubic} in algebraic geometry.
A very nice introduction appears in Dolgachev's article \cite[Section 2]{Dol}.
The threefold $X$ has $10$ nodal singularities. This is the maximal number of isolated singularities among
irreducible cubics in $\PP^4$, and  $X$ is the unique one up to projective transformation. The Segre cubic $X$ contains
$15$ planes. These form a $(15_4,10_6)$ configuration with the nodes \cite[Proposition~2.2]{Dol}.

\begin{prop} \label{prop:segrecubic} Let $Y$ be the union of the $15$ planes on the Segre cubic $X$.
Then $(X,Y)$ is~a positive arrangement, with
${\rm cr}(X,Y) = {\rm cr}_\R(X,Y) = 24$. Its
$60$ regions are positive~geometries. 
\end{prop}

\begin{proof}
Let $X'$ be the blow-up of $X$ at the ten nodes,
and let $Y'$ be the total transform of $Y$. Then $X'$
is a smooth projective $3$-fold, and 
$Y' $ is a normal crossing divisor with $25 = 15+10$ irreducible components.
We learn from \cite[equation (2.2)]{Dol} that $X'$ coincides with the moduli
space $\overline{\mathcal{M}}_{0,6}$ of stable rational curves with six marked points.
Its boundary $Y' = \overline{\mathcal{M}}_{0,6} \backslash \mathcal{M}_{0,6}$
consists of $25$ prime divisors corresponding to stable curves with two components. There are $15$ divisors for which two of the six points lie on one of the components, and $10$ for which each component has three marked points. These are
the $25$ vertices of the dual complex $\Delta(Y')$ in \cite[Definition 4.3]{BD}.
This is a $2$-dimensional simplicial complex with $105$ edges
and $105$ triangles, known as the {\em space of phylogenetic trees} with six leaves.
The top homology of $\Delta(Y')$ has rank $24$, and this is the
combinatorial rank of $(X,Y)$, by \cite[Proposition 4.4]{BD}.

The arrangement complement $X'_\R \backslash Y'_\R$ consists of $60$ open regions, each
combinatorially equivalent to a curvy associahedron of dimension $3$,
 with $14$ vertices, $21$ edges and $9$ facets.
 These positive geometries are known as {\em worldsheet associahedra} in particle physics;
 see \cite[Section 6]{ABHY}.
Their $60$ canonical forms span the $24$-dimensional space
$\Omega_{\rm log}^3(X \backslash Y)$.
Hence, all relative homology classes have real representatives, and we conclude that
$ {\rm cr}_\R(X,Y) = 24$.
\end{proof}

\section{Integrals in Physics}
\label{sec7}

Every positive geometry $\sigma$ has a canonical differential form $\omega(\sigma)$, which encodes
 the recursive structure of the boundary of $\sigma$ via iterated residue operations. 
 This canonical form appears in step 5 of Workflow \ref{workflow}.
 In particular, $\omega(\sigma)$ has poles along $\partial \sigma \subset Y$. Given such a pair of a cycle $\sigma$ and a cocycle $\omega(\sigma)$, it is natural to try to construct \emph{periods} by pairing them. 
However, the integral $\int_\sigma \omega(\sigma)$ diverges, because
$\omega(\sigma)$ has poles along the boundary of $\sigma$.

\begin{example} \label{ex:beta1}
Consider the positive arrangement $X = \mathbb{P}^1$, $Y = \{0, 1, \infty \}$ and the region $\sigma = [0, 1]$. By Example \ref{ex:canformlinesegment}, the canonical form equals $\omega(\sigma) = \frac{{\rm d}x}{x(1-x)}$. Integrating on $\sigma$ gives 
\[ \int_{\sigma} \omega(\sigma) \,\, = \,\, \int_{0}^1 \frac{{\rm d} x }{x} + \int_{0}^1 \frac{{\rm d} x }{1-x} 
\,\, = \,\, [\log x]^{1}_0 \,+\, [-\log (1-x) ]^{1}_0 \,\, = \,\, \infty.\]
\end{example}

To remedy this divergence, we introduce regulated versions of the integral. We define
\begin{equation} \label{eq:stringintegral}
{\cal I}_\sigma (s) \, = \, \int_{\sigma} \phi_1^{ s_1} \cdots \phi_m^{\, s_m} \, \omega(\sigma), \end{equation}
where $s = (s_1, \ldots, s_m)$ are complex parameters and $\phi_1, \ldots, \phi_m$ are rational functions 
on $X$ that vanish along the boundary components of $\sigma$. We shall now make this more precise. 

  Let $(X,Y)$ be a positive arrangement of dimension $d$  and let $\sigma \in Z_d(X,Y)$ be one of its 
  regions. Recall that $\sigma$ is closed.
  We assume that $X$ is normal, and $Y = Y_1 \cup \cdots \cup Y_r$ where the $Y_i$ are distinct prime divisors on $X$. 
  We may further assume that $Y_1, \ldots, Y_m$ intersect $\sigma$, and 
  that $Y_{m+1}, \ldots, Y_r$ are disjoint from $\sigma$. The  region $\sigma$ is contained in the open subset
\begin{equation}
\label{eq:Usigma}
 {\cal U}_\sigma \, = \, X \,\backslash \!\bigcup_{i= m+1}^r Y_i.
  \end{equation}
Suppose that there exist rational functions $\phi_i \in \mathbb{R}(X)$, for $i = 1, \ldots, m$,  which satisfy
\begin{equation} \label{eq:phi}
\phi_i(p) \geq 0 \, \, \text{ for each }  \, \, p \in \sigma \quad \text{and} \quad  
{\rm div}(\phi_i)_{|{\cal U}_\sigma} = Y_i \cap {\cal U}_\sigma.
\end{equation}
These $m$ rational functions specify the integrand 
 in (\ref{eq:stringintegral}) which regulates the integral
 $\int_\sigma \omega(\sigma)$.

\begin{example} \label{ex:beta2}
We continue Example \ref{ex:beta1}, where $m = 2$ and $r = 3$. We set $Y_1 = \{0\}, Y_2 = \{1\}, Y_3 = \{\infty\}$. The open subset ${\cal U}_\sigma$ is $\mathbb{P}^1 \backslash Y_3 \simeq \mathbb{C}$. The functions $\phi_1 = x, \phi_2 = 1-x$ satisfy the conditions \eqref{eq:phi}. Notice that these are rational functions on $\mathbb{P}^1$ with poles at $Y_3$. If the real part of $s_1, s_2$ is positive, then the integral \eqref{eq:stringintegral} converges. It equals the \emph{Beta function}
\[ {\cal I}_\sigma(s_1,s_2) \, =\, B(s_1,s_2) \, = \, \int_{0}^1 x^{s_1} (1-x)^{s_2} \frac{{\rm d} x}{x(1-x)}.\]
\end{example}

Not every positive arrangement admits functions $\phi_1, \ldots, \phi_m$ satisfying \eqref{eq:phi}. No such functions exist for $X = \mathbb{P}^1$, $Y = \{0, 1\}$ and $\sigma = [0,1]$. However, in all examples we are aware of, this can be fixed by enlarging $Y$. When the  $\phi_i$ exist, they are rarely unique.
  To keep notation simple, we chose to not make the dependence on $\phi$ explicit in the notation ${\cal I}_\sigma(s)$. 

Taking cues from string theory \cite{AHL}, we now multiply the exponents $s_i$ by the \emph{inverse string tension} $\alpha' >0$. Concretely, the \emph{stringy integral} associated with $\sigma$ and $\phi = (\phi_1, \ldots, \phi_m)$~is 
\begin{equation}
\label{eq:takingcues}
 {\cal I}_{\sigma, \alpha'} (s) 
\,\, = \,\, (\alpha')^d \cdot {\cal I}_\sigma(\alpha' \, s) 
\,\, = \,\, (\alpha')^d \, \int_\sigma \phi_1^{\alpha' \, s_1} \cdots \phi_m^{\alpha' \, s_m} \, \omega(\sigma). 
\end{equation}
The \emph{field theory limit} of this stringy integral is
the function $r_\sigma(s) \, = \,  \lim_{\alpha' \rightarrow 0} {\cal I}_{\sigma,\alpha'}(s)$. The factor $(\alpha')^d$ compensates for the divergence of the unregularized integral $\int_\sigma \omega(\sigma)$ at $\alpha' = 0$.

\begin{example} \label{ex:beta3}
Suppose that ${\rm Re}(s_i) > 0$ and $\alpha' > 0$. We expand ${\cal I}_{\sigma}(\alpha' \, s)$ around $\alpha' = 0$: 
\[ {\cal I}_{\sigma}(\alpha' \, s) \, = \, \int_{0}^1 x^{\alpha' \, s_1} (1-x)^{\alpha' \, s_2} \frac{{\rm d} x}{x(1-x)} \,\,
 = \,\, \frac{s_1 + s_2}{s_1s_2} \, \frac{1}{\alpha'}\, -\, \frac{\pi^2}{6}(s_1+s_2) \, \alpha' \,+\, O(\alpha'^2) . \]
The pole at $\alpha' = 0$ confirms the divergence in Example \ref{ex:beta1}. 
The field theory limit equals
\begin{equation} \label{eq:fieldtheorylimitP1} 
r_\sigma(s) \, = \, {\rm lim}_{\alpha' \rightarrow 0} \, {\cal I}_{\sigma, \alpha'}(s) \, = \, {\rm lim}_{\alpha' \rightarrow 0} \, \alpha' \cdot {\cal I}_{\sigma}(\alpha' \, s)  \, = \, \frac{s_1 + s_2}{s_1s_2}. 
\end{equation}
\end{example}   

We now return to the $130$ regions on the cubic surface $X$ in Section \ref{sec2}.
Here $Y$ is the divisor of $27$ real lines. Let $\sigma \in Z_2(X,Y)$ be a region. 
The blow-down map $\varphi: X \rightarrow \mathbb{P}^2$
transforms ${\cal I}_\sigma(s)$ into the integral ${\cal I}_{\varphi(\sigma)}(s)$, where $\varphi(\sigma)$ is a region of the positive arrangement $(\tilde{X} = \mathbb{P}^2, \tilde Y)$.
The divisor $\tilde Y$ in $\PP^2$ is given by the $ 15$ lines $F_{ij}$ and the $ 6$ conics $G_i$.
 If $\varphi(\sigma) \subset \mathbb{P}_\R^2$ is a triangle, then the simplest integrals ${\cal I}_{\varphi(\sigma)}(s)$ use one extra line that is disjoint from $\varphi(\sigma)$.

\begin{example} \label{ex:trianglestringy}
Let $X = \mathbb{P}^2$ and let $Y$ be the arrangement of  four lines $xyz(x+y+z) = 0$.
The stringy integral of the region $\sigma = \mathbb{P}^2_{\geq 0}$ with canonical form $\,\omega(\sigma) = 
(xyz)^{-1} \,\Omega_{\mathbb{P}^2}\,$ equals
\[ {\cal I}_\sigma (s) \, = \, \int_\sigma \Big (\frac{x}{x + y + z} \Big )^{s_1} \Big (\frac{y}{x + y + z} \Big )^{s_2} \Big (\frac{z}{x + y + z} \Big )^{s_3} \, \frac{\Omega_{\mathbb{P}^2}}{xyz} \, = \, \frac{\Gamma(s_1)\Gamma(s_2)\Gamma(s_3)}{\Gamma(s_1+s_2+s_3)}.\]
This is sometimes called the \emph{multivariate Beta function}. The expression features the \emph{Gamma function} $\Gamma$. The integral converges when ${\rm Re}(s_i) > 0$. A compact way of writing it is
\begin{equation} \label{eq:Utriangle}
{\cal I}_\sigma (s)\,\,=\,\,
 \int_{U_{\geq 0}} u_1^{s_1}u_2^{s_2}u_3^{s_3} \, \omega(U_{\geq 0}),
\end{equation}
where $U_{\geq 0} = \{ u \in \R^3_{\geq 0} \, : \, u_1 + u_2 + u_3 = 1\}$ is the standard triangle.
  The 2-form $\omega(U_{\geq 0})$ is 
\[ \frac{\Omega_{\mathbb{P}^2}(u_1,u_2,u_3)}{u_1u_2u_3} \,\, = \,\, \frac{{\rm d}u_1 \wedge {\rm d}u_2}{u_1u_2}
\, -\,  \frac{{\rm d}u_1 \wedge {\rm d}u_3}{u_1u_3}  \,+\,  \frac{{\rm d}u_2 \wedge {\rm d}u_3}{u_2u_3} .\]
The field theory limit is the leading coefficient of ${\cal I}_\sigma(\alpha' s) = r_\sigma(s) \, \alpha'^{-2} + O(\alpha'^{-1})$,
which is
\begin{equation} \label{eq:fieldtheorylimitP2}
 r_\sigma(s) \, = \,  \frac{1}{s_1s_2} +  \frac{1}{s_2s_3} +  \frac{1}{s_1s_3} \, = \, \frac{s_1 + s_2 + s_3}{s_1s_2s_3}.
\end{equation}
This rational function appears in \cite[Example 14]{ST} as the \emph{amplitude} of the projective plane~$\mathbb{P}^2$.
The integrand in (\ref{eq:Utriangle}) is the likelihood function
for a ternary random variable in statistics.
\end{example}

To find a representation of ${\cal I}_\sigma(s)$ like in \eqref{eq:Utriangle} in more general cases, 
one defines $U_{\geq 0}$ as the image of $\sigma$ under the map $\phi: {\cal U}_\sigma \rightarrow \mathbb{C}^m$
specified by (\ref{eq:phi}). More examples are given below. 

The remainder of this section deals with stringy integrals on quadrilaterals and pentagons. 
In Sections \ref{sec2} and \ref{sec3}, these polygons live on a cubic surface.
Using $W({\rm E}_6)$-symmetry and the blow-down map, we can 
always represent them in the plane $\PP^2$, as shown in
Figures \ref{fig:firstpentagon} and \ref{fig:adjointquadrilateral}.
In toric geometry, it is natural to re-embed these polygons as the positive region of a  toric surface \cite{CT}.
 This construction is interesting for physics, as the corresponding
  $U$-representations like \eqref{eq:Utriangle} give rise to
   \emph{binary geometries} and \emph{$u$-equations} \cite[Section 2]{Lam_moduli}.

\begin{example} \label{ex:quadint}
Let $X = \mathbb{P}^1 \times \mathbb{P}^1$ and $Y$ is given by $x_0x_1y_0y_1(x_0+x_1)(y_0+y_1) = 0$. The region we are interested in consists of the nonnegative points of $\mathbb{P}^1 \times \mathbb{P}^1$.
This is the quadrilateral 
\[ \sigma \, = \, \bigl\{\,((x_0:x_1), (y_0:y_1)) \in \mathbb{P}^1 \times \mathbb{P}^1 \, : \, x_0, x_1, y_0, y_1 \geq 0 \,\bigr\}. \]
Its canonical 2-form is $\omega(\sigma) = (x_0x_1y_0y_1)^{-1} ({\rm d}x_0 \wedge {\rm d}y_0 - {\rm d}x_0 \wedge {\rm d}y_1 - {\rm d}x_1 \wedge {\rm d}y_0 + {\rm d}x_1 \wedge {\rm d}y_1)$.
We choose $\phi_1 = \frac{x_0}{x_0+x_1}, \phi_2 = \frac{x_1}{x_0+x_1}, \phi_3 = \frac{y_0}{y_0+y_1}, \phi_4 = \frac{y_1}{y_0+y_1}$. The function ${\cal I}_\sigma(s)$ factors as
\[{\cal I}_\sigma(s)  \,  = \,  B(s_1,s_2)B(s_3,s_4) \, = \, \frac{\Gamma(s_1) \Gamma(s_2)}{\Gamma(s_1+s_2)} \, \frac{\Gamma(s_3) \Gamma(s_4)}{\Gamma(s_3+s_4)}.\]
One can write ${\cal I}_\sigma(s)$ as an integral of a monomial on $U_{\geq 0} = \{ u \in \mathbb{R}^4_{\geq 0} \, : \, u_1 + u_2 = u_3 + u_4 = 1\}$: 
\[ {\cal I}_\sigma(s) \, = \, \int_{U_{\geq 0}} u_1^{s_1}u_2^{s_2} u_3^{s_3}u_4^{s_4} \, \omega(U_{\geq 0}). \]
The form $\omega(U_{\geq 0})$ is obtained by substituting $(x_0,x_1,y_0,y_1) = (u_1,u_2,u_3,u_4)$ 
into $\omega(\sigma)$. 

The field theory limit appears, up to a change of variables, in \cite[Example 15]{ST}: 
\begin{equation}
\label{eq:fieldlimit1}
 r_\sigma(s) \, = \, \frac{1}{s_1s_2} +  \frac{1}{s_2s_3} +  \frac{1}{s_3s_4} +  \frac{1}{s_1s_4} \, = \, \frac{(s_1+s_2)(s_3+s_4)}{s_1s_2s_3s_4}.
 \end{equation}
\end{example}

\begin{example} \label{ex:pentint}
We consider the smooth toric surface $X$ corresponding to the lattice pentagon with vertices $(0,0), \, (1,0), \, (2,1), \, (2,2), \, (0,2)$. The dense torus of $X$ is $(\mathbb{C}^*)^2$.
In $(\mathbb{C}^*)^2$ we~set
\[ Y_6^\circ \, = \,  V(1+t_1), \quad Y_7^\circ \,  = \,  V(1+t_2), \quad  Y_8^\circ \,  = \,  V(1+t_2+t_1t_2). \]
Their closures in $X$ are   $Y_6,Y_7,Y_8$.
    The divisor $Y$ is $Y_1 \cup \cdots \cup Y_8$, where $Y_1, \ldots, Y_5$ are the curves on $X$
    given by the edges of the pentagon. These curves bound the nonnegative region $\sigma $ in $ X$.
    We identify $\sigma$ with a pentagon via the moment map. The following  satisfies \eqref{eq:phi}: 
\begin{equation}
\label{eq:mysteryphi}
 \phi \, = \, \Big ( \,\frac{t_1}{1+t_1}, \,\frac{t_2(1+t_1)}{1+t_2+t_1t_2}, \,
\frac{1+t_2+t_1t_2}{(1+t_1)(1+t_2)},\, \frac{1+t_2}{1+t_2+t_1t_2}, \,\frac{1}{1+t_2}\, \Big ). 
\end{equation}
These are rational functions on $X$, represented in local coordinates on $(\mathbb{C}^*)^2 \subset X$. We have 
\begin{equation}
\label{eq:fivefive}
{\cal I}_\sigma(s) \, = \, \int_{\sigma} \phi_1^{s_1} \phi_2^{s_2}  \phi_3^{s_3}  \phi_4^{s_4}  \phi_5^{s_5} \, \omega(\sigma) \, = \, \int_{U_{\geq 0}} u_1^{s_1}u_2^{s_2}u_3^{s_3}u_4^{s_4}u_5^{s_5} \, \omega(U_{\geq 0}). 
\end{equation}
Here $\omega(\sigma) \, = \, \frac{{\rm d}t_1}{t_1} \wedge \frac{{\rm d}t_2}{t_2}$ and $U_{\geq 0} = U \cap \mathbb{R}^5_{\geq 0}$ where $U$ is the surface defined by the {\em $u$-equations}
\begin{equation}
\label{eq:ueqns5}
 u_1+u_3u_4\,=\,u_2+u_4u_5\,=\,u_3+u_1u_5\,=\,u_4+u_1u_2 \,=\, u_5+u_2u_3\,=\,1.
 \end{equation}
This affine surface is parametrized by the map $\phi: {\cal U}_\sigma \rightarrow \mathbb{C}^5$.
Indeed, the reader should check that the five quadratic equations in (\ref{eq:ueqns5})
are satisfied for the five coordinates of $\phi$ in (\ref{eq:mysteryphi}).

The canonical differential form $\omega(U_{\geq 0})$ is given by
\begin{small}
$$
\frac{{\rm d} u_1 \wedge {\rm d} u_2}{u_1 u_2}
+ \frac{{\rm d} u_1 \wedge {\rm d} u_3}{u_1 u_3}
- \frac{{\rm d} u_1 \wedge {\rm d} u_5}{u_1 u_5}
+ \frac{{\rm d} u_2 \wedge {\rm d} u_3}{u_2 u_3}
+ \frac{{\rm d} u_2 \wedge {\rm d} u_4}{u_2 u_4}
+ \frac{{\rm d} u_3 \wedge {\rm d} u_4}{u_3 u_4}
+ \frac{{\rm d} u_3 \wedge {\rm d} u_5}{u_3 u_5}
+ \frac{{\rm d} u_4 \wedge {\rm d} u_5}{u_4 u_5}.
$$
\end{small}
The open subset ${\cal U}_\sigma \subset X$ is affine. It is a partial compactification of the moduli space 
$\mathcal{M}_{0,5}$ of
five marked points on the Riemann sphere \cite[Section 2]{Brown}. The field theory limit of (\ref{eq:fivefive}) is 
\begin{equation}
\label{eq:fieldlimit2}
 r_\sigma(s) \, = \, \frac{1}{s_1s_2} +  \frac{1}{s_2s_3} +  \frac{1}{s_3s_4} +  \frac{1}{s_4s_5} +  \frac{1}{s_1s_5} \,  =  \, \frac{s_3s_4s_5 + s_1s_4s_5 + s_1s_2s_5 + s_1s_2s_3 + s_2s_3s_4}{s_1s_2s_3s_4s_5}. 
 \end{equation}
For physicists, this rational function is a biadjoint scalar amplitude in $\phi^3$ theory.
Note that its
numerator is the \emph{Segre cubic threefold},
which became $\mathcal{M}_{0,6}$ in the proof of Proposition~\ref{prop:segrecubic}.
\end{example}

\section{The Pezzotope}
\label{sec8}

Our most important positive geometry is the ${\rm E}_6$ pezzotope, which has dimension four.
Its interior is embedded in the moduli space $\mathcal{Y}(3,6) = X \backslash Y$ of marked cubic surfaces,
studied in~\cite{EGPSY, HKT, Nar, SY}. This section describes the pezzotope and its $U$-variety 
from Section~\ref{sec7}. Section \ref{sec9} explains ${\cal Y}(3,6)$, and the positive arrangement $(X,Y)$ is discussed in Section~\ref{sec10}.

The pezzotope is the simple $4$-polytope with $45$ vertices whose edge graph is shown in Figure~\ref{fig:pezzotope}.
  Five of its $15$ facets are cubes, which make $40$ vertices.
 Each cube is a bipartite graph with  $8$ vertices and $12$~edges. 
 Five additional vertices are connected
 to four cubes each, namely to one side of each bipartition.
 Pairs of cubes are linked via the other sides of their bipartitions.
 Thus the total number of edges seen in Figure~\ref{fig:pezzotope}
is $90 = 5 \cdot 12 + 5 \cdot 4 + \binom{5}{2}$.

Notice that the field-theory limits \eqref{eq:fieldtheorylimitP1}, \eqref{eq:fieldtheorylimitP2}, \eqref{eq:fieldlimit1}, and \eqref{eq:fieldlimit2} encode the vertices of the polygonal region $\sigma$ from which they arise. In this section, we reverse this perspective and reconstruct the pezzotope from a suitable rational function.
More precisely, we consider a rational function in $15$ variables $s_1,\ldots,s_{15}$, defined as the {\em ${\rm E}_6$ amplitude} in \cite[Theorem 8.2]{EGPSY}:
\begin{equation}
\label{eq:E6amplitude}
\!\!\!\!\!\!
\begin{small}
\begin{matrix}
\textcolor{myblue}{
\frac{1}{s_{10}s_{15}s_4s_6}+
\frac{1}{s_{10}s_{15}s_2s_7}+
\frac{1}{s_{15}s_2s_5s_6}+
\frac{1}{s_{15}s_4s_5s_7}+
\frac{1}{s_{15}s_2s_5s_7}+
\frac{1}{s_{15}s_4s_5s_6}+
\frac{1}{s_{10}s_{15}s_4s_7}+
\frac{1}{s_{10}s_{15}s_2s_6}} \smallskip
\\[2pt]
\textcolor{myorange}{
\frac{1}{s_{13}s_3s_8s_9}+
\frac{1}{s_{13}s_2s_3s_7}+
\frac{1}{s_{13}s_2s_5s_9}+
\frac{1}{s_{13}s_5s_7s_8}+
\frac{1}{s_{13}s_2s_5s_7}+
\frac{1}{s_{13}s_5s_8s_9}+
\frac{1}{s_{13}s_3s_7s_8}+
\frac{1}{s_{13}s_2s_3s_9}} \smallskip
\\[2pt]
\textcolor{mygreen}{
\frac{1}{s_1s_{14}s_8s_9}+
\frac{1}{s_1s_{14}s_4s_6}+
\frac{1}{s_{14}s_5s_6s_9}+
\frac{1}{s_{14}s_4s_5s_8}+
\frac{1}{s_{14}s_4s_5s_6}+
\frac{1}{s_{14}s_5s_8s_9}+
\frac{1}{s_1s_{14}s_4s_8}+
\frac{1}{s_1s_{14}s_6s_9}} \smallskip
\\[2pt]
\textcolor{mypurple}{
\frac{1}{s_1s_{12}s_3s_8}+
\frac{1}{s_1s_{10}s_{12}s_4}+
\frac{1}{s_{10}s_{12}s_3s_7}+
\frac{1}{s_{12}s_4s_7s_8}+
\frac{1}{s_{10}s_{12}s_4s_7}+
\frac{1}{s_{12}s_3s_7s_8}+
\frac{1}{s_1s_{12}s_4s_8}+
\frac{1}{s_1s_{10}s_{12}s_3}} \smallskip
\\[2pt]
\textcolor{mybordeaux}{
\frac{1}{s_1s_{11}s_3s_9}+
\frac{1}{s_1s_{10}s_{11}s_6}+
\frac{1}{s_{10}s_{11}s_2s_3}+
\frac{1}{s_{11}s_2s_6s_9}+
\frac{1}{s_{10}s_{11}s_2s_6}+
\frac{1}{s_{11}s_2s_3s_9}+
\frac{1}{s_1s_{11}s_6s_9}+
\frac{1}{s_1s_{10}s_{11}s_3}} \smallskip
\\[2pt]
{
\frac{1}{s_1s_3s_8s_9}+
\frac{1}{s_1s_4s_6s_{10}}+
\frac{1}{s_2s_3s_7s_{10}}+
\frac{1}{s_2s_5s_6s_9}+
\frac{1}{s_4s_5s_7s_8}}
\end{matrix}
\end{small}
\end{equation}
The colors point to those in Figure \ref{fig:pezzotope}: there are eight terms for each of the cubes. 
Let $\Delta$ be the corresponding $3$-dimensional simplicial complex with vertex set $\{s_1,\ldots,s_{15}\}$.
Thus $\Delta$ has $60$ edges, $90$ triangles and $45$ tetrahedra, namely the
quadruples $s_i s_j s_k s_l$  in the $45$ denominators of (\ref{eq:E6amplitude}).
We note that $\Delta$ is the {\em clique complex} of its edge graph $G$, with
$15$ vertices and $60$ edges, i.e., a set of vertices is in $\Delta$ if and only if all of its pairs are edges in~$G$. The complex $\Delta$ encodes the normal fan of the pezzotope; see Theorem \ref{thm:pezzotopeispolytope}. The $90$ edges in Figure \ref{fig:pezzotope} are the $90$ triangles of $\Delta$ and the $90$ pairs of terms in (\ref{eq:E6amplitude})
that share three~variables.

Our discussion suggests that the
${\rm E}_6$ amplitude (\ref{eq:E6amplitude})
should be the field theory limit of a stringy integral~(\ref{eq:takingcues}).
To derive the stringy integral \eqref{eq:fieldlimitpezzo}, we consider another representation of  the simplicial complex $\Delta$.
This is the following system of $15$ $u$-equations in $15$ variables:
			\begin{equation}
				\label{eq:perfectuE6}  \begin{small} \begin{matrix}
						u_1 + u_2 u_5 u_7 u_{13} u_{15} \,=\,
						u_2 + u_1 u_4 u_8 u_{12} u_{14} \,=\,
						u_3 + u_4 u_5 u_6 u_{14} u_{15} \,=\,
						u_4 + u_2 u_3 u_9 u_{11} u_{13} \,= \\
						u_5 + u_1 u_3 u_{10} u_{11} u_{12}\, =\,
						u_6 + u_3 u_7 u_8 u_{12} u_{13} \,=\,
						u_7 + u_1 u_6 u_9 u_{11} u_{14} \,=\,
						u_8 + u_2 u_6 u_{10} u_{11} u_{15} \,= \,\\
						u_9 + u_4 u_7 u_{10} u_{12} u_{15} \,\,= \,\,
						u_{10} + u_5 u_8 u_9 u_{13} u_{14} \,\,=\,\,
						u_{11} + u_4 u_5 u_7 u_8 u_{12} u_{13} u_{14} u_{15}\,\, = \\
						u_{12} + u_2 u_5 u_6 u_9 u_{11} u_{13} u_{14} u_{15} \quad = \quad
						u_{13} + u_1 u_4 u_6 u_{10} u_{11} u_{12} u_{14} u_{15} \quad=\,\, \\ \qquad
						u_{14} + u_2 u_3 u_7 u_{10} u_{11} u_{12} u_{13} u_{15} \quad = \quad
						u_{15} + u_1 u_3 u_8 u_9 u_{11} u_{12} u_{13} u_{14} \quad = \quad 1.
				\end{matrix} \end{small}
			\end{equation}
These equations are the trinomials $\,u_i + \prod_{ij \not\in G} u_j -1 \,$ for $i=1,2,\ldots,15$.
The system (\ref{eq:perfectuE6}) encodes the graph $G$ and hence its clique complex $\Delta$.
A computation yields the next result.

\begin{lem} \label{lem:15trinomials}
The $15$ trinomials in (\ref{eq:perfectuE6}) generate a prime ideal
in the polynomial ring $\C[u_1,\ldots,u_{15}]$. The 
irreducible variety $U$ it defines in $\C^{15}$ has dimension $4$ and degree $192$.
\end{lem}			

Next, we find the functions $\phi_i$ from \eqref{eq:phi} from a parametric representation of $U$.
Let $P$ be a $3 \times 6$ matrix
whose columns are coordinates for six points in $\PP^2 $.
We write $p_{ijk} $ for the $3 \times 3$ determinant formed by the
columns $i,j,k$ of~$P$.
We assume that no three points lie on a line (i.e., $p_{ijk} \not= 0$)
and the six points
are not on a conic. This means algebraically that
  \begin{equation}
                \label{eq:coconic}
                q \,\,=\,\, p_{134}\,p_{156}\,p_{235}\,p_{246} - p_{135} \,p_{146}\, p_{234}\, p_{256} \, \, \neq \, \, 0.
        \end{equation}
The following lemma is also  proved by a direct computation:

		\begin{lem} \label{lem:uparaE6}
			The variety $U$ defined by the $u$-equations  in (\ref{eq:perfectuE6}) has the parametrization
			$$
			\begin{matrix}
				\! u_1 = \frac{-q}{p_{126}p_{135}p_{234}p_{456}},
				&\!\! \!u_2 = \frac{p_{134}p_{156}p_{235}p_{246}}{p_{135}p_{146}
					p_{234}p_{256}} ,&  u_3 = \frac{p_{134}p_{356}}{p_{135}p_{346}}, & 
				\! u_4 = \frac{p_{136}
					p_{145}}{p_{135}p_{146}}, &     
				\!\! u_5 = \frac{p_{125}p_{136}p_{246}p_{345}}{p_{126}p_{135}p_{245} p_{346}}, \medskip \\
				u_6 = \frac{p_{136}p_{235}}{p_{135}p_{236}}, & \!\! \! u_7 = \frac{p_{123}p_{145}p_{246}
					p_{356}}{p_{124}p_{135}p_{236}p_{456}} ,& u_8 = \frac{p_{125}p_{356}}{p_{135}p_{256}},
				& \! u_9 = \frac{p_{125}p_{134}}{p_{124}p_{135}}, &
				u_{10} = \frac{p_{145}p_{235}}{p_{135} p_{245}} , \medskip \\
				u_{11} = \frac{p_{135}p_{234}}{p_{134}p_{235}}, &
				u_{12} = \frac{p_{135}p_{456}}{p_{145} p_{356}}, & \!
				\!\!\!\! u_{13} = \frac{p_{124}p_{135}p_{256}p_{346}}{p_{125}p_{134}p_{246} p_{356}}, \! &
				\!\! u_{14} = \frac{p_{126}p_{135}}{p_{125}p_{136}} ,&
				\!\!   u_{15} = \frac{p_{135}p_{146}p_{236}  p_{245}}{p_{136}p_{145}p_{235}p_{246}}.
			\end{matrix}
			$$
			\end{lem}			
The expressions in this parametrization should be thought of as cross ratios
for six points in $\PP^2$. They are invariant under row operations and under scaling the columns of $P$. This observation allows us to parametrize $U$ birationally. Following \cite[Theorem 10.6]{EGPSY}, we define 
\begin{equation}
    \label{eq:sixpoints7}  P \,\,=\,\, \left[
  \begin{array}{cccccc}
    1 & 0 & 0 & 1 & \frac{b+1}{b} & \frac{b c+c+1}{b c} \medskip \\        
    0 & 1 & 0 & \!-1 &\!\! -\frac{\left(a+1\right)
      \left(b+1\right)}{a b+b+1} & -\frac{\left(a+1\right)
      \left(b c d + b c + b d + c d + c + d + 1      
      \right)}{
      a b c d + a b c + a b d + b c d + b c + b d + c d + c + d + 1  }\!\! \medskip \\
    0 & 0 & 1 & 1 & 1 & 1 \\                                                                                    
        \end{array}                                                                                                         
  \right].
  \end{equation}
  
  \begin{prop} \label{prop:positiveparam}
  Restricting the map from Lemma \ref{lem:uparaE6} to the matrix $P$ from \eqref{eq:sixpoints7} gives a birational map $\psi: \mathbb{C}^4 \dashrightarrow U$. Moreover, the restriction $\psi_{|\mathbb{R}^4_{> 0}}: \mathbb{R}^4_{>0} \rightarrow U \cap \mathbb{R}^{15}_{>0}$ is a bijection.
  \end{prop}
  \begin{proof}
  We check that the following monomial map $\,{\rm im} \, \psi \rightarrow \mathbb{C}^4\,$ inverts $\,\psi$: 
 \begin{equation}
 \label{eq:abcdfrac}
  a \, = \, \frac{u_{10}}{u_5 u_8 u_9 u_{13} u_{14} }\,, \quad
        b\,=\,  \frac{u_9 u_{11}}{u_4 u_7 u_{12} u_{15}}\, ,\quad
        c \, =\, \frac{u_4 u_6 u_{14} u_{15}}{u_3 u_{13}}\, , \quad
        d \, = \,\frac{u_1 u_4 u_8 u_{12} u_{14}}{u_2}.
\end{equation}
Therefore, the Zariski closure of the image of $\psi$ is irreducible of dimension four. Since this image closure is contained in $U$ by Lemma \ref{lem:uparaE6} and $U$ is irreducible of dimension four by Lemma \ref{lem:15trinomials}, we conclude that $\psi: \mathbb{C}^4 \dashrightarrow U$ is indeed birational.
The $20$ maximal minors $p_{ijk}$
  and the conic constraint $q$ in (\ref{eq:coconic}) are rational 
  functions in the variables $a,b,c,d$ with positive coefficients. These rational functions
  exhibit the following $11$  irreducible factors:
 \begin{equation}
 \label{eq:g11}
 \begin{matrix}
  & g_1 = a+1, \quad g_2 = b+1, \quad g_3 = a b+b+1, \quad g_4 = c+1,\quad
   g_5 = b c+c+1,  \smallskip \\ & g_6 = d+1, \quad
     g_7 = a b d+b d+d+1,\quad
   g_8 = bcd+ b c+ bd + cd+ c + d + 1, \smallskip \\
&     g_9 \,=\,\, g_8 + abd, \, 
 \quad g_{10} \,=\,g_9 + abcd  , \quad g_{11} \,=\,  g_{10} + abc .
\end{matrix}
\end{equation}    
Every minor $p_{ijk}$ is a Laurent monomial in $a,b,c,d,g_1,\ldots,g_{11}$~and
$ - q  \, = \, \frac{a \,d\,g_1\, g_2 \,g_4}{b\, c \,g_3 \,g_{11}}$.
This shows that $\psi(\mathbb{R}^4_{>0}) \subseteq \mathbb{R}^{15}_{>0}$. Conversely, \eqref{eq:abcdfrac} shows that $\psi^{-1}(U \cap \mathbb{R}^{15}_{>0}) = \mathbb{R}^4_{>0}$.
  \end{proof}

The matrix $P$ in (\ref{eq:sixpoints7}) is very useful.
Returning to the beginning of this paper, we had chosen
$\, a\,=\,b\,=\,c\,=\,1 \,$ and $\,d=2\,$ to define the configuration
 (\ref{eq:sixpoints}) and hence the surface (\ref{eq:cubicsurface}).
  The positive parametrization from Proposition \ref{prop:positiveparam} reveals our main object, the ${\rm E}_6$ pezzotope.

\begin{thm} \label{thm:pezzotopeispolytope} There exists a simple $4$-polytope, called the {\em ${\rm E}_6$ pezzotope},
whose normal fan is given by the simplicial complex $\Delta$. Hence the
 f-vector of the ${\rm E}_6$ pezzotope is $(45,90,60,15)$.
 \end{thm}

\begin{proof}
The pezzotope is the Newton polytope of the product of the polynomials in \eqref{eq:g11}:
\begin{equation}
\label{eq:pezzotope}
 {\rm New}(g_1 g_2 \cdots g_{11}) \,\,= \,\,
{\rm New}(g_1) + {\rm New}(g_2) + \cdots + {\rm New}(g_{11}). 
\end{equation}
A computation verifies that this $4$-polytope is simple and it has the correct normal fan $\Delta$.
\end{proof}

The variety $U$ is a
{\em positive chart} of the smooth projective toric variety 
that is associated with the pezzotope.
This toric variety is constructed in \cite[Example 4.3]{CT}. The article \cite{CT}
offers a conceptual framework for $u$-equations, like
  (\ref{eq:ueqns5}) and (\ref{eq:perfectuE6}).
  Such equations are inspired by physics.
A key insight of \cite{CT} is that the $u_i$ are
 Cox coordinates in toric geometry \cite[Section 6.1]{MS}.
The facet normals of the pezzotope (\ref{eq:pezzotope})
are the columns of the $4 \times 15$ matrix
\setcounter{MaxMatrixCols}{15}
\begin{equation}
\label{eq:4by15}
\begin{bmatrix} 
\,\,0 & 0 & 0 &  0 & \!\! -1 & 0 & 0 & \!\! -1 & \!\! -1 & 1 & 0 & 0 & \!\! -1 & \!\! -1 &  0 \, \\
\,\,0 & 0 & 0 & \!\! -1 & 0 & 0& \!\! -1 & 0 &  1 &  0 &  1 & \!\! -1 & 0 &  0 & \!\! -1 \,\, \\
\,\,0 & 0 & \!\! -1 & 1 & 0 &  1 & 0 & 0 &  0 &  0 &  0 &  0 & \!\! -1 & 1 &  1 \,\\
\,\,1 &  \!\! -1 & 0 & 1 & 0 & 0 & 0  & 1 & 0  & 0 & 0 & 1 & 0 & 1 &  0 \, \\
\end{bmatrix}.
\end{equation}
The monomial map \eqref{eq:abcdfrac} is given by the four rows of this matrix. 
We substitute the monomials \eqref{eq:abcdfrac} into
the polynomial $g_i$ in  (\ref{eq:g11}).
This results in a rational function in $u = (u_1,\ldots,u_{15})$.
Let $g_i^h(u)$ denote the numerator of this rational function.
Replacing $g_i$ with $g_i^h(u)$ corresponds to homogenization in toric geometry.
Namely,
we pass from the affine coordinates  $a,b,c,d$ on the torus $(\C^*)^4$ to the
 Cox coordinates on the smooth
toric variety of the pezzotope.

\begin{prop}
The polynomials 
$g_1^h(u) - 1,\ldots, \,g_{11}^h(u) - 1$ generate the
prime ideal in Lemma \ref{lem:15trinomials}.
Hence, the
associated irreducible variety $\,U$
is a complete intersection in~$\C^{15}$.
\end{prop}

\begin{proof}
It follows from \cite[Corollary 3.10]{CT} that the polynomials $g_1^h(u) - 1,\ldots, \,g_{11}^h(u) - 1$ define the $U$-variety in the torus $(\mathbb{C}^*)^{15}$. The fact that the ideal is prime can be checked via computer algebra, and is expected to hold for general positive charts \cite[Conjecture~3.11]{CT}.  
\end{proof}

The intersection of our $4$-dimensional variety   with the nonnegative orthant is bounded:
 $$ U_{\geq 0} \,\, = \,\, U \, \cap \, \R^{15}_{\geq 0} \quad \subset \quad [0,1]^{15}.$$
It is stratified by the relatively open faces
 of $\R^{15}_{\geq 0}$. The strata are the faces of the pezzotope.

\begin{cor}  \label{cor:curvypezzo}
The semialgebraic set $\,U_{\geq 0}$ is a curvy realization of the ${\rm E}_6$ pezzotope.
\end{cor}
\begin{proof}
The variety $U$ is an affine subset of a toric variety. The semialgebraic set $U_{\geq 0}$ is identified with the ${\rm E}_6$ pezzotope via the toric moment map. This map identifies the strata of $U_{\geq 0}$ diffeomorphically with the open faces of the pezzotope \cite[Theorem 5.1]{CT}. The structure of these strata can alternatively be understood from the structure of the $u$-equations (\ref{eq:perfectuE6}). 
\end{proof}

The connection with stringy integrals is seen as follows: the function \eqref{eq:E6amplitude} is the limit
\begin{equation} \label{eq:fieldlimitpezzo} \lim_{\alpha' \rightarrow 0} \, (\alpha')^4 \, {\cal I}_{\sigma} (\alpha'  s)  \,\,\, = \,\,\,
 \lim_{\alpha' \rightarrow 0}\, (\alpha')^4 \, \int_{U_{\geq 0}} \! u_1^{\alpha'  s_1} u_2^{\alpha' s_2}  \cdots \, u_{15}^{ \alpha'  s_{15}} \, \, \omega(U_{\geq 0}), 
\end{equation}
where $\omega(U_{\geq 0})$ is the $4$-form ${\rm dlog} \, a \wedge {\rm dlog} \, b  \wedge {\rm dlog} \, c  \wedge {\rm dlog} \, d $, with $a,b,c,d$ as in \eqref{eq:abcdfrac}. Here $\sigma$ is the natural embedding of the pezzotope in its smooth toric variety $X$, and $Y$ consists of the toric boundary and the hypersurfaces \eqref{eq:g11}. 
This is analogous to Examples \ref{ex:quadint} and \ref{ex:pentint}.

\smallskip

Six labeled points in $\mathbb{P}^2$ in general position give a cubic surface as explained in Section~\ref{sec2}. By Proposition \ref{prop:positiveparam}, each point in the curvy pezzotope $U_{> 0}$ corresponds to such a point configuration by substituting the corresponding positive values of $a,b,c,d$ into the matrix $P$ in \eqref{eq:sixpoints7}. Let us take a closer look at the cornucopia of cubic surfaces arising in this way.
Consider the partition of the orthant $\R^4_{>0}$ given by the signs
of the following ten polynomials:
\begin{equation}
\label{eq:10eckardt}
 \begin{footnotesize} 
  \begin{matrix}
ad-1, abd+bd-a+d, \,bcd-ab+bc+bd+cd+c+d+1,\, ab^2d+b^2d-bcd-bc-cd-c-d-1, \\
ab^2c-abcd{+}b^2c-abd-bcd-bd-cd-c-d-1, \,\,
a^2b^2cd{+}ab^2cd{-}abd{-}bcd{-}bc{-}bd{-}cd{-}c{-}d{-}1,\\
abcd^2{-}bc^2d{+}abd^2{+}bcd^2{-}bc^2{-}c^2d{+}bd^2{+}cd^2{-}c^2{+}d^2{-}c{+}d, \\
abc^2d{+}abc^2{+}ac^2d{+}ac^2{-}abd{+}acd{-}bcd{+}ac{-}bc{-}bd{-}cd{-}c{-}d{-}1, \\
ab^2c^2d^2{+}ab^2c^2d{+}ab^2cd^2{+}abc^2d^2{+}b^2c^2d^2{+}abc^2d{+}b^2c^2d  {+}abcd^2{+}b^2cd^2 
  {+}2bc^2d^2{+}2bc^2d{+}2bcd^2  \\ \,\,+ \,c^2d^2 -abd  {+}c^2d{+}cd^2{-}bc{-}bd{-}c{-}d{-}1\,,\qquad
ab^2c^2d{-}a^2bcd{+}ab^2c^2{+}ab^2cd{+}abc^2d {+}b^2c^2d{-}a^2bd \\ \qquad \qquad \qquad \,+ \,abc^2
  {+}b^2c^2{+}b^2cd{+}2bc^2d{-}abd{-}acd{+}2bc^2{+}2bcd{+}c^2d{-}ac{-}ad{+}bc{+}c^2{+}cd{-}a{+}c.
  \end{matrix}
  \end{footnotesize}
\end{equation}
Their product vanishes precisely when one of the
$10$ triangles in~(\ref{eq:10triangles}) becomes an Eckardt~point. These polynomials were obtained in \cite{Zenodo} as follows. We express the condition for the existence of an Eckardt point as a polynomial relation between the Pl\"ucker coordinates of the matrix $P$ in \eqref{eq:sixpoints7}. This polynomial defines a hypersurface in ${\cal Y}(3,6)$, called the \emph{Eckardt divisor}. Among its irreducible factors, the polynomials \eqref{eq:10eckardt} are the only ones with both positive and negative coefficients, so these are the only factors which could potentially vanish on $\mathbb{R}^4_{>0}$. 

\begin{thm}  \label{thm:clebsch}
The ten hypersurfaces in (\ref{eq:10eckardt}) divide the
orthant $\R^4_{>0}$ into $120$ chambers. Within each chamber,
the cubic surface has $130$ regions, with precisely the
$10$ triangles~in~(\ref{eq:10triangles}).
On each hypersurface, one of the triangles becomes an Eckardt point. 
The orthant contains a unique {\em Clebsch cubic}, where
all ten question marks in Figure \ref{fig:endler2} are Eckardt points,~namely 
$$ (a,b,c,d) \in \R^4_{>0} \quad \hbox{with coordinates} \quad
d = \frac{3-\sqrt{5}}{2},\,\,
a = 3-d, \,
b = 2-d,\,
c = 1-d.
$$
\end{thm}

\begin{proof}
The Weyl group $W({\rm E}_6)$ acts on ${\cal Y}(3,6)_\mathbb{R}$, see Section \ref{sec9}. Consider the subgroup 
\[ {\rm Stab}(\mathbb{R}^4_{>0}) \, = \,  \{ g \in W({\rm E}_6) \, : \, g \cdot (a,b,c,d) \in \mathbb{R}^4_{>0} \text{ for each } (a,b,c,d) \in \mathbb{R}^4_{>0} \}.\]
This group is isomorphic to $S_5$ (Lemma \ref{lem:Weylaction}). Hence, ${\rm Stab}(P)$ acts on $P$. The fixed points of this action constitute the Eckardt divisor \cite[Remark 6(i), Corollary 1]{SY}. Moreover, the action of ${\rm Stab}(\mathbb{R}^4_{>0})$ tiles up $\mathbb{R}^4_{>0}$ into $|S_5| = 120$ fundamental domains or Coxeter chambers, each of whose closure in the pezzotope is combinatorially a $4$-simplex \cite{Y}. The union of the $S_5$-translates of such a simplex is the full pezzotope. See also \cite[proof of Theorem 7.2]{ACT}. This proves the number 120. The other statements are shown by computation. The Clebsch cubic was computed in \cite[Section 4.6]{SY}. It is called the \emph{center} of the pezzotope in~\cite{Y}.
\end{proof}

\begin{rmk}
One easily computes that, for $(a,b,c,d) \in \R^4_{>0}$,
the polynomials in (\ref{eq:10eckardt}) attain at least $120$
distinct sign vectors in $\{-1,+1\}^{10}$. 
Hence the division of $\R^4_{>0}$ has at least $120$ chambers. One could try to show that each of these  $120$ ``sign regions'' is connected via the cylindrical algebraic decomposition of $\mathbb{R}^4_{>0}$ minus the Eckardt divisor \cite[Chapter 5]{BPR}. Such computations are highly complex and do not terminate within reasonable time. A numerical calculation using the Julia package \texttt{HypersurfaceRegions.jl} \cite{BSW} gives numerical evidence for the number 120 in Theorem \ref{thm:clebsch} within a few minutes \cite{Zenodo}, but this does not give a~proof.
\end{rmk}

\section{Moduli of Cubic Surfaces}
\label{sec9}

We now explain the  moduli space $\mathcal{Y}(3,6)$
of  configurations of six points in general position in $\PP^2$.
Formally,  $\mathcal{Y}(3,6)$ is defined as the quotient ${\rm Gr}(3,6)^\circ/(\C^*)^6$, where
${\rm Gr}(3,6)^\circ$ denotes the open subset of the  complex Grassmannian  ${\rm Gr}(3,6)$
where all $20$ Pl\"ucker coordinates are nonzero
and additionally the polynomial $q$ in (\ref{eq:coconic})
is nonzero as well. This quotient is a very affine variety of dimension $4$.
It is embedded in the torus $(\C^*)^4$ by taking all matrices
\begin{equation} \label{eq:sixpoints8} \tilde{P} \, = \, \begin{bmatrix}
1 & 0 & 0 & 1 & x & y \\ 0 & 1 & 0 & -1 & -z & -w \\ 0 & 0 & 1 & 1 & 1 & 1
\end{bmatrix}
\end{equation}
whose $3 \times 3$ minors are nonzero
 and where $q = wx(y+z-1) + yz(1-w-x)$ is nonzero too.

 Six of the $20$ minors of $\tilde P$ are equal to one, 
 and four of them are $x,y,z,w$. Hence only 
 $10$ of the minors are needed for our arrangement.
We conclude that   $\mathcal{Y}(3,6)$ is the complement of $11 = 10+1$ hypersurfaces
in $(\C^*)^4$. 
The real moduli space ${\cal Y}(3,6)_\R$ consists of matrices \eqref{eq:sixpoints8} in which $\tilde p = (x,y,z,w) \in \mathbb{R}^4$, and $\tilde p$ avoids these $11$ hypersurfaces. This space is disconnected, and its connected components will later serve as \emph{regions} in a positive arrangement $(X,Y)$ compactifying ${\cal Y}(3,6)$. To count regions, the action of the Weyl group $W(\rm E_6)$ is essential. 

The Weyl group $W(\rm E_6)$ is generated by simple transpositions $g_{i,j} \in S_6$ and the Cremona transformation $g_{C}$. These act on $6$-point configurations via birational transformations in $(x,y,z,w)$. The generator $g_{i,j}$ swaps columns $i$ and $j$. For instance, for $(i,j) = (2,5)$, 
\begin{equation} \label{eq:g25}  \begin{bmatrix}
\, 1 & x & 0 & 1 & 0 & y \, \smallskip \\
\, 0 & \! -z & 0 & -1 & 1 &\! -w \,\smallskip \\
\, 0 & 1 & 1 & 1 & 0 & 1\, \\
\end{bmatrix} \, \sim \, \begin{bmatrix}
\, 1 & 0 & 0 & 1 & \frac{-x z + x}{x - z} & \frac{x z w - x w - y z^2 + y z}{x z - x w - z^2 + z w} \, \smallskip \\
\, 0 & 1 & 0 & -1 & z - 1 & \frac{-z w + w}{z - w} \, \smallskip \\
\, 0 & 0 & 1 & 1 & 1 & 1 \,
\end{bmatrix}
         \end{equation}
gives $g_{2,5} \cdot (x,y,z,w) = (\frac{-xz+x}{x-z},\ldots, \frac{zw-w}{z-w})$. The 
Cremona transformation $g_C$ is centered at the points $1$, $2$ and $3$. These points are mapped to
 $(1:0:0), (0:1:0), (0:0:1)$.
 For the other points we apply $(x_0:x_1:x_2) \mapsto (\ell_{12}\ell_{13}: \ell_{12}\ell_{23}: \ell_{13}\ell_{23})$, where $\ell_{ij}(x_0,x_1,x_2)$ is a linear form vanishing at points $i$ and $j$. In our coordinates, we have $g_C \cdot (x,y,z,w) = (x^{-1},y^{-1},z^{-1},w^{-1})$. 

A convenient way to see the action of $W({\rm E}_6)$ 
on $\mathcal{Y}(3,6)$ is to  use the following matrix:
  \begin{equation}\label{eq:dimatrix}
                P \,\, = \,\,
                \begin{bmatrix}
                        1 & 1 & 1 & 1 & 1& 1\\
                        d_1  & d_2 & d_3 & d_4 & d_5 & d_6 \\
                        d_1^3 & d_2^3 & d_3^3 & d_4^3 &d_5^3 & d_6^3\\
                \end{bmatrix} .
        \end{equation}
Every general $3 \times 6$ matrix $\tilde P$ can be written in this form after row operations and column scalings. The $3 \times 3$ minors of $P$ and the coconic condition decompose into linear factors:
$$ \, p_{ijk}  \,=\, (d_i-d_j)(d_i-d_k)(d_j-d_k) (d_i+d_j+d_k) \quad {\rm and}
 \quad q  \,=\,  (d_1+d_2+d_3+d_4+d_5+d_6) \!\! \prod_{1 \leq i < j \leq 6}\!\!\! (d_i-d_j). $$
The $36$ linear forms in these factorizations are the elements in the
{\em root system} of type ${\rm E}_6$. They define $36$ hyperplanes
through the origin in $\R^6$. This is the ${\rm E}_6$
{\em reflection arrangement}. Each of its regions is labeled by an element of the group $W({\rm E}_6)$, so there are
$51840$ regions. The Weyl group $W({\rm E}_6)$ acts freely and transitively on these regions by linear reflections. 

\begin{example}
Each transposition $g_{i,j} \in W(\rm E_6)$ acts on (\ref{eq:dimatrix}) by swapping columns.
It is instructive to work out the Cremona transformation $g_C$. This sends the matrix \eqref{eq:dimatrix} to 
\begin{equation} \label{eq:cremona2} \begin{bmatrix}
1 & 0 & 0 & f_1(d_4) & f_1(d_5) & f_1(d_6) \\ 
0 & 1 & 0 & f_2(d_4) & f_2(d_5) & f_2(d_6) \\ 
0 & 0 & 1 & f_3(d_4) & f_3(d_5) & f_3(d_6)
\end{bmatrix} \, \sim \,  \begin{bmatrix}
1 & 1 & 1 & 1 & 1 & 1\\
v_1 & v_2&v_3 & v_4 & v_5 & v_6 \\ 
v_1^3 & v_2^3 & v_3^3 &v_4^3 & v_5^3 & v_6^3 
\end{bmatrix}\end{equation}
where $f_i(d) = (d-d_i)(d+d_i+d_j)(d+d_i+d_k)$ for $\{i,j,k\} = \{1,2,3\}$, $v_i = d_i-\frac{2}{3}(d_1+d_2+d_3)$ for $i = 1, 2, 3$ and $v_i = d_i + \frac{1}{3}(d_1+d_2+d_3)$ for $i = 4, 5, 6$. The curve $(f_1:f_2:f_3)$ is a cuspidal cubic. It takes the form $(1:v:v^3)$ after applying the coordinate~change suggested in the first three columns of the righthand side. The linear map $(d_1, \ldots, d_6) \mapsto (v_1, \ldots, v_6)$ is the reflection at the hyperplane $d_1+d_2+d_3$, defined by using the appropriate inner product. 
\end{example}

Let ${\cal R}$ be the set of connected components of ${\cal Y}(3,6)_\R$, and let ${\rm comp}: {\cal Y}(3,6)_\R \rightarrow {\cal R}$ be the map that sends a point $\tilde p \in {\cal Y}(3,6)_\R$ to its connected component.
In symbols, ${\rm comp}(\tilde p) \ni \tilde p$. 

\begin{lem} \label{lem:Weylaction}
The Weyl group $W({\rm E}_6)$ acts transitively on ${\cal R}$ by $g \cdot R = {\rm comp}(g \cdot \tilde p)$,
where $R \in {\cal R}$ and $\tilde p \in R$. 
Each stabilizer group ${\rm Stab}(R) = \{ g \in W({\rm E_6}) \, : \, g \cdot R = R \}$ has order $120$.
\end{lem} 

\begin{proof}
The action $g \cdot R = {\rm comp}(g \cdot \tilde p)$ is well-defined:
 if $\tilde p_1$ and $\tilde p_2$ are in the same region $R$, then, by continuity,
there is a path $\gamma$ from $\tilde p_1$ to $\tilde p_2$ in $R$ whose image $g \cdot \gamma$ is in  ${\rm comp}(g \cdot \tilde p_1)$. To see that the action is transitive, note that the complement in $\mathbb{R}^6$ of the $\rm E_6$ reflection arrangement maps surjectively to ${\cal Y}(3,6)_\R$ via \eqref{eq:dimatrix}. The Weyl group acts transitively on the chambers of that  arrangement, and each chamber maps into a region of ${\cal Y}(3,6)$. 
To show $|{\rm Stab}(R)| = 120$, we compute the stabilizer of any connected component $R \in {\cal R}$. That is, for  some $\tilde{p} \in R$ and for each $g \in W({\rm E}_6)$, we compute the set ${\rm Stab}(R) = \{ g \in G \, : \, {\rm comp}(g \cdot \tilde{p}) = R \}$ in \cite{Zenodo}, and confirm that it has order 120. By \cite[Theorem 2(2)]{SY}, ${\rm Stab}(R) \simeq S_5$. 
\end{proof}

\begin{cor} \label{cor:432}
The open real moduli space $\mathcal{Y}(3,6)_\R$ has $|{\cal R}| = 432$ connected components.
\end{cor}
\begin{proof}
By Lemma \ref{lem:Weylaction} and the orbit-stabilizer theorem, we have $|W({\rm E}_6)| = |{\cal R}| \cdot 120$. 
\end{proof}

To double-check the number $432$ from Corollary \ref{cor:432}, we partition $\R^4$ into regions where the minors $p_{ijk}$
of the matrix $\tilde P$ have non-zero constant signs and so does $q$.
There are $432$ such regions, and each region is~connected.
See also the computation after (\ref{eq:again432}) below.

The number $432$ counts choices of ten Eckardt triples complementary to a Schl\"afli double-six.
There are $36$ double-sixes
in  $\mathcal{S}^{10}_{27}$. For simplicity,  fix 
the standard one in (\ref{eq:doublesix2}). Each Eckardt triple has the form
$F_{ij} F_{kl} F_{mn}$, where $ij, kl, mn$ is one of the $15$ perfect matchings
of the complete graph $K_6$. This graph has six $1$-factorizations,
i.e.~partitions of its edge set into five perfect matchings.
A choice of ten Eckhardt triples  is the complement of a $1$-factorization,
so there are six choices. Finally, our count has an extra factor of two
because we can swap the two rows in (\ref{eq:doublesix2}).
Thus the total number of labelings equals $36 \times 6 \times 2 = 432$.

Finally, we can embed the moduli space $\mathcal{Y}(3,6)_\R$ into $\R^{15}$ by the map
in Lemma \ref{lem:uparaE6}. Namely, $\mathcal{Y}(3,6)_\R$ is the subset of
 $U_\R$ where all $15$ coordinates $u_i$ are non-zero.
From that map, we find
that precisely $432$ of the $2^{15}$ sign patterns can be realized by solutions of
(\ref{eq:perfectuE6}).

\begin{lem}	\label{lem:P5toC15}
Evaluating the minors $p_{ijk}$ on the matrix (\ref{eq:dimatrix}), 
the map in Lemma \ref{lem:uparaE6} becomes
			$$ \begin{scriptsize}
				\begin{matrix}  
					u_1 & = &
					(d_6{-}d_3) (d_2{-}d_5) (d_1{-}d_4) (d_1{+}d_2{+}d_3{+}d_4{+}d_5{+}d_6)/((d_4{+}d_5{+}d_6)(d_2{+}d_3{+}d_4)(d_1{+}d_3{+}d_5)(d_1{+}d_2{+}d_6)) , \\
					u_2 & = & (d_1{+}d_3{+}d_4) (d_1{+}d_5{+}d_6) (d_2{+}d_3{+}d_5) (d_2{+}d_4{+}d_6)/((d_2{+}d_5{+}d_6)(d_2{+}d_3{+}d_4)(d_1{+}d_4{+}d_6)(d_1{+}d_3{+}d_5)) , \\
					u_3 & = & (d_1{-}d_4) (d_1{+}d_3{+}d_4) (d_5{-}d_6) (d_3{+}d_5{+}d_6)/((d_4{-}d_6)(d_3{+}d_4{+}d_6)(d_1{-}d_5)(d_1{+}d_3{+}d_5)) , \\
					u_4 & = & (d_3{-}d_6) (d_1{+}d_3{+}d_6) (d_4{-}d_5) (d_1{+}d_4{+}d_5)/((d_4{-}d_6)(d_1{+}d_4{+}d_6)(d_3{-}d_5)(d_1{+}d_3{+}d_5)) , \\
					u_5 & = & (d_1{+}d_2{+}d_5) (d_1{+}d_3{+}d_6) (d_2{+}d_4{+}d_6) (d_3{+}d_4{+}d_5)/((d_3{+}d_4{+}d_6)
					(d_2{+}d_4{+}d_5)(d_1{+}d_3{+}d_5)(d_1{+}d_2{+}d_6)) , \\ 
					u_6 & = & (d_1{-}d_6) (d_1{+}d_3{+}d_6) (d_2{-}d_5) (d_2{+}d_3{+}d_5)/((d_2{-}d_6)(d_2{+}d_3{+}d_6)(d_1{-}d_5)(d_1{+}d_3{+}d_5)) , \\
					u_7 & = & (d_1{+}d_2{+}d_3) (d_1{+}d_4{+}d_5) (d_2{+}d_4{+}d_6) (d_3{+}d_5{+}d_6)/((d_4{+}d_5{+}d_6)(d_2{+}d_3{+}d_6)(d_1{+}d_3{+}d_5)(d_1{+}d_2{+}d_4)) , \\
					u_8 & = & (d_1{-}d_2) (d_1{+}d_2{+}d_5) (d_3{-}d_6) (d_3{+}d_5{+}d_6)/((d_2{-}d_6)
					(d_2{+}d_5{+}d_6)(d_1{-}d_3)(d_1{+}d_3{+}d_5)) , \\
					u_9 & = & (d_2{-}d_5) (d_1{+}d_2{+}d_5) (d_3{-}d_4) (d_1{+}d_3{+}d_4)/((d_3{-}d_5)
					(d_1{+}d_3{+}d_5)(d_2{-}d_4)(d_1{+}d_2{+}d_4)) , \\
					u_{10} & = & (d_1{-}d_4) (d_1{+}d_4{+}d_5) (d_2{-}d_3) (d_2{+}d_3{+}d_5)/((d_2{-}d_4)(d_2{+}d_4{+}d_5)
					(d_1{-}d_3)(d_1{+}d_3{+}d_5)) , \\
					u_{11} & = & (d_1{-}d_5) (d_1{+}d_3{+}d_5) (d_2{-}d_4) (d_2{+}d_3{+}d_4)/((d_2{-}d_5)
					(d_2{+}d_3{+}d_5)(d_1{-}d_4)(d_1{+}d_3{+}d_4)) , \\
					u_{12} & = & (d_1{-}d_3) (d_1{+}d_3{+}d_5) (d_4{-}d_6) (d_4{+}d_5{+}d_6)/((d_3{-}d_6)(d_3{+}d_5{+}d_6)
					(d_1{-}d_4)(d_1{+}d_4{+}d_5)) , \\
					u_{13} & = & (d_1{+}d_2{+}d_4) (d_1{+}d_3{+}d_5) (d_2{+}d_5{+}d_6) (d_3{+}d_4{+}d_6)/
					((d_3{+}d_5{+}d_6)(d_2{+}d_4{+}d_6)(d_1{+}d_3{+}d_4)(d_1{+}d_2{+}d_5)) , \\
					u_{14} & = & (d_2{-}d_6) (d_1{+}d_2{+}d_6) (d_3{-}d_5) (d_1{+}d_3{+}d_5)/
					((d_3{-}d_6)(d_1{+}d_3{+}d_6)(d_2{-}d_5)(d_1{+}d_2{+}d_5)) , \\
					u_{15} & = & (d_1{+}d_3{+}d_5) (d_1{+}d_4{+}d_6) (d_2{+}d_3{+}d_6) (d_2{+}d_4{+}d_5)
					/((d_2{+}d_4{+}d_6)(d_2{+}d_3{+}d_5)(d_1{+}d_4{+}d_5)(d_1{+}d_3{+}d_6)).
				\end{matrix}
			\end{scriptsize}
			$$
These rational functions of degree $0$ give a parametrization $\PP^5 \dashrightarrow \C^{15}$ of the
fourfold $\,U$.
\end{lem}

Lemma \ref{lem:P5toC15} reveals the action of
$W({\rm E}_6)$ on the $432$ connected components of $\mathcal{Y}(3,6)_\R$.
Each region in the complement of the $36$ hyperplanes in $\R^6$ has a sign vector in $\{-1,+1\}^{36}$.
Plugging this sign vector into the formulas above, we obtain a sign vector in
$\{-1,+1\}^{15}$ which labels an open region in $U_\R$. Under this map of
sign vectors, each fiber has the same cardinality $120$. Hence, the $51840$ regions of the ${\rm E}_6$
arrangement are organized into $432$ groups of size $120$. All regions in the same group
are mapped to the same region in $\mathcal{Y}(3,6)_\R$.

We have shown that the moduli space $\mathcal{Y}(3,6)_\R$ has
$432$ regions, and these  lie in a single orbit under the action of the
Weyl group $W({\rm E}_6)$. Each  open region can  be identified with the set 
$U_{>0}$ of positive solutions to (\ref{eq:perfectuE6}).
By Corollary~\ref{cor:curvypezzo}, its closure $U_{\geq 0}$ is a curvy pezzotope.

In Theorem \ref{thm:clebsch}, the $432$ pezzotopal regions in  ${\cal Y}(3,6)_\R$ were further subdivided by the \emph{Eckardt divisor} ${\cal E}_\R$. This is the real part of the divisor consisting of all cubic surfaces with at least one Eckardt point. Set ${\cal Y}(3,6)^\circ_\R = {\cal Y}(3,6)_\R \backslash {\cal E}_\R$, and let ${\cal R}^\circ$ be the set of its connected components, referred to as \emph{Eckardt regions} below. The map ${\rm comp}^\circ: {\cal Y}(3,6)_\R^\circ \rightarrow {\cal R}^\circ$ sends a point to its Eckardt region. 
We are now ready to conclude the proof of Theorem \ref{thm:130}. 

\begin{proof}[Proof of Theorem \ref{thm:130}]
For  any real cubic surface $X$ with $27$ real lines and no Eckardt points, we defined a graph ${\cal G}(X)$.
The vertices are the $135$ intersection points of the lines. Two vertices make an edge if the two intersection points are connected by a line segment on $X$ which does not meet any other line. 
We computed ${\cal G}(X)$ for the surface in \eqref{eq:cubicsurface} via Table~\ref{tab:circular}.
We shall now prove the following:
{\em 
For any real cubic surface $X$ with $27$ real lines and having no Eckardt points, the graph ${\cal G}(X)$ contains precisely $10$ three-cycles, $90$ chordless four-cycles, and $30$ chordless five-cycles. For $k \geq 6$, there are no chordless $k$-cycles  in ${\cal G}(X)$. }

Let $\tilde p \in {\cal Y}(3,6)_\R^\circ$ be a marked cubic surface corresponding to $X = X_{\tilde p}$. It is clear that ${\cal G}(X) = {\cal G}(X_{\tilde q})$ for any $\tilde q \in {\rm comp}^\circ(\tilde p)$ in the same Eckardt region as $\tilde p$. Indeed, for the graph ${\cal G}(X)$ to change, one of the cyclic orderings in Table \ref{tab:circular} needs to change, which cannot happen without crossing ${\cal E}_\R$. By Lemma \ref{lem:Weylaction}, we may assume that $\tilde p$ lies in the positive pezzotope with coordinates $(a,b,c,d) \in \mathbb{R}^4_{>0}$. That pezzotope only meets 10 components of the Eckardt divisor, listed in \eqref{eq:10eckardt}. These correspond to the 10 triangles in \eqref{eq:10triangles}, or the 10 question marks in Figure \ref{fig:endler2}. We verify that the number of $k$-cycles for $k =3, 4, 5$ is independent of the way the question marks in Figure \ref{fig:endler2} are filled in, and no larger cycles are ever present. Hence, crossing over to a different Eckardt region within our pezzotope does not change the count of the theorem, so the first part of Theorem \ref{thm:130} follows from our computation in Section~\ref{sec2}. 

We proved in Lemma \ref{lem:109} that $\dim \Omega^2_{\log}(X \backslash Y) = 109$. The fact that each of the $130$ polygons is a positive geometry was stated in Proposition \ref{prop:130canforms}, which also provides their canonical forms. By Remark \ref{rmk:realrank}, these canonical forms span a $109$-dimensional space. 
\end{proof}

\section{Positive Arrangements for Moduli Space}
\label{sec10}

We next present a positive arrangement $(X,Y)$ which compactifies $\mathcal{Y}(3,6)$.
The projective variety $X$ is the closure of the image of
the map $\mathcal{Y}(3,6) \rightarrow \PP^{39}$ whose $40$ coordinates are
\[ \begin{footnotesize}
 \begin{matrix}
p_{123} p_{124} p_{156} p_{256} p_{345} p_{346} , 
p_{123} p_{125} p_{146} p_{246} p_{345} p_{356} , 
p_{123} p_{126} p_{145} p_{245} p_{346} p_{356} , p_{123} p_{134} p_{156} p_{245} p_{246} p_{356} ,
\\  p_{123} p_{135} p_{146} p_{245} p_{256} p_{346},
p_{123} p_{136} p_{145} p_{246} p_{256} p_{345} , 
p_{123} p_{145} p_{146} p_{234} p_{256} p_{356} , 
p_{123} p_{145} p_{156} p_{235} p_{246} p_{346} ,\\
 p_{123} p_{146} p_{156} p_{236} p_{245} p_{345} , p_{124} p_{125} p_{136} p_{236} p_{345} p_{456}, 
p_{124} p_{126} p_{135} p_{235} p_{346} p_{456} , p_{124} p_{134} p_{156} p_{235} p_{236} p_{456},\\
 p_{124} p_{135} p_{136}p_{234} p_{256} p_{456} , p_{124} p_{135} p_{146} p_{236} p_{256} p_{345} , 
 p_{124} p_{135} p_{156} p_{236} p_{245} p_{346},  p_{124} p_{136} p_{145} p_{235} p_{256} p_{346} , \\
  p_{124} p_{136} p_{156} p_{235} p_{246} p_{345} , p_{125} p_{126} p_{134}p_{234} p_{356} p_{456} , 
  p_{125} p_{134} p_{136} p_{235} p_{246} p_{456} , p_{125} p_{134} p_{146} p_{236} p_{245} p_{356},  \\
p_{125} p_{134} p_{156} p_{236} p_{246} p_{345} , p_{125} p_{135} p_{146} p_{234} p_{236} p_{456} , 
p_{125} p_{136} p_{145} p_{234} p_{246} p_{356} , p_{125} p_{136} p_{146} p_{234} p_{256} p_{345} , \\
 p_{126} p_{134} p_{135} p_{236} p_{245} p_{456},  p_{126} p_{134} p_{145} p_{235} p_{246} p_{356} , 
 p_{126} p_{134} p_{156} p_{235} p_{245} p_{346} , \,p_{126} p_{135} p_{145} p_{234} p_{256} p_{346} ,\\
  p_{126} p_{135} p_{146} p_{234} p_{245} p_{356} ,\, p_{126} p_{136} p_{145} p_{234} p_{235} p_{456},\, \\
p_{123} p_{456} q ,\, p_{124} p_{356} q , \, p_{125} p_{346} q ,\, p_{126} p_{345} q , \,p_{134} p_{256} q ,\,
 p_{135} p_{246} q, \,p_{136} p_{245} q ,\, p_{145} p_{236} q ,\, p_{146} p_{235} q , \,p_{156} p_{234} q.
\end{matrix}
\end{footnotesize} \]
Each of these $40$ products has 
degree $(3,3,3,3,3,3)$ in the $\Z^6$-grading
on the polynomial ring in $20$ variables $p_{ijk}$.
The $30$ monomials are characterized by the
property that their six subsets $\{i,j,k\}$
cover all $15$ index pairs.
This recipe gives a well-defined map
$\mathcal{Y}(3,6) \rightarrow X$.

We have explained the action of $W({\rm E}_6)$ on ${\cal Y}(3,6)$ in Section \ref{sec9}. This induces an action on $X$. A main desirable feature of the compactification $X$ is that $W({\rm E}_6)$ acts on $X$ by permuting coordinates in $\PP^{39}$. 
This is easy to check for $g_{i,j}$, which permutes the labels of $p_{ijk}$.
The divisor $Y$ is the intersection of $X$
with the union of the $40$ coordinate hyperplanes.

The variety $X$ is called the {\em Yoshida variety}, after
Masaaki Yoshida \cite{SY}. 
We refer to the pair $(X,Y)$ as the {\em Yoshida arrangement}, and we refer to 
the set difference $X \backslash Y$ as the {\em open Yoshida variety}.
This is the locus of points on $X$ where all $40$ coordinates are non-zero.

\begin{rmk} \label{rem:degree9}
Consider the parametrization of $X$ on the matrix (\ref{eq:dimatrix}).
The $40$ coordinates are polynomials of degree $24$
that factor into linear factors in $d_1,\ldots,d_6$.
After removing the common factor $\prod_{1 \leq i < j \leq 6}(d_i-d_j)$,
the parametrization of $X$ is given by $40$ polynomials of degree nine.
Namely, each coordinate is a product of nine linear factors that are ${\rm E}_6$ roots.
\end{rmk}

The Yoshida variety $X$ was studied  by Ren, Sam and Sturmfels in \cite[Section 6]{RSS1}.
They found that the homogeneous prime ideal of $X$ is minimally generated by
$30$ linear forms and $30$ cubic binomials in the $40$ variables. Geometrically,
$X$ is the intersection of a linear subspace $\PP^9$ with
a $15$-dimensional toric variety inside $\PP^{39}$; see \cite[Theorem 6.1]{RSS1}.
The parametrization of $X$ is invertible on $X \backslash Y$. In particular, the monomials in Lemma \ref{lem:uparaE6} are expressed as monomials in the $40$ entries of ${\cal Y}(3,6) \rightarrow \mathbb{P}^{39}$ by solving a linear algebra~problem.

\begin{lem} \label{lem:864}
The moduli space $\mathcal{Y}(3,6)$ is isomorphic to the
open Yoshida variety $X \backslash Y$.
In particular, there are $2 \cdot 432$ distinct sign patters in $\{-1,+1\}^{40}$
which are realized in $X_\R \backslash Y_\R$.
\end{lem}

The following result completes steps 1-4 in the Workflow \ref{workflow} for the moduli~space.
And, in light of Corollary \ref{cor:432} and Lemma \ref{lem:864},
 it also completes the proof of Theorem \ref{thm:432}.
 
\begin{thm} \label{thm:cr150}
The Yoshida pair $(X,Y)$ is a positive arrangement with
 $432$ regions, and
 \begin{equation}
\label{eq:150}
{\rm cr}_{\mathbb{R}}(X,Y) \,=\, {\rm cr}(X,Y) \,=\, 150. 
\end{equation}
\end{thm}

\begin{proof}
We verify the four conditions in Definition  \ref{def:positivearrangement}. For this, we use several facts about $X$ and its desingularization stated in \cite[Section 6]{RSS1}. 
First, the  Yoshida variety $X$ has a real parametrization and hence it has a smooth real point.
The boundary $Y$ is the intersection of $X$ with 
the $40$ coordinate hyperplanes in $\PP^{39}$.
On each coordinate hyperplane, $Y$ decomposes into $9$
irreducible components, as described in Remark~\ref{rem:degree9}.
In total,  the divisor $Y$ has $36$ irreducible components, one for each
hyperplane in the ${\rm E}_6$ arrangement. 
Each irreducible component is defined over $\R$.
The image of a general real point on that hyperplane is smooth on
its component of $Y$. This verifies the second condition in  Definition~\ref{def:positivearrangement}.
The singular locus ${\rm Sing}(X)$ of the Yoshida variety consists of $40$ points,
and these lie on the boundary $Y$. 

The singular points are computed from the $120$
linearly dependent triples in the ${\rm E}_6$ root system,
such as $\{d_1-d_2, d_1+d_3+d_4,d_2+d_3+d_4\}$.
These are the root subsystems of type $A_2$.
The parametrization of $X$ via $(d_1,d_2,d_3,d_4,d_5,d_6)$ contracts
the corresponding codimension~$2$ subspace to a singular point of $X$ with $28$ coordinates equal to zero. 
Each singular point arises from three distinct  
triples which form an $A_2^{\times 3}$ root subsystem;
see \cite[equation (31)]{EGPSY}.
Hence $X \backslash Y$ is contained in
$X_{\rm reg}$, as required in the third condition of  Definition  \ref{def:positivearrangement}.

Point four is that  the Yoshida pair $(X,Y)$ has genus zero. Let $X'$ be the projective variety obtained by blowing up
the $40$ singular points on $X$, and let $Y'$
be the total transform of $Y$. The variety $X'$ is {\em Naruki's cross-ratio variety}~\cite{Nar}. The divisor $Y'$ has
$76 = 36+40$ irreducible components, and it is simple normal crossing. Since $(X',Y') \rightarrow (X, Y)$ is a \emph{modification} in the sense of \cite[Definition 1.2]{BD}, by \cite[\S 3.2.1]{BD} the genus of $(X,Y)$ equals that of $(X',Y')$. Hence, it is equivalent to show that $g(X',Y') = 0$. For this, we use \cite[Corollary 3.13]{BD}: 
\[  g(X') \, \leq \,  g(X',Y') \, \leq \, g(X') + \sum_{i = 1}^{36} g(Y'_i, Z'_i) + \sum_{j=1}^{40} g(Y'_{36+j}, Z'_{36+j}), \]
where $Y'_1, \ldots, Y'_{36}$ are the irreducible components of the strict transform of $Y$, and $Y'_{37}, \ldots, Y'_{76}$ are the $40$ irreducible components of the exceptional divisor. For $i = 1, \ldots, 76$, we set $Z'_i = Y_i' \cap \bigcup_{j \neq i} Y_j'$. First, $g(X') = 0$ because $X'$ is smooth and birational to $\mathbb{P}^4$ \cite[Proposition 3.2]{BD}. 
Hacking, Keel and Tevelev characterized $(X',Y')$ as the
log canonical model of $X \backslash Y = X' \backslash Y'$ \cite[Corollary 9.2]{HKT}. 
They also showed that each of the $36$ components $Y'_1, \ldots, Y'_{36}$ of the strict transform of $Y$
is the log canonical model for the root system  ${\rm A}_5$.
That is, for $i = 1, \ldots, 36$, $(Y_i', Z_i')$ is the $3$-dimensional moduli space $\overline{\mathcal{M}}_{0,6}$
together with its Deligne-Mumford boundary.
That pair is well-known to have genus zero. Each pair $(Y_{36+j}', Z_{36+j}'), j = 1, \ldots, 40$, corresponding to an irreducible component of the exceptional divisor is isomorphic to the toric threefold $\mathbb{P}^1 \times \mathbb{P}^1 \times \mathbb{P}^1$ with its toric boundary. This pair also has genus zero. We conclude that $g(X',Y') = g(X,Y) = 0$. 
This establishes the first claim in Theorem~\ref{thm:cr150}.

Here is how to see $\overline{\mathcal{M}}_{0,6}$ 
geometrically in the boundary $Y'$. Consider what happens
when $q$ becomes zero in our parametrization of $X$.
The last $10$ coordinates of $\PP^{39}$ vanish.
In the limit, our six points in $\PP^2$ lie on a conic. 
The moduli space for six points on a conic is $\overline{\mathcal{M}}_{0,6}$. There are $36$ such conics, corresponding to the $36$ double-sixes in the Schl\"afli graph ${\cal S}^{10}_{27}$.

We now determine the combinatorial rank ${\rm cr}(X',Y')$, which equals ${\rm cr}(X,Y)$ by \cite[\S 4.3.1]{BD}. The {\em Naruki complex} is the intersection complex of the
Naruki boundary $Y'$. It is a $3$-dimensional simplicial complex on $76$ vertices with $630$ edges, $1620$ triangles, and $1215$ tetrahedra.
A classification of all simplices according to $W({\rm E}_6)$-orbits
is given in \cite[Table 5]{RSS1}.
We now apply \cite[Proposition 4.4]{BD}. The alternating sum of face numbers is ${\rm cr}(X',Y')$:
\begin{equation}
\label{eq:150altsum}
  {\rm cr}(X,Y)  \,\,=\,\,{\rm cr}(X',Y')\,\,=\,\, 1215 - 1620 + 630 - 76 + 1 \,\, = \,\, 150. 
\end{equation}

We will give a computational proof of ${\rm cr}_\R(X,Y) = {\rm cr}(X,Y)$ in Section \ref{sec11}. Here, we give a theoretical argument to show that every relative cycle  in $H_4(X,Y) = H_4(X',Y')$ has a real representative.
The claim follows from the action of $W({\rm E}_6)$ on  $X$ and its tropicalization ${\rm trop}(X)$.
The Naruki complex is combinatorially isomorphic to the tropical Yoshida variety ${\rm trop}(X)$.
This was proved in \cite{RSS2} by applying \cite[Theorem 5.5.1]{MS} to the
parametrization of $X$ by monomials in linear forms (Remark \ref{rem:degree9}).
The positive part of ${\rm trop}(X)$ is the normal fan 
of the pezzotope. For details see  the proof of \cite[Theorem~8.1]{EGPSY}.
The relative $4$-cycle given by the pezzotope $U_{\geq 0}$
corresponds to the $3$-cycle $\Delta$ in the Naruki complex.
We saw that
the latter is the positive tropical Yoshida variety.
The top homology of the Naruki complex is generated by the
 $W({\rm E}_6)$-orbit of $\Delta$. Each cycle in this orbit
corresponds to a real relative cycle in $(X,Y)$, namely a translated pezzotope.
Hence, these cycles generate~$H_4(X',Y')$. 
 \end{proof}
 
 \begin{rmk}
The computation of ${\rm trop}(X)$ in \cite{RSS2}
starts from the {\em Bergman fan} of the ${\rm E}_6$ hyperplane arrangement.
This is the cone over a $4$-dimensional simplicial complex with 
$f$-vector $(750,17679,105930,219240,142560)$.
See \cite[Lemma 3.1]{RSS2}. That Bergman complex is
computed from the flats of the
rank $6$ matroid on $36$ elements given by ${\rm E}_6$.
A list of all flats, indexed by root subsystems, is given in \cite[Table 4]{RSS1}.
The main point of \cite[Section~3]{RSS2} is the tropicalization of
the universal family over the moduli space $\mathcal{Y}(3,6)$.
That universal family admits a parametrization by monomials in linear forms.
See the commutative square in \cite[equation (3.1)]{RSS2}.
The relevant combinatorics, summarized in \cite[Lemma 3.1]{RSS2}, 
lays the foundation for  the {\em classification of all tropical cubic surfaces}, which is
 \cite[Theorem 1.1]{RSS2}. 
\end{rmk}

An important takeaway from the proof of Theorem \ref{thm:cr150} is that the Naruki modification $(X',Y')$ of the Yoshida arrangement is such that $X'_\R \backslash Y'_\R$ is a union of curvy pezzotopes. It is instructive to match this boundary structure with Figure \ref{fig:pezzotope}, which shows $5$ colorful cubes and, less visibly, $10$ three-dimensional associahedral facets. The associahedra and the cubes appeared in our proof. Their Schlegel diagrams are seen in \cite[Figures 5 and 6]{SY}. That paper also describes the adjacency of the pezzotopes. The associahedra are the nonnegative parts of ten copies of $\overline{{\cal M}}_{0,6}$ among the $36$ boundary components coming from $Y$. By symmetry, each pezzotope intersects $10$ of these $36$ boundary components. 
The exceptional divisor of $X' \rightarrow X$ over each singular point of $X$ is the toric threefold
$\PP^1 \times \PP^1 \times \PP^1$, whose nonnegative part is a curvy cube.
There are $40$ of these cubes, and each pezzotope intersects five of them.

One virtue of our workflow is that the study of canonical differential forms
is separated from checking the positive geometry property.
The count in (\ref{eq:150altsum}) proves that $\Omega_{\rm log}^4(X \backslash Y)$
has dimension $150$, but  we have not yet written down any 
differential form. This will be done in Section \ref{sec11}.
We close this section with a description of 
$\Omega_{\rm log}^4(X \backslash Y)$ that is
analogous to (\ref{eq:thirdproof}).
The role of the product of
$9$ linear forms, which cuts out the $27$ lines, will be played by (\ref{eq:degree51}).

The ideal of the Yoshida variety $X$ is minimally generated by $30$ linear forms
and $30$ cubics in $40$ variables. Therefore, we may view $X$ as a subvariety of $\mathbb{P}^9$ defined by $30$ cubics. Let $\C[X]_3$ be the space of homogeneous cubics in ten generators of the coordinate ring
of $X \subset \PP^{9}$. This is a  vector space of dimension $\binom{12}{3} - 30 = 190$. Let ${\cal L} \subset \C[X]_3$ denote the linear system of cubics in $\C[X]_3$ that vanish
at the $40$ singular points in~$X$. We present a projective embedding for the
log canonical compactification  given by ${\rm E}_6$ in
\cite[Section~9]{HKT}. Hollering, Pavlov and Pratt \cite{HPP} undertook
 an in-depth computational study of such embeddings.

\begin{thm} \label{thm:nicecubics}
The linear system ${\cal L}$ has dimension $150$, and it
 is isomorphic to $\Omega_{\rm log}^4(X \backslash Y)$.
The image of $X$ in $\PP^{149}$ under these $150$ cubics
is the log canonical model of the moduli space $\mathcal{Y}(3,6)$.
This log canonical embedding  is isomorphic to Naruki's cross-ratio variety $X'$.
\end{thm}

\begin{proof}
One checks that the $40$ singular points impose independent conditions
on the cubics, so the dimension is indeed $190-40 = 150$.
The isomorphism starts from $\Omega_{\rm log}^4(X \backslash Y)$,
using its representation given in Corollary \ref{cor:elfdrei} below.
This is the space spanned by the $S_6$ orbit of all rational functions that appear  in
(\ref{eq:canonicalfunction}) and (\ref{eq:again432}), namely as the factor in front of $\eta$.
We replace the Pl\"ucker coordinates $p_{ijk}$ and $q$ with the
corresponding products of ${\rm E}_6$ roots. This gives a space
of homogeneous rational functions of degree $-24$ in 
the unknowns $(d_1,d_2,d_3,d_4,d_5,d_6)$.
Their common denominator is  the following homogeneous polynomial of degree $51$:
\begin{equation}
\label{eq:degree51} (d_1+d_2+d_3+d_4+d_5+d_6)\,\, \cdot \!\!\! \prod_{1 \leq i < j \leq 6} (d_i - d_j)^2 \, \cdot
 \!\!\!\! \prod_{1 \leq i \leq j \leq k \leq 6}\!\!\! (d_i+d_j+d_k) .
\end{equation}
Multiplying our rational functions with this denominator, we obtain
 a $150$-dimensional linear space of homogeneous polynomials
of degree $27$ in $(d_1,d_2,d_3,d_4,d_5,d_6)$.
One checks that this linear space is equal to the linear system $\mathcal{L}$
of special cubics on the Yoshida variety $X$.

The map from $\PP^{9}$ to $\PP^{149}$ given by the 
linear system ${\cal L}$ blows up the $40$ singular points
of $X$. This implies that the closure of its image is isomorphic to the Naruki variety $X'$. 
\end{proof}

\section{Differential Forms on Moduli Space} \label{sec11}

Our final section executes steps 5 and 6 in the Workflow \ref{workflow}.
In the context of the moduli space $\mathcal{Y}(3,6)$, the input $S$ is the ${\rm E}_6$ pezzotope from Section \ref{sec8}. Concretely, we define $S$ to be the semialgebraic subset of ${\cal Y}(3,6)_\R$ consisting of points with positive  coordinates
$(a,b,c,d)$.

To view this as a positive geometry, we have several natural choices for the positive arrangement in step 1 of Workflow \ref{workflow}. One could let $(X,Y)$ be the Yoshida arrangement from Section \ref{sec10}, or the Naruki arrangement $(X',Y')$. Alternatively, one could embed $S$ into the smooth toric variety $X''$ of the fan \eqref{eq:4by15}, and let $Y''$ be the union of the toric boundary and the hypersurfaces in \eqref{eq:g11}. These compactify ${\cal Y}(3,6)$, in that $X\backslash Y \simeq X' \backslash Y' \simeq X'' \backslash Y'' \simeq {\cal Y}(3,6)$.
The closures of $S$ in the respective compactifications are $\sigma \in Z_4(X,Y)$, $\sigma' \in Z_4(X',Y')$ and $\sigma'' \in Z_4(X'',Y'')$. Their canonical forms live in the same space $\Omega_{\rm log}^4({\cal Y}(3,6))$. In fact, we have $\omega(\sigma) = \omega(\sigma') = \omega(\sigma'') =: \omega(S)$. This follows from the isomorphism $H_4(X,Y) \simeq H_4^{\rm lf}(X \backslash Y)$ in \cite[Equation (8)]{BD}, which holds for all pairs $(X,Y)$ for which $X \backslash Y$ is smooth. Here $H_4^{\rm lf}(X\backslash Y)$ denotes locally finite homology. Hence, each of our compactifications leads to the same canonical form. We use $(X'', Y'')$, which fits best to the $u$-equations in Section~\ref{sec8}. 

The coordinates $(a,b,c,d)$ on $(\mathbb{C}^*)^4$ from
equation (\ref{eq:sixpoints7})
 are local coordinates on the toric variety $X''$. In these coordinates, the pezzotope is the orthant $\mathbb{R}^4_{\geq 0}$. We consider the~form 
\begin{equation} \label{eq:omega_abcd} \omega \, = \, {\rm dlog} \, a \wedge {\rm dlog} \, b \wedge {\rm dlog} \, c \wedge {\rm dlog} \, d \, \, \in \, \Omega^4_{\rm log}({\cal Y}(3,6)). \end{equation}
This is a natural form to start from, as it has logarithmic poles along the four boundary components of $\mathbb{R}^4_{\geq 0}$. More of the boundary structure is seen from a global expression for $\omega$, which uses \emph{Cox coordinates} on $X''$. These are the $u$-coordinates $u_1, \ldots, u_{15}$ from Section \ref{sec8}; there is one Cox coordinate for each facet of the pezzotope. Substituting \eqref{eq:abcdfrac} into
(\ref{eq:omega_abcd}) gives 
	$$  \omega \,=\, {\rm dlog}  \left(\!\frac{u_{10}}{u_5 u_8 u_9 u_{13} u_{14}}\!\right) \,\wedge\, {\rm dlog}  \left(\!\frac{u_9 u_{11}}{u_4 u_7 u_{12} u_{15}}\!\right)
			\, \wedge \,{\rm dlog}  \left(\!\frac{u_4 u_6 u_{14} u_{15}}{u_3 u_{13}}\!\right) \,\wedge \,{\rm dlog}  \left(\!\frac{u_1 u_4 u_8 u_{12} u_{14}}{u_2}\!\right)\!.
			$$
Like in the examples of Section \ref{sec7}, we view this as a top form on the $4$-dimensional variety $U \subset \C^{15}$, which is an affine chart of $X''$ containing the closed curvy pezzotope $\overline{S} = U_{\geq 0}$.

\begin{prop}
The form $\omega$ from \eqref{eq:omega_abcd} is the canonical form $\omega(S)$ of the~pezzotope. 
\end{prop}

\begin{proof}
By \cite[\S 2.3.2]{BD}, we may instead compute the canonical form of $S$ viewed as a region of $(X'',D)$, where $D = D_1 \cup \cdots \cup D_{15} \subset Y''$ is the toric boundary of $X''$. Note that, as $X''$ is smooth, the divisor $D$ is simple normal crossing, and each irreducible component $D_i$ intersects the positive chart $U \subset X''$. We may therefore compute all residues in $u$-coordinates.

By \cite[Proposition 2.15]{BD}, we need to show that the residues along the divisors $D_i$ are the canonical forms of the curvy facets of our pezzotope $U_{\geq 0}$. This leads to a recursion.

The pezzotope has two types of facets. Each facet labeled by
$\{u_1,\ldots,u_{10}\}$ is an associahedron, so it has
the f-vector $(14,21,9)$.  Each facet labeled by
$\{u_{11},\ldots,u_{15}\}$ is a  regular $3$-cube, with
f-vector $(8,12,6)$.
These represent familiar positive geometries of
dimension $3$.

We first consider the residue of $\omega$ along the associahedral boundary
$D_1$. On that boundary we have
$u_1 = 0$ and $u_2 = u_5 = u_7 = u_{13} = u_{15} = 1$, as seen in (\ref{eq:perfectuE6}). Therefore the residue~is
			\begin{equation}
				\label{eq:u1iszero}
				{\rm Res}_{u_1=0}\,\omega \,\,  = \,\,
				{\rm dlog}  \left(\frac{u_{10}}{u_8 u_9 u_{14}}\right)\wedge {\rm dlog}  \left(\frac{u_9 u_{11}}{u_4 u_{12}}\right)\wedge {\rm dlog}  \left(\frac{u_4 u_6 u_{14}}{u_3}\right).
			\end{equation}
			This coincides with the expression for the canonical form of the \emph{worldsheet associahedron} \cite[page 11]{ABHY}, after relabeling.  The $u$-variables
			occurring in (\ref{eq:u1iszero}) satisfy the 
			$u$-equations 
			\begin{eqnarray*}
				&&u_{10}+u_8 u_9 u_{14}=u_{11}+u_4 u_8 u_{12} u_{14}=u_6+u_3 u_8 u_{12}=u_9+u_4 u_{10} u_{12}=u_3 u_{10} u_{11} u_{12}+u_{14}\\
				&&= \, u_8+u_6 u_{10} u_{11\,}=\,u_4+u_3 u_9 u_{11} \,=\,u_{12}+u_6 u_9 u_{11} u_{14}\,=\,u_3+u_4 u_6 u_{14}\,=\,1,
			\end{eqnarray*}
			obtained from (\ref{eq:perfectuE6}) by setting $u_1=0$.
			Their solutions in $(\C^*)^9$ form $\mathcal{M}_{0,6}$, and the closure in $\mathbb{C}^9$ is a partial compactification $U_1 = U \cap \{u_1 = 0 \} \subset D_1$. The solutions in $\mathbb{R}^9_{\geq 0}$ form a curvy three-dimensional associahedron $S_1$ inside $(D_1)_\R$. We must show that \eqref{eq:u1iszero} is its canonical form with respect to the pair $(D_1, D_1 \cap (\cup_{i = 2}^{15} D_i))$. The recursive nature is seen from the fact that this is again a smooth toric variety $D_1$ with its toric boundary.
The nine rays of its fan are read from the monomials in \eqref{eq:u1iszero}. 
Taking residues of \eqref{eq:u1iszero} once again, we find six toric surfaces 
of pentagons and three of quadrilaterals---the nine facets of the associahedron. 
These forms appear under the integral sign in Examples \ref{ex:quadint} 
and \ref{ex:pentint}. Finally, we are down to line segments, 
and checking that we obtain the form from Example \ref{ex:canformlinesegment} concludes the proof.

For completeness, we also consider a cube facet $D_{11}$. The residue of $\omega$ at $\{u_{11}=0\}$ is 
\begin{equation}
				\label{eq:u11iszero}
 {\rm Res}_{u_{11} = 0}\, \omega \,\, = \,\,
			{\rm dlog} \left(\frac{u_{10}}{u_9}\right)\,\wedge \, {\rm dlog}  \left(\frac{u_6}{u_3}\right) \, \wedge \, {\rm dlog}  \left(\frac{u_1}{u_2}\right).
\end{equation}		
			This is the canonical form of $S_{11} = U_{\geq 0} \cap D_{11}$ with respect to $D_{11} \simeq \PP^1 \times \PP^1 \times \PP^1$ and its toric boundary.
			From (\ref{eq:perfectuE6}) we~get $u_1+u_2 \,=\, u_3+u_6 \,=\, u_9+u_{10} \,=\, 1$.
			The $3$-cube is the set of non-negative solutions to these 
			three equations in six variables. Taking iterated residues gives quadrilaterals and line segments, 
			and it proves that \eqref{eq:u11iszero} equals $\omega(S_{11})$.
\end{proof}

In Step~6 of the Workflow \ref{workflow} it is suggested that we
consider other regions of interest. For this, we take advantage of the symmetry explained in Section \ref{sec9}. Each element $g$ of the Weyl group $W({\rm E}_6)$ induces an isomorphism $f_g : {\cal Y}(3,6) \rightarrow {\cal Y}(3,6)$. For instance, we computed in \eqref{eq:g25} that in the coordinates $x,y,z,w$ from \eqref{eq:sixpoints8}, the permutation $g_{2,5}$ induces 
\begin{equation} \label{eq:fg25} f_{g_{2,5}}(x,y,z,w) \, = \, \Big
 (\, \frac{x-xz }{x-z}\,,\, \frac{xzw - xw - yz^2 + yz}{xz-xw-z^2+zw}\,, \,1-z\,, \,\frac{zw-w}{z-w} \,\Big). \end{equation}
The canonical form has the functorial property that $\omega(f_g(S)) = (f_g)_* \,  \omega(S)$. That is, the canonical form of the image of $S$ is the pushforward of the canonical form of $S$ \cite[\S 2.3.4]{BD}. Equivalently, using the action of $W({\rm E}_6)$ on the regions ${\cal R}$ of ${\cal Y}(3,6)$ from Lemma \ref{lem:Weylaction} and the fact that $f_g^{-1} = f_{g^{-1}}$, we have $\omega(g \cdot S) = f_{g^{-1}}^* \, \omega(S)$. For example, the form \eqref{eq:omega_abcd} equals
\begin{equation} \label{eq:omega_xyzw}  \frac{z \, {\rm d} x \wedge {\rm d} y \wedge {\rm d} z \wedge {\rm d} w}{(x-z)(z-w)\left(xyz-xyw-xzw+xw+yzw-yz\right)} \, = \, h(x,y,z,w) \, {\rm d} x \wedge {\rm d} y \wedge {\rm d} z \wedge {\rm d} w.\end{equation}

Let $(\phi_1, \ldots, \phi_4)$ denote the right hand side in \eqref{eq:fg25}.
Then
the pullback along $f_{g_{2,5}}$, is 
\[ f_{g_{2,5}}^* \omega \, = \, h(\phi_1, \phi_2,\phi_3, \phi_4) \, {\rm d} \phi_1 \wedge {\rm d} \phi_2 \wedge {\rm d} \phi_3 \wedge {\rm d} \phi_4 \, = \, h(\phi_1, \phi_2,\phi_3, \phi_4) \, (\det J) \, {\rm d} x \wedge {\rm d} y \wedge {\rm d} z \wedge {\rm d} w, \] 
where $\det J$ is the $4 \times 4$ Jacobian determinant of \eqref{eq:fg25}.
 Since $W({\rm E}_6)$ acts transitively on ${\cal R}$, we obtain the canonical form of each of the 432 regions in this way \cite{Zenodo}. By Lemma~\ref{lem:Weylaction},~each form is found $120$ times. The sign of $f_{g^{-1}}^* \, \omega(S)$ depends on whether $f_g$ preserves orientation.

\begin{cor}  \label{cor:elfzwei}
Using the $W({\rm E}_6)$ action as described above, we obtain
the canonical forms $\omega( g \cdot S)$ for all $432$ pezzotopes. 
These span the $150$-dimensional vector space
$\Omega^4_{\rm log} (X \backslash Y)$.
\end{cor}

\smallskip

Artyom Lisitsyn brought the idea of using symmetry for computing canonical forms to our attention. By using oriented matroid theory, he realized that we may find all $432$ forms from four $S_6$-orbits of size $720$, rather than one $W({\rm E}_6)$-orbit of size $51840$. This simplifies the computation significantly. He kindly agreed to let us explain this computation in our~article.

The point of departure is the
moduli space $\mathcal{X}(3,6)$ of six points in linearly
general position in $\PP^2$. Removing the coconic locus $\{ q = 0 \}$ from ${\cal X}(3,6)$, we recover our moduli space ${\cal Y}(3,6)$.
The real locus $\mathcal{X}(3,6)_\R$ has $372 = 60+180+ 120+12$~regions,
 which come in four orbits under the symmetric group $S_6$.
 These are the four reorientation classes of uniform oriented matroids of rank $3$ on six elements.
 A positive parametrization for each region was found by Antolini and Early \cite[Theorem 6.1]{AE}.
We shall now list four canonical forms found in the proof of their theorem. We work in coordinates \eqref{eq:sixpoints8} and let $\eta = {\rm d} x \wedge {\rm d}y \wedge {\rm d}z \wedge {\rm d} w$:
\begin{equation}
\label{eq:canonicalfunction}
 \begin{matrix} (p_{123} p_{234} p_{345} p_{456} p_{156} p_{126})^{-1} \, \eta & 
(p_{123} p_{126} p_{145} p_{234} p_{356} p_{456})^{-1} \, \eta \\
 p_{245} (p_{124} p_{136} p_{145} p_{234} p_{235} p_{256} p_{456})^{-1} \, \eta  \,\,\,& \quad
 q   (p_{125} p_{126} p_{134} p_{136} p_{145} p_{234} p_{235} p_{246} p_{356} p_{456})^{-1}  \, \eta.\\
\end{matrix}
\end{equation}
The first expression is the {\em Parke-Taylor form} which is prominent in particle physics;
see \cite[Section~7]{EGPSY} and \cite{HP}.  The $p_{ijk}$ are polynomials in $(x,y,z,w)$ given by the $(i,j,k)$ minor of \eqref{eq:sixpoints8}.
Note that the last denominator matches the $10$ triangles in \cite[Figure 3 (a)]{EGPSY}.
By acting with $S_6$, we obtain $372$ rational forms in $(x,y,z,w)$. 
Lisitsyn computed the dimension of the linear span of these $372$ forms by evaluating in many sample points. He found it to be
\begin{equation}
\label{eq:126altsum}
 126 \,\,=\,\,1035\,-\,1395 \,+\, 550 \,-\, 65 \,+\, 1 .
 \end{equation}
These are the face numbers of the tropical Grassmannian  ${\rm trop}({\rm Gr}(3,6))$.
The Euler characteristic $126$ was first reported in \cite[Theorem 5.4]{speyer},
well over two decades ago. It was determined to be the number of independent canonical forms on ${\cal X}(3,6)$ in \cite{CachazoEarlyZhang}. In the terminology of this paper,  {\em
$126$ is the real combinatorial rank of the moduli space $\mathcal{X}(3,6)$.} 

\smallskip

We finally make the situation more symmetric, by 
replacing $S_6= W({\rm A}_5)$ with $W({\rm E}_6)$.
The coconic divisor $\{q=0\}$ cuts the 
positive region of $\mathcal{X}(3,6)$  into two pieces.
By \cite[equation (27)]{EGPSY}, this replaces the
Parke-Taylor form in (\ref{eq:canonicalfunction}) with the new rational form
\begin{equation}
\label{eq:again432}
 p_{135} \cdot ( \,p_{123} p_{345} p_{156} \, q\, )^{-1} \, \eta. 
 \end{equation}
The resulting form is equal to \eqref{eq:omega_abcd} and \eqref{eq:omega_xyzw}. The last three forms in (\ref{eq:canonicalfunction}) remain unchanged.

Up to sign, the $S_6$-orbits of (\ref{eq:again432})  and the last three forms in
(\ref{eq:canonicalfunction}) have combined cardinality
$$ 432 \,=\, 120+180+120+12. $$
This verifies Corollary \ref{cor:432}.
The $S_6$-orbit gives 
canonical forms for all pezzotopes in $\mathcal{Y}(3,6)_\R$. 

We computed 
the dimension of the linear span of all functions in (\ref{eq:canonicalfunction}) and (\ref{eq:again432}) and found it to be $150$ \cite{Zenodo}.
This confirms (\ref{eq:150}), and it elucidates the transition from
(\ref{eq:126altsum}) to (\ref{eq:150altsum}).

\begin{cor} \label{cor:elfdrei}
Using the $S_6$ action on (\ref{eq:canonicalfunction})--(\ref{eq:again432}) as described above,
 we obtain~the canonical forms for all $432$ pezzotopes. 
These span the $150$-dimensional vector space
$\Omega^4_{\rm log} (X \backslash Y)$.
\end{cor}

Corollaries \ref{cor:elfzwei} and \ref{cor:elfdrei}
complete the Workflow \ref{workflow}
for cubic surfaces and their moduli space.
Our article can be a blueprint for the study of positive geometries
in similar settings.

\vspace{1.7cm}

\noindent {\bf Acknowledgements}.
We thank Cl\'ement Dupont, Ulysse Mounoud and Dmitrii Pavlov for helpful communications.
Thomas Endler created several of the diagrams for this article.
This project grew out of the lecture by the first author
at the DMV Topic Days (September 2025) on the 
occasion of the 100th anniversary of Felix Klein's death.
We are grateful to the organizers.
 Our work on this article was supported by
the European Research Council through the synergy grant UNIVERSE+, 101118787. 
\begin{scriptsize}Views~and~opinions expressed
are however those of the authors only and do not necessarily reflect those of the European Union or the 
European
Research Council Executive Agency. Neither the European Union nor the granting authority
can be held responsible for them.
\end{scriptsize}

\bigskip
\bigskip
		
\footnotesize
\noindent {\bf Authors' addresses:}
		
\smallskip
		
\noindent Simon Telen, MPI-MiS Leipzig \hfill \url{simon.telen@mis.mpg.de}
				
\noindent  Bernd Sturmfels, MPI-MiS Leipzig   \hfill \url{bernd@mis.mpg.de}
\end{document}